\documentclass[11pt]{article}
\usepackage[total={7in, 8in}]{geometry}
\usepackage{graphicx}
\usepackage[margin=0.5in]{caption}
\usepackage{comment}
\usepackage{amsmath, amsthm, latexsym, amssymb, color,cite,enumerate, physics, framed}
\usepackage{caption,subcaption,verbatim, empheq, cancel}
\usepackage[inline]{enumitem}
\usepackage{mathtools}
\usepackage{appendix}
\newtheorem{theorem}{Theorem}[section]
\newtheorem{lemma}[theorem]{Lemma}
\newtheorem{definition}[theorem]{Definition}
\newtheorem{corollary}[theorem]{Corollary}
\newtheorem{proposition}[theorem]{Proposition}

\newtheorem{examplex}{Example}
\newenvironment{example}
  {\pushQED{\qed}\examplex}
  {\popQED\endexamplex}
\theoremstyle{remark}
\newtheorem{remark}{Remark}

\newcommand*{\myproofname}{Proof}
\newenvironment{subproof}[1][\myproofname]{\begin{proof}[#1]}{\end{proof}}
\renewcommand\Re{\operatorname{Re}}%%redefined Re and Im
\renewcommand\Im{\operatorname{Im}}
\newcommand\End{\operatorname{End}} %Prefix linear transformations
\newcommand\GldR{\mbox{Gl}_d(\mathbb{R})}%dxd invertible matrices
\newcommand\Gl{\operatorname{Gl}} %Prefix for general linear group
\newcommand\OdR{\mbox{O}(\mathbb{R}^d)} %Orthogonal transformations
\newcommand\Sym{\operatorname{Sym}}
\newcommand\Exp{\operatorname{Exp}}
\newcommand\diag{\operatorname{diag}}
\newcommand\supp{\operatorname{Supp}}
\newcommand\Spec{\operatorname{Spec}}
\renewcommand\det{\operatorname{det}}

\newcommand\Interior{\operatorname{Int}}
\newcommand\R{\mathbb{R}}
\newcommand{\lp}{\left(}
\newcommand{\rp}{\right)}
\newcommand{\lb}{\left[}
\newcommand{\rb}{\right]}

\newcommand{\p}{\partial}
\newcommand{\f}[2]{\frac{#1}{#2}}
\newcommand{\Vol}{\operatorname{Vol}}
\newcommand{\iprod}{\mathbin{\lrcorner}}
\newcommand{\al}{\alpha}
\newcommand{\be}{\beta}

\usepackage{hyperref}
\usepackage{tensor}
\usepackage{xcolor}
\hypersetup{
	colorlinks,
	linkcolor={black!50!black},
	citecolor={blue!50!black},
	urlcolor={blue!80!black}
}
\newcommand{\slantedslash}{\mathbin{\rotatebox[origin=c]{23}{$-$}}}

\newcommand{\docircint}[2]{%
  \ifx#1\displaystyle
    \displayrint
  \else
    \normalrint{#1}%
  \fi
}
\newcommand{\displayrint}{\displaystyle \slantedslash \mkern-18mu}
\newcommand{\normalrint}[1]{%
  \smallerc{#1}\ifx#1\textstyle\mkern-9mu\else\mkern-8.2mu\fi
}
\newcommand{\smallerc}[1]{%
  \vcenter{\hbox{$\ifx#1\textstyle\scriptstyle\else\scriptscriptstyle\fi \slantedslash $}}%
}

\author{Huan Q. Bui\\
\normalsize  Department of Mathematics \& Statistics\\[-0.8ex]
\normalsize Colby College
\and
Evan Randles\thanks{Corresponding author: erandles@colby.edu}\\
\normalsize  Department of Mathematics \& Statistics\\[-0.8ex]
\normalsize Colby College\\\\}

\title{A generalized polar-coordinate integration formula with applications to the study of convolution powers of complex-valued functions on $\mathbb{Z}^d$. }
\date{\today}
\begin{document}
\maketitle

\abstract{In this article, we consider a class of functions on $\mathbb{R}^d$, called positive homogeneous functions, which interact well with certain continuous one-parameter groups of (generally anisotropic) dilations. Generalizing the Euclidean norm, positive homogeneous functions appear naturally in the study of convolution powers of complex-valued functions on $\mathbb{Z}^d$. As the spherical measure is a Radon measure on the unit sphere which is invariant under the symmetry group of the Euclidean norm, to each positive homogeneous function $P$, we construct a Radon measure $\sigma_P$ on $S=\{\eta \in \mathbb{R}^d:P(\eta)=1\}$ which is invariant under the symmetry group of $P$. With this measure, we prove a generalization of the classical polar-coordinate integration formula and deduce a number of corollaries in this setting. We then turn to the study of convolution powers of complex functions on $\mathbb{Z}^d$ and certain oscillatory integrals which arise naturally in that context. Armed with our integration formula and the Van der Corput lemma, we establish sup-norm-type estimates for convolution powers; this result is new and partially extends results of \cite{randles_convolution_2015} and \cite{randles_convolution_2017}.}\\



\section{Introduction}\label{sec:Introduction}

The spherical measure and the related polar-coordinate integration formula are important tools and objects of study in several areas of mathematical analysis \cite{stein_harmonic_1993, baker_integration_1997,helms_potential_2009}. To have them at our fingertips, let us first denote by $\mathbb{S}$ and $\mathbb{B}$ the unit sphere and open unit ball in $\mathbb{R}^d$, respectively, let $m$ be the Lebesgue measure on $\mathbb{R}^d$, and write $dx=m(dx)=dm(x)$. The spherical measure is the canonical Radon measure on $\mathbb{S}$ for which $\Theta(\mathbb{S})=d\cdot m(\mathbb{B})$ and $\Theta(OF)=\Theta(F)$ for every orthogonal transformation $O$ and Borel set $F\subseteq\mathbb{S}$. With this measure, we state the classical polar coordinate integration formula as follows: For every $f\in L^1(\mathbb{R}^d)$ (or non-negative measurable $f$),
\begin{equation}\label{eq:StandardPolarIntegrationFormula}
\int_{\mathbb{R}^d}f(x)\,dx=\int_{\mathbb{S}}\left(\int_0^\infty f(r\eta)r^{d-1}\,dr\right)\,\Theta(d\eta)=\int_0^\infty\left(\int_\mathbb{S}f(r\eta)\Theta(d\eta)\right)r^{d-1}\,dr.
\end{equation}
Precise formulations of this classical result can be found in \cite{stein_real_2009} and \cite{folland_real_2013}. For two interesting applications which provide some useful context, we encourage the reader to see \cite{baker_integration_1997} and \cite{folland_how_2001}.  In this article, we generalize the polar coordinate integration formula \eqref{eq:StandardPolarIntegrationFormula} and the spherical measure in a way that is well-adapted to the analysis of certain oscillatory integrals which appear in the study of convolution powers of complex-valued functions on $\mathbb{Z}^d$. Our generalized formula will prove to be a useful tool in the analysis of such oscillatory integrals, leading to the advertised sup-norm-type estimate (Theorem \ref{thm:ConvolutionPowerEstimate}) for convolution powers of complex-valued functions and, in a forthcoming article, a theory of local limit theorems.\\

\noindent To describe the generalized polar coordinate integration formula treated in this article, we must first introduce a class of functions on $\mathbb{R}^d$ which share several desirable properties with the Euclidean norm. The first such property is that such functions ``play well" with spatial dilations of the following form: Let $\{T_r\}_{r>0}\subseteq \Gl(\mathbb{R}^d)$ be a continuous one-parameter group, i.e., $T$ is a continuous group homomorphism from the multiplicative group of positive real numbers into the general linear group, $\Gl(\mathbb{R}^d)$. It is well-known (c.f., \cite{randles_convolution_2017,engel_one-parameter_2000,engel_short_2006}) that every continuous one-parameter group $\{T_r\}$ has the unique representation
\begin{equation*}
T_r=r^E=\exp((\ln r) E)=\sum_{k=0}^\infty \frac{(\ln r)^k}{k!}E^k
\end{equation*}
for some $E\in\End(\mathbb{R}^d)$; here,  $\End(\mathbb{R}^d)$ is the algebra of endomorphisms of $\mathbb{R}^d$ which we take equipped with the operator norm $\|\cdot\|$. $E$ is called the (infinitesimal) generator of $\{T_r\}$ and $\{T_r\}$ is said to be generated by $E$. This, of course, gives a one-to-one correspondence between $\End(\mathbb{R}^d)$ and the collection of continuous one-parameter groups.  A continuous one-parameter group $\{T_r\}$ is said to be \textbf{contracting} if
\begin{equation*}
\lim_{r\to 0}\|T_r\|=0. 
\end{equation*}
We note that, if $E$ is the generator of a contracting group $\{T_r\}$, then $\tr E=\tr(E)>0$. This fact (Proposition \ref{prop:ContractingTrace}) and its proof can be found in the appendix along with some other basic results on one-parameter groups. Also, we encourage the reader to look at the excellent texts \cite{engel_one-parameter_2000} and \cite{engel_short_2006} on one-parameter (semi) groups.\\

\noindent Given a function $P:\mathbb{R}^d\to\mathbb{R}$, we shall call
\begin{equation*}
    S=\{\eta\in\mathbb{R}^d:P(\eta)=1\}
\end{equation*}
\textbf{the unital level set of $P$}. We say that $P$ is \textbf{positive definite} if $P$ is non-negative and $P(x)=0$ only when $x=0$. Given a continuous one-parameter group $\{T_r\}$, we say that $P$ is \textbf{homogeneous with respect to $\{T_r\}$} if
\begin{equation}\label{eq:IntroductionHomogeneous}
    rP(x)=P(T_r x)=P(r^Ex)
\end{equation}
for all $r>0$ and $x\in\mathbb{R}^d$. By an abuse of language, we will also say that $P$ is homogeneous with respect to $E$ whenever \eqref{eq:IntroductionHomogeneous} is satisfied. The set of all such $E\in \End(\mathbb{R}^d)$ for which \eqref{eq:IntroductionHomogeneous} holds is denoted by $\Exp(P)$ and called the \textbf{exponent set of $P$}. We have the following characterization whose proof can be found in Section \ref{sec:Homogeneous}.

\begin{proposition}\label{prop:PositiveHomogeneousCharacterization}
Let $P:\mathbb{R}^d\to\mathbb{R}$ be continuous, positive definite, and have $\Exp(P)\neq \varnothing$. The following are equivalent:
\begin{enumerate}[label=(\alph*), ref=(\alph*)]
\item\label{cond:SisCompact} $S$ is compact.
\item\label{cond:PisAboveOne} There is a positive number $M$ for which
\begin{equation*}
P(x)>1
\end{equation*}
for all $|x|\geq M$. 
\item\label{cond:Contracting} For each $E\in\Exp(P)$, $T_r=r^E$ is contracting.
\item\label{cond:ThereExistsContracting} There exists $E\in\Exp(P)$ for which $T_r=r^E$ is contracting.
\item\label{cond:InfiniteLimit} We have
\begin{equation*}
\lim_{x\to\infty}P(x)=\infty.
\end{equation*}
\end{enumerate}
\end{proposition}

\begin{definition}
Let $P:\mathbb{R}^d\to\mathbb{R}$ be continuous, positive definite and have $\Exp(P)\neq \varnothing$. If any one (and hence all) of the equivalent conditions in Proposition \ref{prop:PositiveHomogeneousCharacterization} are fulfilled, we say that $P$ is positive homogeneous.
\end{definition}

\noindent Before we introduce several examples of positive homogeneous functions, it is helpful to fix some notation and introduce some basic topological objects connected to positive homogeneous functions. We shall denote by $\mathbb{Z}$, $\mathbb{N}$, and $\mathbb{N}_+$ the set of integers, (non-negative) natural numbers, and positive natural numbers, respectively; the $d$-tuples formed by elements of these sets will be denoted by $\mathbb{Z}^d$, $\mathbb{N}^d$, and $\mathbb{N}_+^d$. Throughout this article, our setting is the $d$-dimensional Euclidean space $\mathbb{R}^d$ with coordinates $(x^1,x^2,\dots,x^d)$ and equipped with the inner product $x\cdot y=\sum_{k=1}^d(x^k)(y^k)$ and associated Euclidean norm $|x|=\sqrt{(x^1)^2+(x^2)^2+\cdots+(x^d)^2}$. We take $\mathbb{R}^d$ to be equipped with its usual topology and oriented smooth structure. Given $x\in\mathbb{R}^d$ and $r>0$, the open (Euclidean) ball with center $x$ and radius $r$ is denoted by $\mathbb{B}_r(x)$ and its corresponding sphere is denoted by $\mathbb{S}_r(x)$. In the case that $x=0$, we write $\mathbb{B}_r=\mathbb{B}_r(0)$ and $\mathbb{S}_r=\mathbb{S}_r(0)$ for $r>0$. For a subset $A$ of a topological space, we denote by $\Interior{(A)}$, $\overline{A}$, and $\partial A$ its interior, closure, and boundary, respectively. Given a positive homogeneous function $P$, we define
\begin{equation*}
B_r=\{\xi\in\mathbb{R}^d:P(\xi)<r\}\hspace{1cm}\mbox{and}\hspace{1cm}A_s^r=\{\xi\in\mathbb{R}^d:s\leq P(\xi)<r\}
\end{equation*}
for $r>0$ and $0\leq s<r$; these are $P$-adapted analogues of the Euclidean ball, $\mathbb{B}_r$, and the annulus of inner radius $s$ and outer radius $r$, respectively. In view of of Proposition \ref{prop:PositiveHomogeneousCharacterization} and the continuity of $P$, we see that, for each $r>0$, $B_r$ is open and $\overline{B_r}$ is compact. Further, by setting $B=B_1$, it is a straightforward exercise to see that $\overline{B}=B\cup S$ where $\partial B=S$ is the unital level set associated to $P$.

\begin{example}\label{exp:EuclideanNorm}\normalfont
For any $\alpha>0$, the $\alpha$th-power of the Euclidean norm $x\mapsto |x|^\alpha$ is positive homogeneous.  In this case, the unital level set $S$ is the standard unit sphere $\mathbb{S}=\mathbb{S}_1$, $B=\mathbb{B}=\mathbb{B}_1$, and
\begin{equation*}
    \Exp(|\cdot|^\alpha)=\frac{1}{\alpha}I+\mathfrak{o}(d)
\end{equation*}
where $I$ is the identity and $\mathfrak{o}(d)$ is the Lie algebra of the orthogonal group $\OdR$ and is characterized by the set of skew-symmetric matrices. 
\end{example}

\begin{example}\label{exp:Polynomial}\normalfont
In the language of L. H\"{o}rmander \cite{hormander_analysis_1983}, consider semi-elliptic polynomial of the form
\begin{equation}\label{eq:SemiEllipticIntro}
    P(x)=\sum_{|\alpha:\mathbf{n}|=1}a_\alpha x^\alpha,
\end{equation}
where $\mathbf{n}=(n_1,n_2,\dots,n_d)$ is a $d$-tuple of positive even natural numbers\footnote{In Subsection \ref{subsec:Examples}, we will write this as $\mathbf{n}=2\mathbf{m}$ for $\mathbf{m}\in\mathbb{N}_+^d$.}, and, for each multi-index $\alpha =(\alpha_1,\alpha_2,\dots,\alpha_d)\in\mathbb{N}^d$,
\begin{equation*}
    |\alpha:\mathbf{n}|:=\sum_{k=1}^d\frac{\alpha_k}{n_k},
\end{equation*}
and
\begin{equation*}
    x^\alpha=\left(x^1\right)^{\alpha_1}\left(x^2\right)^{\alpha_2}\cdots\left(x^d\right)^{\alpha_d}
\end{equation*}
for $x=\left(x^1,x^2,\dots,x^d\right)\in\mathbb{R}^d$. If we consider $E\in\End(\mathbb{R}^d)$ whose standard matrix representation is $\diag(1/n_1,1/n_2,\dots,1/n_d)$, we have
\begin{equation*}
    P\left(r^Ex\right)=\sum_{|\alpha:\mathbf{n}|=1}a_{\alpha}\left(r^{1/n_1}x^1\right)^{\alpha_1}\left(r^{1/n_2}x^2\right)^{\alpha_2}\cdots\left(r^{1/n_d}x^d\right)^{\alpha_d}=\sum_{|\alpha:\mathbf{n}|=1}a_\alpha r^{|\alpha:\mathbf{n}|}x^\alpha=rP(x)
\end{equation*}
for all $x\in\mathbb{R}^d$ and $r>0$ and therefore $E\in\Exp(P)$. It is easy to see that $T_r=r^E$ is a contracting group and so we have the following statement by virtue of Proposition \ref{prop:PositiveHomogeneousCharacterization}: \begin{center}\textit{If a semi-elliptic polynomial $P(x)$ of the form \eqref{eq:SemiEllipticIntro} is positive definite, then it is positive homogeneous.}
\end{center}

\noindent For two concrete examples, consider the polynomials $P_1$ and $P_2$ on $\mathbb{R}^2$ defined by
\begin{equation*}
    P_1(x,y)=x^2+y^4\hspace{1cm}\mbox{and}\hspace{1cm}P_2(x,y)=x^2+\frac{3}{2}xy^2+y^4
\end{equation*}
defined for $(x,y)\in\mathbb{R}^2$. It is straightforward to see that $P_1$ and $P_2$ are both positive definite and semi-elliptic of the form \eqref{eq:SemiEllipticIntro} with $\mathbf{n}=(2,4)$. Figure \ref{fig:PoneAndtwo} illustrates $P_1$ and $P_2$ along with their associated unital level sets $S_1$ and $S_2$ and corresponding sets $B_1=\{(x,y)\in\mathbb{R}^2:P_1(x,y)<1\}$ and $B_2=\{(x,y)\in\mathbb{R}^2:P_2(x,y)<1\}$ written with a slight abuse of notation.
\begin{figure}[!htb]
    \centering
    \hspace{10pt}
    \begin{subfigure}{0.5\textwidth}
    \centering
    \includegraphics[scale=0.6]{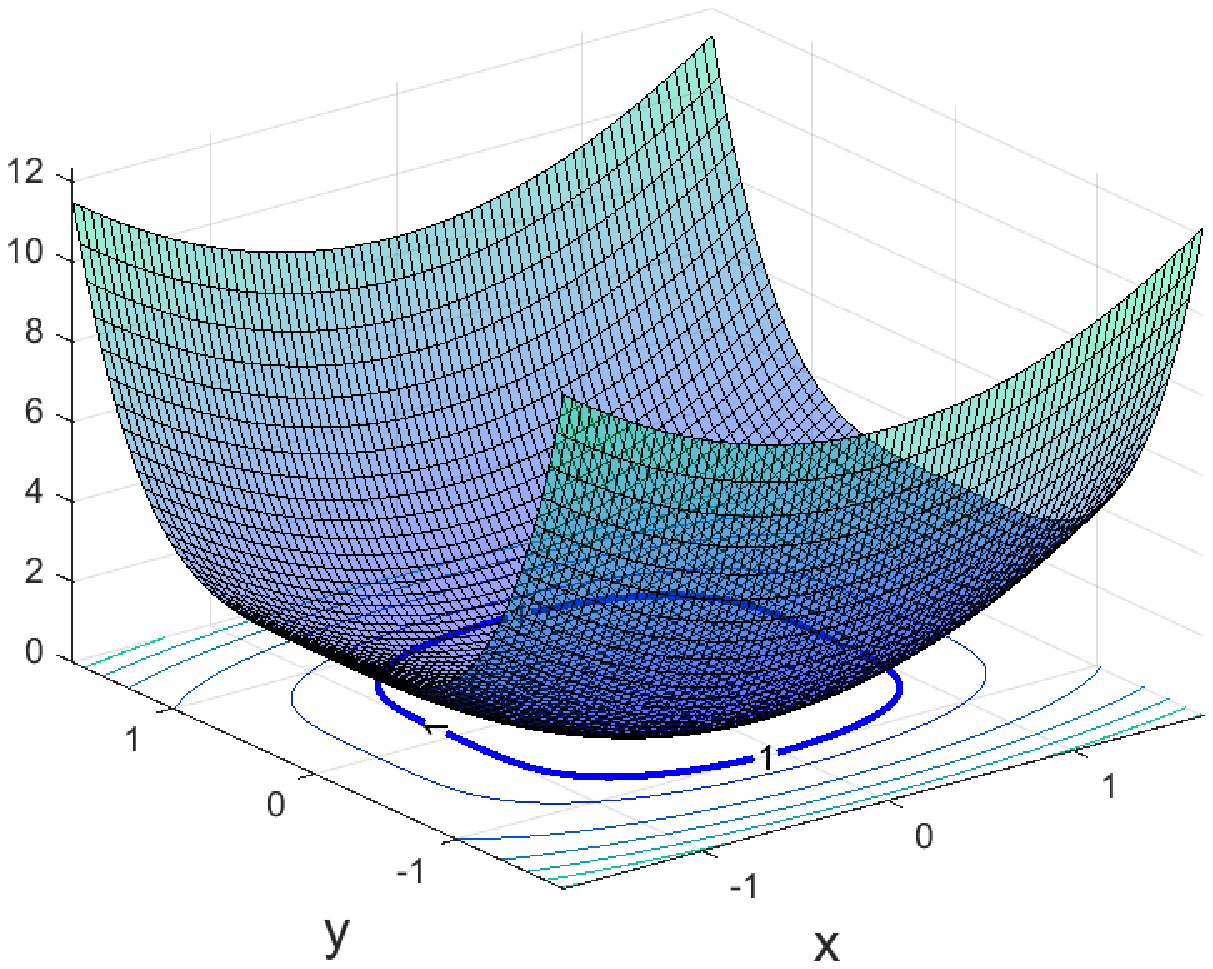}
    \vspace{-10pt}
    \includegraphics[scale=0.6]{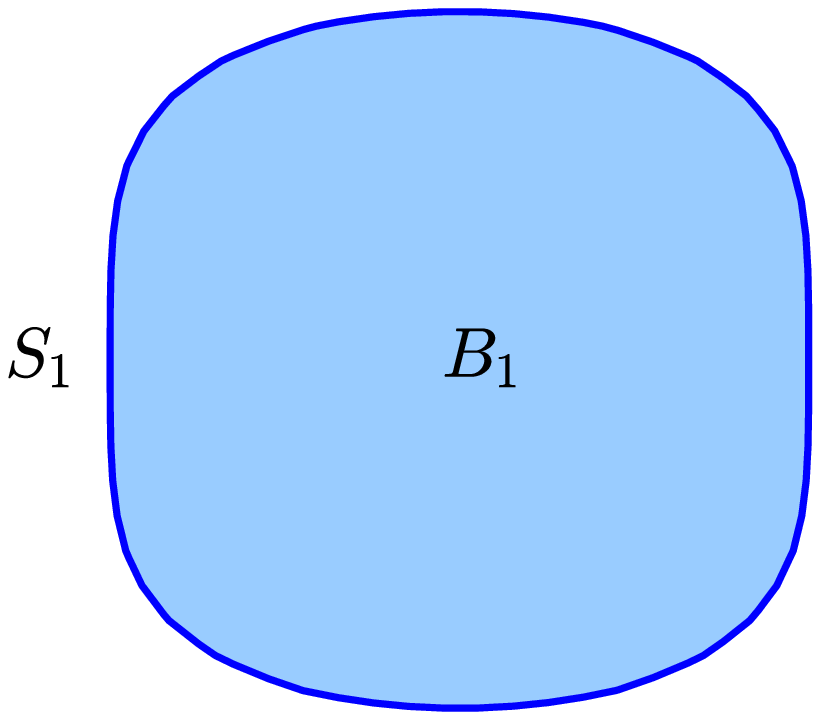}
    %\caption{}
    %\label{fig:convex_SP}
    \end{subfigure}%
    \hspace{-20pt}
    \begin{subfigure}{0.5\textwidth}
    \centering
    \includegraphics[scale=0.6]{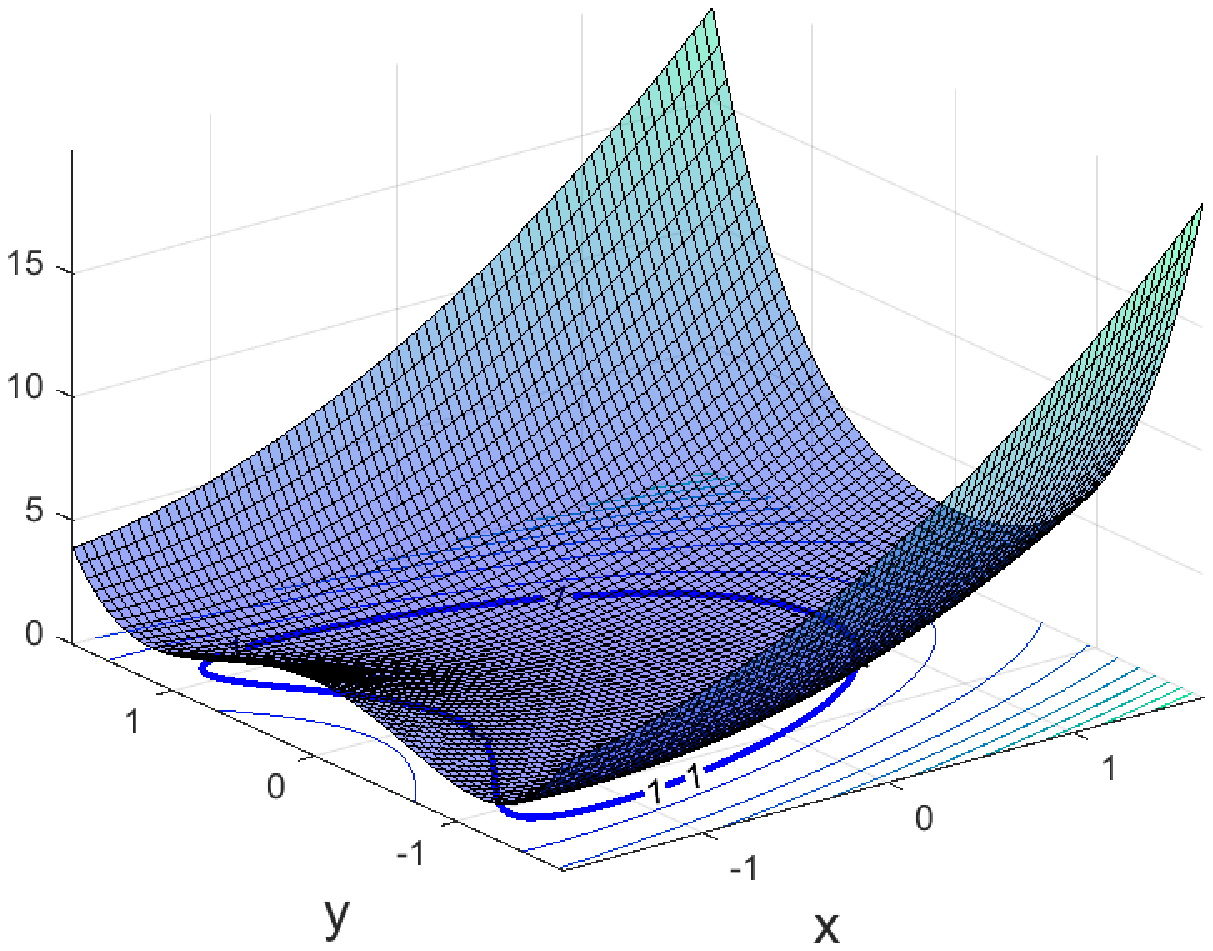}
    \vspace{-10pt}
    \includegraphics[scale=0.6]{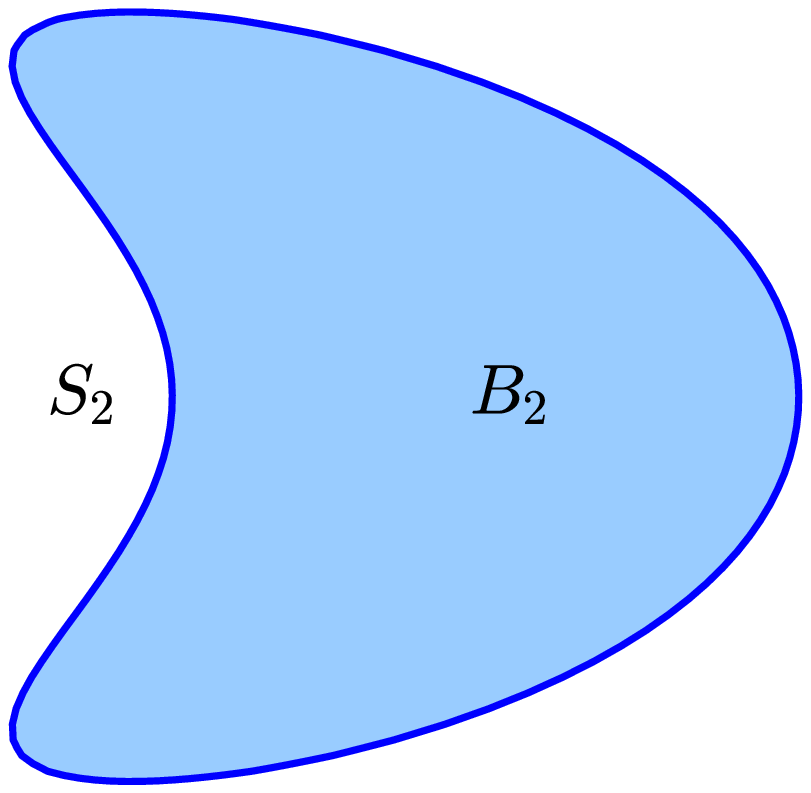}
    %\caption{}
    %\label{fig:non_convex_SQ}
    \end{subfigure}
    \caption{The left column illustrates the graph of $P_1$ with its associated $S_1$ and convex $B_1$. The right column illustrates the graph of $P_2$ with its associated $S_2$ and non-convex $B_2$.}
    \label{fig:PoneAndtwo}
\end{figure}
\begin{remark}\label{rmk:PositiveHomogeneousPolynomialsVSFunctions}
In \cite{randles_convolution_2017}, a positive homogeneous polynomial $P$ is, by definition, a complex-valued multivariate polynomial on $\mathbb{R}^d$ for which $\Exp(P)$ contains an element of $\End(\mathbb{R}^d)$ whose spectrum is purely real and for which $R=\Re P$ is positive definite (see Proposition \ref{prop:PosHomSufficientCondition} below). By virtue of Proposition 2.2 of \cite{randles_convolution_2017}, for each such polynomial $P$ and $E\in\Exp(P)$ with real spectrum, there exists $A\in\Gl(\mathbb{R}^d)$ (representing a change of basis of $\mathbb{R}^d$) and a $d$-tuple of even positive natural numbers $\mathbf{n}=(n_1,n_2,\dots,n_d)\in\mathbb{N}_+^d$ for which $A^{-1}EA$ has standard matrix representation $\diag(1/n_1,1/n_2,\dots,1/n_d)$ and $(P\circ A)(x)$ is semi-elliptic of the form \eqref{eq:SemiEllipticIntro} with, in this case, complex coefficients. It follows that every real-valued positive homogeneous polynomial (in the sense of \cite{randles_convolution_2017}) is a positive homogeneous function in the sense of the present article. Of course, the semi-elliptic polynomials discussed above are positive homogeneous polynomials in the sense of \cite{randles_convolution_2017} where $A=I$. We refer the reader to Section 7.3 of \cite{randles_convolution_2017} which presents a real-valued positive homogeneous polynomial which is not semi-elliptic (and so $A\neq I$).
\end{remark}
\end{example}

\begin{example}\label{exp:Weierstrass}\normalfont
Let $Q$ be a positive homogeneous function with exponent set $\Exp(Q)$ and compact unital level set $S_Q=\{\eta:Q(\eta)=1\}$. Given any $f\in C^0(S_Q)$ for which $f(\eta)>0$ for all $\eta\in S_Q$ and $E\in \Exp(Q)$, define $P=P_{f,E,Q}:\mathbb{R}^d\to\mathbb{R}$ by
\begin{equation*}
P(x)=\begin{cases}
Q(x)f\left((Q(x))^{-E}x\right) & x\neq 0\\
0 & x=0
\end{cases}
\end{equation*}
for $x\in\mathbb{R}^d$. We claim that $P$ is positive homogeneous and $E\in\Exp(P)$.

\begin{subproof}To see this, we first observe that, for any $x\in\mathbb{R}^d\setminus \{0\}$, $Q((Q(x))^{-E}x)=Q(x)/Q(x)=1$ and hence $Q(x)^{-E}x\in S_Q$ and so the above formula makes sense and ensures that $P$ is continuous on $\mathbb{R}^d\setminus\{0\}$. Furthermore, because $f$ is continuous and positive on the compact set $S_Q$, we have $0<\min f\leq \max f<\infty$. From this it follows that $P$ is positive definite and, by virtue of the squeeze theorem, continuous at $x=0$. For any $r>0$ and $x\in\mathbb{R}^d$, we have
\begin{equation*}
P(r^Ex)=Q(r^Ex)f(Q(r^Ex)^{-E}r^Ex)=rQ(x)f(Q(x)^{-E}x)=rP(x)
\end{equation*}
and therefore $E\in\Exp(P)$. Upon noting that $\{r^E\}$ is contracting by virtue of Proposition \ref{prop:PositiveHomogeneousCharacterization}, we conclude that $P$ is positive homogeneous.
\end{subproof}
\noindent The utility of this construction allows us to see that ``most'' positive homogeneous functions are not smooth. To see this, we fix a positive homogeneous function $Q\in C^{\infty}(\mathbb{R}^d)$ and remark that $S_Q$ is necessarily a compact smooth embedded hypersurface of $\mathbb{R}^d$ (see Proposition \ref{prop:InnerProdIsOne}). If $P=P_{f,Q}$ is $C^\infty(\mathbb{R}^d)$, $P\vert_{S_Q}=f$ is necessarily $C^\infty(S_Q)$. It follows that $P\notin C^\infty(\mathbb{R}^d)$ whenever $f$ is chosen from $C^0(S_Q)\setminus C^\infty(S_Q)$. By precisely the same argument, we see that $P\in C^0(\mathbb{R}^d)\setminus C^k(\mathbb{R}^d)$ whenever $f\in C^0(S_Q)\setminus C^k(S_Q)$ for each $k\in\mathbb{N}_+$.\\

\noindent  As a straightforward example, consider $Q(x,y)=|(x,y)|=\sqrt{x^2+y^2}$ on $\mathbb{R}^2$ with $S_Q=\mathbb{S}$ and define
\begin{equation*}
f(x,y)=w(\mbox{Arg}(x,y))+3
\end{equation*}
where $w:\mathbb{R}\to\mathbb{R}$ is defined by
\begin{equation*}
    w(t) = \sum_{n=0}^\infty 2^{-n} \cos\lp 3^n t \rp
\end{equation*}
for $t\in\mathbb{R}$; $w$ is a continuous $2\pi$-periodic version of the Weierstrass function. The resulting positive homogeneous function $P$ is continuous but nowhere differentiable. Figure \ref{fig:Weierstrass} illustrates this function $P$ alongside $Q$ and together with their associated unital level sets. We note that $S_P\neq S_Q$ and this is generally the case unless $f\equiv 1$.

\begin{figure}[!htb]
    \centering
    \hspace{10pt}
    \begin{subfigure}{0.5\textwidth}
    \centering
    \includegraphics[scale=0.6]{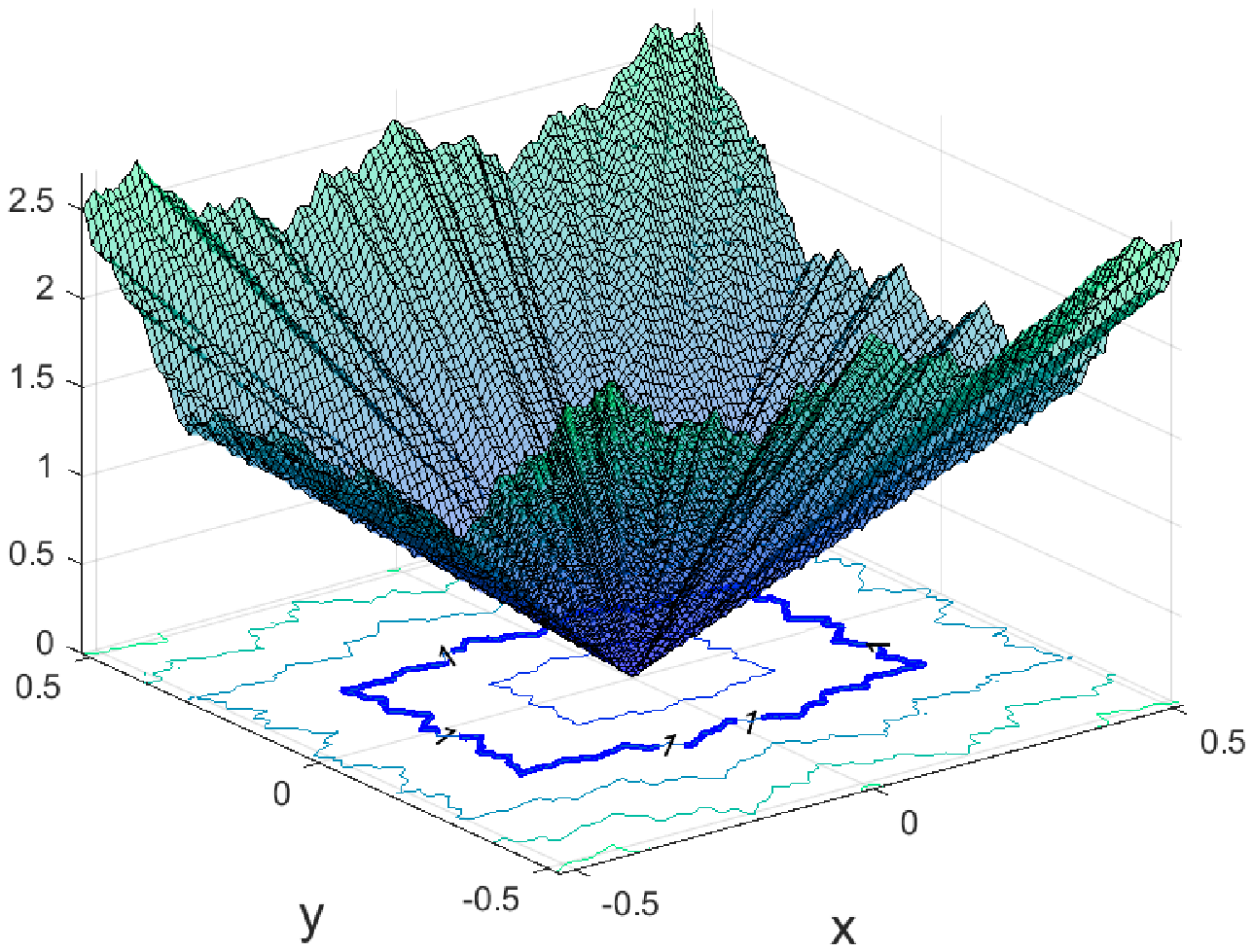}
    \vspace{-10pt}
    \includegraphics[scale=0.6]{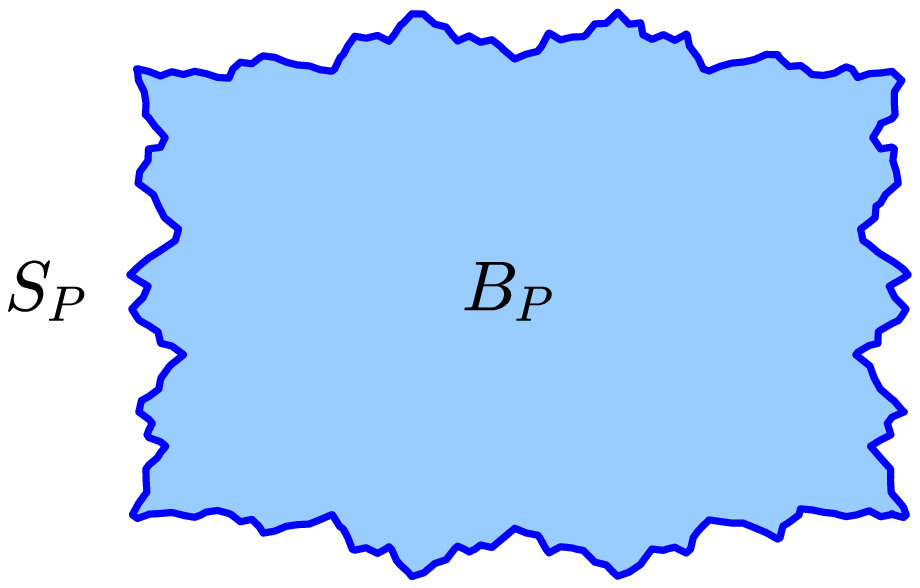}
    %\caption{}
    %\label{fig:WeierstrassP_levelsets}
    \end{subfigure}%
    \hspace{-20pt}
    \begin{subfigure}{0.5\textwidth}
    \centering
    \includegraphics[scale=0.6]{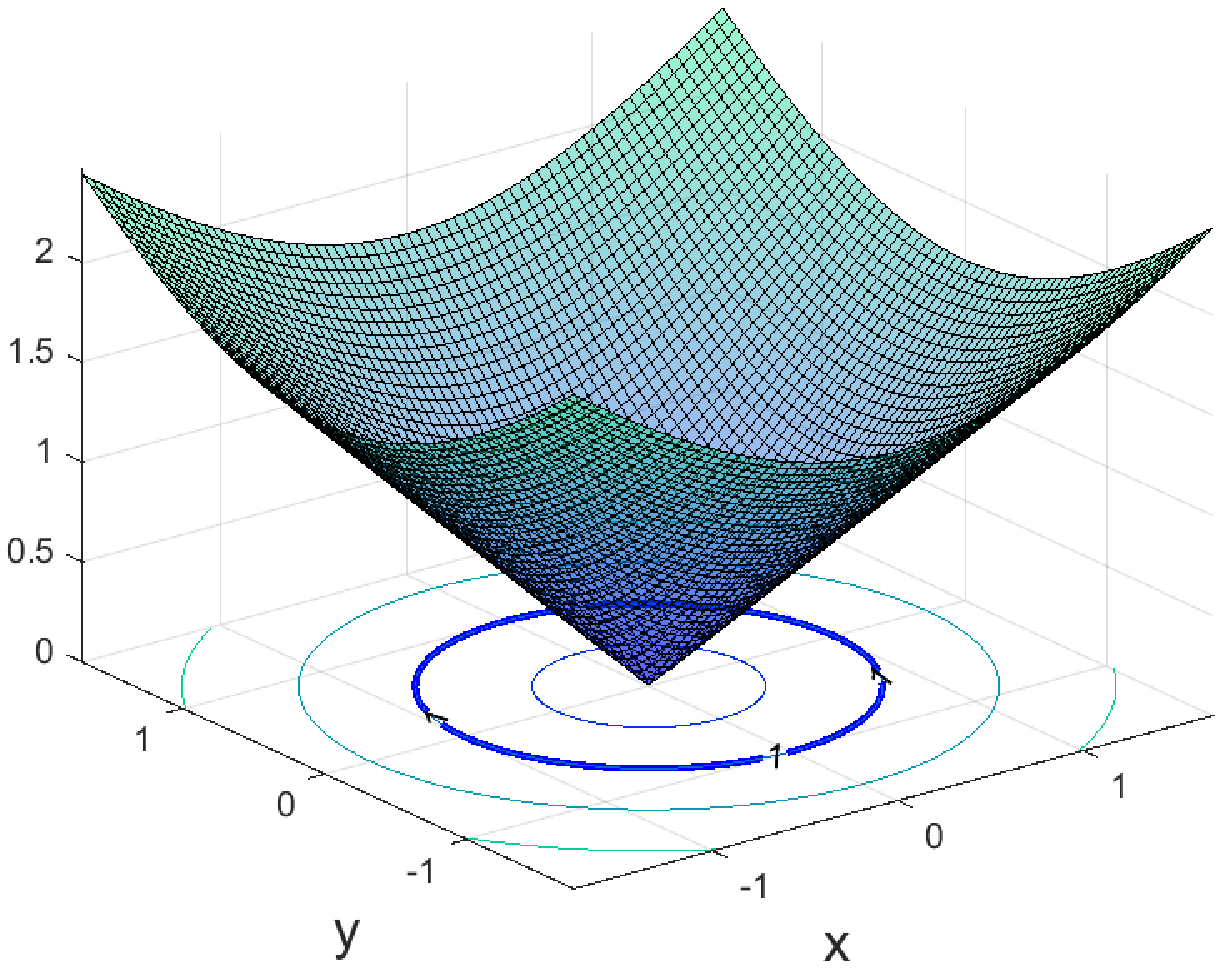}
    \vspace{-10pt}
    \includegraphics[scale=0.6]{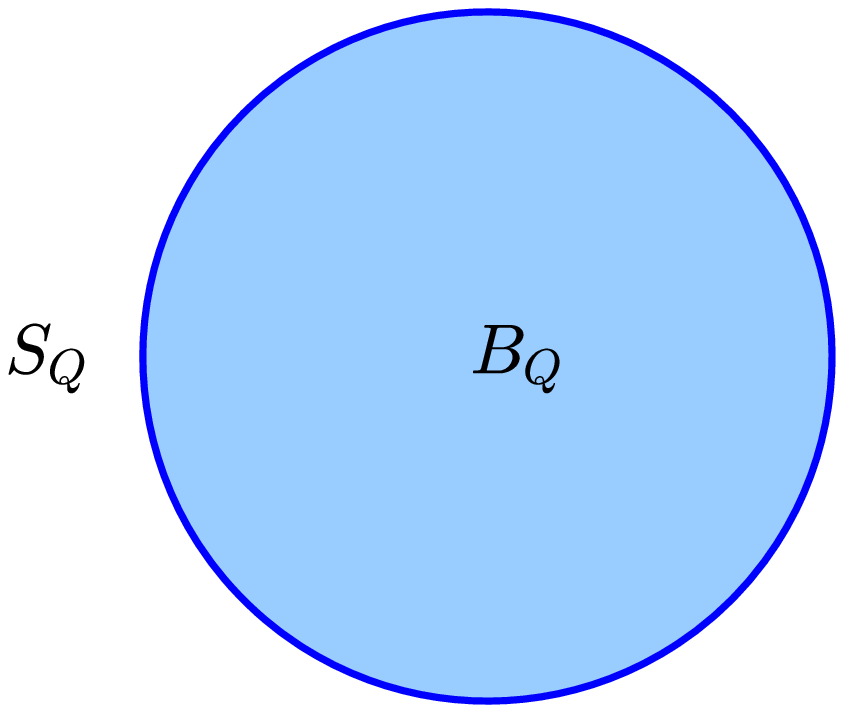}
    %\caption{}
    %\label{fig:Q_levelsets}
    \end{subfigure}
    \caption{The left column illustrates $P$'s graph and associated level set $S_P$ containing $B_P$. The right column illustrates $Q$'s graph and associated level set $S_Q$ containing $B_Q$.}
    \label{fig:Weierstrass}
\end{figure}
\end{example}

\noindent Given a positive homogeneous function $P$, let $\Sym(P)$ be the set of $O\in\End(\mathbb{R}^d)$ for which
\begin{equation*}
P(Ox)=P(x)
\end{equation*}
for all $x\in\mathbb{R}^d$ and observe that $OS\subseteq S$ whenever $O\in\Sym(P)$. By virtue of the positive definiteness of $P$, it is straightforward to verify that $\Sym(P)$ is a subgroup of $\Gl(\mathbb{R}^d)$. For this reason, $\Sym(P)$ is said to be the \textbf{symmetry group associated to $P$}. In fact, we will show that $\Sym(P)$ is a compact subgroup of $\Gl(\mathbb{R}^d)$ and hence a subgroup of the orthogonal group; this is Proposition \ref{prop:SymCompact}. As a consequence of this, we will prove that $\tr E=\tr E'$ for all $E,E'\in\Exp(P)$ (Corollary \ref{cor:TraceisInvariant}) and this allows us to define the \textbf{homogeneous order of $P$} to be the unique positive number $\mu_P$ for which
\begin{equation*}
\mu_P=\tr E
\end{equation*}
for all $E\in\Exp(P)$. As we will see, the ``radial" measure $r^{d-1}\,dr$ in \eqref{eq:StandardPolarIntegrationFormula} will be replaced by $r^{\mu_P-1}\,dr$ in our generalized polar coordinate integration formula; they coincide when $P$ is the Euclidean norm.

\begin{example}\normalfont
In Example \ref{exp:EuclideanNorm}, $\Sym(|\cdot|^\alpha)$ is precisely the orthogonal group $\OdR$ and $\mu_{|\cdot|^\alpha}=d/\alpha$. In Example \ref{exp:Polynomial}, the symmetric set of a semi-elliptic polynomial $P$ of the form \eqref{eq:SemiEllipticIntro} depends on the specific nature of the polynomial in question. Concerning the polynomials $P_1$ and $P_2$ in that example, it is easily shown that $\Sym(P_1)$ is the four-element dihedral group $D_2$ and $\Sym(P_2)$ is the two-element group consisting of the identity and the transformation $(x,y)\mapsto (x,-y)$. For a semi-elliptic polynomial $P$ of the form \eqref{eq:SemiEllipticIntro}, 
\begin{equation*}
    \mu_P=|\mathbf{1}:\mathbf{n}|=\frac{1}{n_1}+\frac{1}{n_2}+\cdots+\frac{1}{n_d}
\end{equation*}
and, in particular, $\mu_{P_1}=\mu_{P_2}=1/2+1/4=3/4.$
\end{example}

\noindent Armed with the notion of positive homogeneous functions and their associated contracting groups, we are ready to introduce our generalization of spherical measure and the polar-coordinate integration formula. To this end, denote by $\mathcal{M}_d$ the Lebesgue $\sigma$-algebra on $\mathbb{R}^d\setminus\{0\}$ and by $\mathcal{L}$ the Lebesgue $\sigma$-algebra on $(0,\infty)$. Given a positive homogeneous function $P$ with homogeneous order $\mu_P>0$, we let $\lambda_P$ denote the $\sigma$-finite measure on $((0,\infty),\mathcal{L})$ with $\lambda_P(dr)=r^{\mu_P-1}\,dr$. Our main theorem is as follows\footnote{We refer the reader to Section 3.6 and 9.2 of \cite{bogachev_measure_2007} which provides some basic context and vocabulary.}.

\begin{theorem}\label{thm:BestIntegrationFormula} Let $P$ be a positive homogeneous function on $\mathbb{R}^d$ and let $S$, $\Exp(P)$, $\Sym(P)$, and $\mu_P$ be $P$'s associated unital level set, exponent set, symmetric group, and homogeneous order, respectively.
There exists a $\sigma$-algebra $\Sigma_P$ on $S$ containing the Borel $\sigma$-algebra on $S$, $\mathcal{B}(S)$,
and a finite Radon measure $\sigma_P$ on $(S,\Sigma_P)$ which satisfies the following properties:
\begin{enumerate}
\item\label{property:Completion} $(S,\Sigma_P,\sigma_P)$ is the completion of $(S,\mathcal{B}(S),\sigma_P)$. In particular, $(S,\Sigma_P,\sigma_P)$ is a complete measure space.
\item\label{property:Invariance} For any $F\in\Sigma_P$ and $O\in\Sym(P)$, $OF\in\Sigma_P$ and $\sigma_P(OF)=\sigma_P(F)$.
\item\label{property:DefiningConditionofsigma} For $F\subseteq S$, $F\in\Sigma_P$ if and only if $\widetilde{F_E}:=\{r^E\eta\in\mathbb{R}^d\setminus\{0\}:0<r<1,\eta\in F\}\in\mathcal{M}_d$ for every $E\in \Exp(P)$. In this case
\begin{equation*}
    \sigma_P(F)=\mu_P\cdot m(\widetilde{F_E})
\end{equation*}
for all $E\in\Exp(P)$.
\end{enumerate}
Further, denote by $\left((0,\infty)\times S,(\mathcal{L}\times\Sigma_P)',\lambda_P\times\sigma_P\right)$ the completion of the product measure space \break $((0,\infty)\times S,\mathcal{L}\times\Sigma_P,\lambda_P\times\sigma_P)$. We have
\begin{enumerate}
\item\label{property:BestPointIsomorphism} Given any $E\in \Exp(P)$, the map $\psi_E:(0,\infty)\times S\to\mathbb{R}^d\setminus\{0\}$, defined by $\psi_E(r,\eta)=r^E\eta$ for $r>0$ and $\eta\in S$, is a point isomorphism of the measure spaces $\left((0,\infty)\times S,(\mathcal{L}\times\Sigma_P)',\lambda_P\times\sigma_P\right)$ and $(\mathbb{R}^d\setminus\{0\},\mathcal{M}_d,m)$. That is
\begin{equation*}
\mathcal{M}_d=\left\{A\subseteq \mathbb{R}^d\setminus\{0\}:\psi_E^{-1}(A)\in (\mathcal{L}\times\Sigma_P)'\right\}
\end{equation*}
and, for each $A\in\mathcal{M}_d$,
\begin{equation*}
m(A)=(\lambda_P\times\sigma_P)(\psi_E^{-1}(A)).
\end{equation*}
\item\label{property:BestIntegrationFormula} Given any Lebesgue measurable function $f:\mathbb{R}^d\to\mathbb{C}$ and $E\in \Exp(P)$, $f\circ \psi_E$ is $(\mathcal{L}\times\Sigma_P)'$-measurable and the following statements hold:
\begin{enumerate}
\item If $f\geq 0$, then
\begin{equation}\label{eq:BestIntegrationFormula}
\int_{\mathbb{R}^d}f(x)\,dx=\int_0^\infty\left(\int_S f(r^E\eta)\,\sigma_P(d\eta)\right)r^{\mu_P-1}\,dr=\int_S\left(\int_0^\infty f(r^E\eta)r^{\mu_P-1}\,dr\right)\sigma_P(d\eta).
\end{equation}
\item When $f$ is complex-valued, we have 
\begin{equation*}
    f\in L^1(\mathbb{R}^d)\hspace{.1cm}\mbox{ if and only if }\hspace{.1cm}f\circ\psi_E \in  L^1\left((0,\infty)\times S,(\mathcal{L}\times\Sigma_P)',\lambda_P\times\sigma_P\right)
\end{equation*} and, in this case, \eqref{eq:BestIntegrationFormula} holds.
\end{enumerate}
\end{enumerate}
\end{theorem}

\noindent For comparison with the above theorem, we would like to highlight two related (but distinct) results which also generalize \eqref{eq:StandardPolarIntegrationFormula} and have found utility in their respective contexts. The first appears in the context of Riemannian manifolds and can be found as (III.3.4) of \cite{chavel_riemannian_2006}. In that context, $S$ is replaced by the geodesic unit sphere $S_q$ centered at a point $q$ in a Riemannian manifold $M$ (with metric $g$ and connection $\nabla$), the paths $r\mapsto r^E\eta$ are replaced by geodesics $r\mapsto \gamma_\eta(r)$ for each $\eta\in S_q$, and the product measure $r^{\mu_P-1}\,dr\,\sigma_P(d\eta)$ is replaced by a measure of the form $\det(\mathcal{A}(r:\eta))\,dr\, d\mbox{Vol}_{S_q}(\eta)$ where $d\mbox{Vol}_{S_q}$ is the Riemannian volume measure on $S_q$ and $\mathcal{A}(r;\eta)$ is determined by the Riemannian curvature $R$ along the geodesic paths $\gamma_{\eta}(r)$. The second related result appears in the context of homogeneous (Lie) groups and can be found as Proposition 1.15 of \cite{folland_hardy_1982}. In that context, $P$ is replaced by a homogeneous norm $|\cdot|_G$ on the homogeneous group $G$, $S$ is replaced by  the unit sphere $S_G$ in the homogeneous norm, and $\mu_P$ is replaced by the homogeneous dimension of $G$. Perhaps obviously, the second result on homogeneous groups shares the most in common with Theorem \ref{thm:BestIntegrationFormula} but differs in context and in that we ask much less of $P$ and, consequently, $S$. In particular, $P$ can be continuous yet nowhere differentiable (in contrast to $|\cdot|_G$) and $S$ can be non-smooth and $B$ non-convex (in contrast to $S_G$ and $B_G$).\\

\noindent

\noindent Before we discuss the ideas behind the construction of $\sigma_P$ and the proof of Theorem \ref{thm:BestIntegrationFormula}, we first discuss two corollaries with well-known analogues in the classical setting. 

\begin{corollary}
Given $0\leq a<b$, suppose that $f:\overline{A_a^b}\to\mathbb{C}$ is continuous and define $\mathcal{I}:[a,b]\to\mathbb{C}$ by
\begin{equation*}
\mathcal{I}(r)=\int_{A_a^r}f(x)\,dx.
\end{equation*}
for $a\leq r\leq b$. Then $\mathcal{I}$ is continuously differentiable and
\begin{equation*}
\mathcal{I}'(r)=r^{\mu_P-1}\int_S f(r^E\eta)\sigma_P(d\eta)
\end{equation*}
on $[a,b]$ (or on $(0,b]$ provided that $a=0$ and $\mu_P<1$) where $E\in\Exp(P)$. In particular, if $f$ is a complex-valued function which is continuous on some open neighborhood of $\overline{B}$, then
\begin{equation*}
\int_S f(\eta)\sigma_P(\eta)=\frac{d}{dr}\left(\int_{B_r}f(x)\,dx\right)\bigg\vert_{r=1}
\end{equation*}
where this derivative is two-sided.
\end{corollary}
\begin{proof}
By virtue of Theorem \ref{thm:BestIntegrationFormula}, we have
\begin{equation*}
\mathcal{I}(r)=\int_0^\infty\int_S \chi_{A_a^r}(t^E\eta)f(t^E\eta)\,\sigma_P(d\eta) t^{\mu_P-1}\,dt=\int_a^r g(t)t^{\mu_P-1}\,dt
\end{equation*}
where
\begin{equation*}
g(t)=\int_S f(t^E\eta)\,\sigma_P(d\eta).
\end{equation*}
By virtue of the continuity of $f$ and the compactness of $S$, it is easy to see that $g(t)$ is continuous on $[a,b]$. By an appeal to the fundamental theorem of calculus, it follows that $\mathcal{I}$ is differentiable on $[a,b]$ (or on $(0,b]$ provided that $a=0$ and $\mu_P<1$) and
\begin{equation*}
\mathcal{I}'(r)=g(r)r^{\mu_P-1}=r^{\mu_P-1}\int_S f(r^E\eta)\,\sigma_P(d\eta).
\end{equation*}
To prove the final assertion, we first claim that there exists $\epsilon>0$ for which $\overline{B}\subseteq B_{1+\epsilon}\subseteq \mathcal{O}$. To see this, we assume to reach a contradiction that, for each $n\in\mathbb{N}$, there exists $x_n \in B_{1+1/n} \setminus \mathcal{O}$. Because $\overline{B_2}$ is compact, the sequence $\{x_n\}$ is bounded and thus has a convergent subsequence by the Bolzano–Weierstrass theorem. By a (possible) reassignment of $\{x_n\}$, we may therefore assume that $\lim_{n\to \infty} x_n = x$. Because $P$ is continuous and $1 \leq P(x_n) \leq 1 + 1/n$ for all $n\in \mathbb{N}$, $P(x)=\lim_{n\to\infty}P(x_n)=1$ and therefore $x\in S \subseteq \overline{B} \subseteq \mathcal{O}$. This implies that $\mathcal{O}$ contains an accumulation point $x$ of $\mathbb{R}^d\setminus\mathcal{O}$ which is impossible because $\mathcal{O}$ is an open set. As a consequence, the second assertion of the corollary follows immediately from the first where the derivative at $r=1$ is two-sided.
\end{proof}

\noindent As an application of the preceding corollary, information can be exchanged between the Fourier transforms $\{\widehat{\chi_{B_r}}\}_{r>0}$ of the characteristic functions $\{\chi_{B_r}\}_{r>0}$ and the Fourier transform of the surface measure $\sigma_P$ defined by
\begin{equation*}
\widehat{\sigma_P}(x)=\frac{1}{(2\pi)^d}\int_{S} e^{-i\eta\cdot x}\,\sigma_P(d\eta)
\end{equation*}
for $x\in\mathbb{R}^d$. Specifically, we have
\begin{equation}\label{eq:FTofSurfaceMeasureRelation}
\widehat{\sigma_P}(x)=\frac{d}{dr}\widehat{\chi_{B_r}}(x)\Big\vert_{r=1}\hspace{0.5cm}\mbox{and}\hspace{0.5cm}\widehat{\chi_B}(x)=\int_0^1\widehat{\sigma_P}(r^{E^*}x)r^{\mu_P-1}\,dr
\end{equation}
for $x\in\mathbb{R}^d$ where $E^*$ is the transpose of $E\in\Exp(P)$. It is well known that decay estimates for the Fourier transform of surface-carried measures (of the form $\sigma_P$) have rich applications to the theory of maximal averages \cite{stein_harmonic_1993}.  With such applications in mind, we refer the reader to the recent article \cite{greenblatt_fourier_2021} which skillfully utilizes relationships of the form \eqref{eq:FTofSurfaceMeasureRelation}.

\begin{corollary}\label{cor:IntegrateOnS}
Given $g:S\to\mathbb{C}$ and $E\in\Exp(P)$, define $f:\mathbb{R}^d\to\mathbb{C}$ by
\begin{equation*}
f(x)=\begin{cases}
\mu_P\cdot \chi_{(0,1)}(P(x))g(P(x)^{-E}x) & \mbox{ for }x\neq 0\\
0 & \mbox{ for }x=0
\end{cases}
\end{equation*}
for $x\in\mathbb{R}^d$ where $\chi_{(0,1)}(\cdot)$ is the indicator function of the interval $(0,1)$. Then $g\in L^1(S,\Sigma_P,\sigma_P)$ if and only if $f\in L^1(\mathbb{R}^d)$ and, in this case, 
\begin{equation}\label{eq:ACharacterizationofsigma}
    \int_{\mathbb{R}^d}f(x)\,dx=\int_Sg(\eta)\sigma_P(d\eta).
\end{equation}
\end{corollary}
\begin{proof}
Observe that, for the $(\mathcal{L}\times\Sigma_P)'$-measurable function $k(r,\eta)=\chi_{(0,1)}(r)g(\eta)$,
\begin{equation*}
    k\circ\psi_E^{-1}(x)=k(P(x),P(x)^{-E}x)=f(x)
\end{equation*}
for $x\in\mathbb{R}^d\setminus \{0\}$. By virtue of Theorem \ref{thm:BestIntegrationFormula}, it follows that $f$ is Lebesgue measurable and 
\begin{eqnarray*}
   \int_{\mathbb{R}^d}|f(x)|\,dx&=&\int_{S}\left(\int_{(0,\infty)}|k(r,\eta)|r^{\mu_P-1}\,dr\right)\sigma_P(d\eta)\\
    &=&\left(\int_S|g(\eta)|\sigma_P(d\eta)\right)\left(\int_0^1 \mu_Pr^{\mu_P-1}\,dr\right)\\
    &=&\int_S\abs{g(\eta)}\sigma_P(d\eta)
\end{eqnarray*}
and therefore $f\in L^1(\mathbb{R}^d)$ if and only if $g\in L^1(S,\Sigma_P,\sigma_P)$ and $\|f\|_{L^1(\mathbb{R}^d)}=\|g\|_{L^1(S)}$. By an analogous computation (for $f$ instead of $|f|$), we have
\begin{equation*}
    \int_{\mathbb{R}^d}f(x)\,dx=\int_S g(\eta)\sigma_P(d\eta),
\end{equation*}
by virtue of Property \ref{property:BestIntegrationFormula} of Theorem \ref{thm:BestIntegrationFormula}. 
\end{proof}

\noindent As it turns out, \eqref{eq:ACharacterizationofsigma} is characterizing in the sense that it can be used to construct $\sigma_P$ via the Riesz representation theorem. This is the approach taken, for example, in \cite{folland_hardy_1982} and \cite{baker_integration_1997}; in fact, it is asserted in \cite{baker_integration_1997} that viewing \eqref{eq:ACharacterizationofsigma} as a consequence of Theorem \ref{thm:BestIntegrationFormula} (in the classical setting) is folkloric. We have decided to follow this so-called folkloric approach, which avoids the Riesz representation theorem and is that suggested by \cite{rudin_real_1987} and \cite{folland_how_2001} (in the classical setting), mainly because it is constructive, illustrative, and allows us to precisely describe the $\sigma$-algebra $\Sigma_P$.\\

\noindent For the remainder of this introductory section, we outline the remainder of this article and herein describe the heuristics of our construction of $\sigma_P$. Section \ref{sec:Homogeneous} is a short section which contains a proof of Proposition \ref{prop:PositiveHomogeneousCharacterization}, outlines the basics of positive homogeneous functions, and introduces the related notion of subhomogeneous functions (see Subsection \ref{subsec:SubhomogeneousFunctions}). In Section \ref{sec:ConvolutionPowers}, we focus on the study of convolution powers of complex-valued function on $\mathbb{Z}^d$. After a short account of the study's history, we introduce the main result of that section, Theorem \ref{thm:ConvolutionPowerEstimate}. The theorem establishes sup-norm-type estimates for the iterative convolution powers $\phi^{(n)}$ of a complex-valued function $\phi$ whose Fourier transform satisfies certain hypotheses. Partially extending results of \cite{randles_convolution_2015} and \cite{randles_convolution_2017}, our result is new and its proof makes use of Theorem \ref{thm:BestIntegrationFormula} and the Van der Corput lemma. Subsection \ref{subsec:Examples} treats several concrete examples to which Theorem \ref{thm:ConvolutionPowerEstimate} applies. A forthcoming article will treat local limit theorems for complex-valued functions satisfying the hypotheses of Theorem \ref{thm:ConvolutionPowerEstimate}. Section \ref{sec:ProofofBest} focuses on the proof of Theorem \ref{thm:BestIntegrationFormula} and is broken into several subsections. In Subsection \ref{subsec:ConstructionofSigma}, we take $E\in\Exp(P)$ and define a measure $\sigma_{P,E}$ on $S$ by adapting the classical construction of the spherical measure, described in Remark 2 of \cite{folland_how_2001}, to our positive homogeneous setting. Specifically, for a given set $F\subseteq S$, we contract $F$ into a quasi-conical region of $B$ by putting $\widetilde{F}=\widetilde{F}_E=\{r^E\eta:0<r<1,\eta\in F\}$
and setting
\begin{equation*}
\sigma_{P,E}(F)=\mu_P \cdot m(\widetilde{F})
\end{equation*}
provided that $\widetilde{F}$ is Lebesgue measurable. Figures \ref{fig:level_set_F} and \ref{fig:level_set_F_3D} illustrate these quasi-conical regions in $\mathbb{R}^2$ and $\mathbb{R}^3$, respectively.

\begin{figure}[!htb]
    \centering
    \includegraphics[scale=0.7, trim={1cm 1cm 1cm 0.5cm},clip]{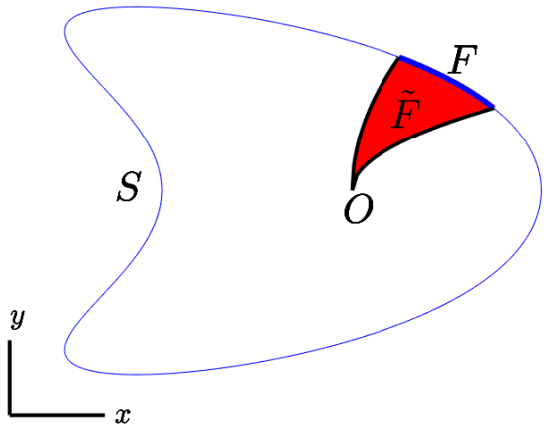}
    \caption{Quasi-conical region $\widetilde{F}=\widetilde{F}_E$ (in red) for $F\subseteq S$. Here, $S$ is the unital level set of $P=P_2$ from Example \ref{exp:Polynomial} and $E\in\Exp(P_2)$ has the standard representation $\diag(1/2,1/4)$.}
    \label{fig:level_set_F}
\end{figure}

\begin{figure}[!htb]
    \centering
    \includegraphics[scale=0.7, trim={1cm 3cm 1cm 2cm},clip]{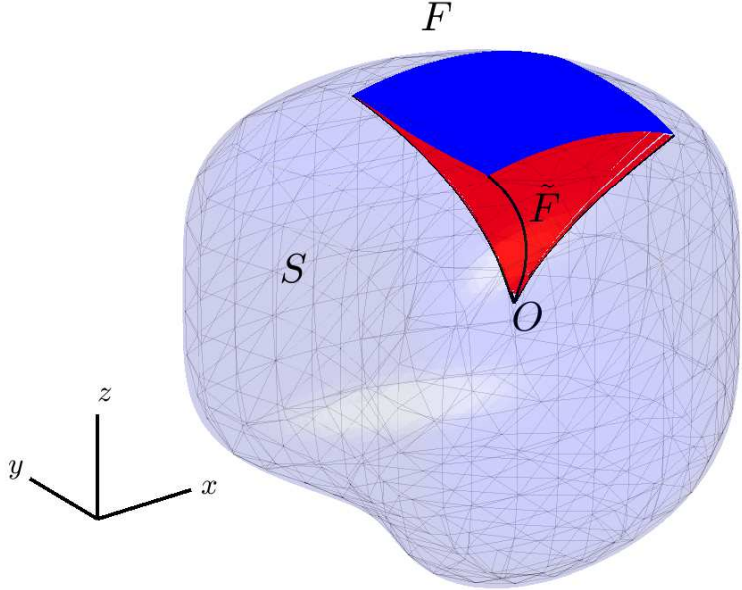}
    \caption{Quasi-conical region $\widetilde{F}=\widetilde{F}_E$ (in red) for $F\subseteq S$. Here, $S$ is the unital level set of $P(x,y,z) = x^2 + xy^2 + y^4 + z^4$ and $E\in\Exp(P)$ has standard representation $\diag(1/2,1/4,1/4)$. }
    \label{fig:level_set_F_3D}
\end{figure}

\noindent In Subsection \ref{subsec:ProductMeasure}, we turn our focus to an associated product measure $\lambda_P\times\sigma_{P,E}$ on $(0,\infty)\times S$ with which we are able to formulate and prove a generalization of \eqref{eq:StandardPolarIntegrationFormula} using the measure $\sigma_{P,E}$; this is Theorem \ref{thm:MainIntegrationFormula}. We then derive a number of corollaries of Theorem \ref{thm:MainIntegrationFormula}, including the result that $\sigma_{P,E}$ is a Radon measure on $S$. As everything done in Subsections \ref{subsec:ConstructionofSigma} and \ref{subsec:ProductMeasure} is done using the contracting group $\{r^E\}$ for a chosen $E\in\Exp(P)$, it isn't clear, \textit{a priori}, exactly how $\sigma_{P,E}$ is dependent on the choice of $E\in\Exp(P)$, if at all. In Subsection \ref{subsec:IndependentofE}, we prove that $\sigma_{P,E}$ is, in fact, independent of the choice of $E\in \Exp(P)$ and, upon writing $\sigma_P=\sigma_{P,E}$, we immediately obtain Theorem \ref{thm:BestIntegrationFormula}. All throughout Section \ref{sec:ProofofBest}, our construction uses only tools from point-set topology and measure theory.  In Section \ref{sec:SigmaForSmoothP}, we study the special case in which a positive homogeneous function $P$ is additionally smooth. In that case, we will find that $S$ is a smooth compact embedded hypersurface of $\mathbb{R}^d$ and the measure $\sigma_P$ is characterized by a differential form $d\sigma_P$ and closely related to the Riemannian volume on $S$; this is Theorem \ref{thm:RiemannLebesgue}. The section also contains a number of results helpful to the computation of integrals on $S$ with respect to $\sigma_P$. Finally, the appendix outlines some basic results concerning one-parameter contracting groups.

\section{Positive Homogeneous Functions}\label{sec:Homogeneous}

In this short section, we treat some basic results on positive homogeneous functions and introduce the useful and related concept of subhomogeneous functions. Several times throughout, we will appeal to results concerning contracting groups which can be found in the appendix.

\begin{proof}[Proof of Proposition \ref{prop:PositiveHomogeneousCharacterization}]
In the case that $d=1$, it is easy to see that every function satisfying the hypotheses is of the form
\begin{equation*}
P(x)=\begin{cases}
P(1)x^\alpha & \mbox{ for }x\geq 0 \\
P(-1)(-x)^\alpha &\mbox{ for }x<0
\end{cases}
\end{equation*}
for some $\alpha>0$ where $P(1),P(-1)>0$ and $\Exp(P)$ consists only of the linear function $x\mapsto x/\alpha$. In this setting, it is easy to see that Conditions \ref{cond:SisCompact}--\ref{cond:InfiniteLimit} are satisfied (always and) simultaneously. We shall therefore assume that $d>1$ for the remainder of the proof.

\begin{subproof}[$\ref{cond:SisCompact}\Rightarrow\ref{cond:PisAboveOne}$]
Given that $S$ is compact, it is bounded and so we have a positive number $M$ for which $P(x)\neq 1$ for all $|x|\geq M$. Observe that, if for two points $x_1,x_2\in \mathbb{R}^d\setminus\mathbb{B}_M$, $P(x_1)<1<P(x_2)$ or $P(x_2)<1<P(x_1)$, then by virtue of the path connectedness of $\mathbb{R}^d\setminus\mathbb{B}_M$ and the intermediate value theorem, we would be able to find  $x_0\in\mathbb{R}^d\setminus\mathbb{B}_M$ for which $P(x_0)=1$, an impossibility. Therefore, to show that Condition \ref{cond:PisAboveOne} holds, we must simply rule out the case in which $P(x)<1$ for all $|x|\geq M$. Let us therefore assume, to reach a contradiction, that this alternate condition holds. In this case, we take $E\in\Exp(P)$ and $y\in\mathbb{R}^d\setminus \{0\}$ and observe that
\begin{equation*}
\lim_{r\to\infty}P(r^Ey)=\lim_{r\to\infty}rP(y)=\infty.
\end{equation*}
By virtue of our supposition, we find that $|r^Ey|<M$ for all sufficiently large $r$. In particular, there exists a sequence $r_k\to\infty$ for which $|r_k^Ey|\leq M$ for all $k$ and 
\begin{equation*}
\lim_{k\to\infty}P(r_k^Ey)=\infty.
\end{equation*}
Because $\overline{\mathbb{B}_M}$, the closure of $\mathbb{B}_M$, is compact, $\{r_k^Ey\}$ has a convergent subsequence which we also denote by $\{r_k^Ey\}$ by a slight abuse of notation. In view of the continuity of $P$ at $\eta=\lim_{k\to\infty}r_k^Ey$, we have
\begin{equation*}
P(\eta)=\lim_{k\to\infty}P(r_k^Ey)=\lim_{k\to\infty}r_kP(y)=\infty,
\end{equation*}
which is impossible. Thus Condition \ref{cond:PisAboveOne} holds.
\end{subproof}
\begin{subproof}[$\ref{cond:PisAboveOne}\Rightarrow\ref{cond:Contracting}$]
We shall prove the contrapositive statement. Suppose that, for $E\in\Exp(P)$, $\{r^E\}$ is not contracting. In this case, by virtue of Proposition \ref{prop:ContractingCharacterization}, there exists $x\in\mathbb{R}^d\setminus\{0\}$ and a sequence $r_k\to 0$ for which $r_k^Ex$ does not converge to zero in $\mathbb{R}^d$. If our sequence $\{r_k^Ex\}$ is bounded, then it must have a convergent subsequence $\{r_{k_m}^Ex\}$ with non-zero subsequential limit $\eta=\lim_{m\to\infty}r_{k_m}^Ex$. By the continuity of $P$, we have
\begin{equation*}
P(\eta)=\lim_{m\to\infty}P(r_{k_m}^Ex)=\lim_{k\to\infty}r_{k_m}P(x)=0
\end{equation*}
which cannot be true for it would violate the positive definiteness of $P$. We must therefore consider the other possibility: The sequence $\{r_k^Ex\}$ is unbounded. In particular, there must be some $k_0$ for which $r_{k_0}<1/P(x)$ and $|r_{k_0}^Ex|>M$. Upon putting $y=r_{k_0}^Ex$, we have $|y|>M$ and  $P(y)=P(r_{k_0}^Ex)=r_{k_0}P(x)<1$ which shows that Condition \ref{cond:PisAboveOne} cannot hold.
\end{subproof}
\begin{subproof}[$\ref{cond:Contracting}\Rightarrow\ref{cond:ThereExistsContracting}$] This is immediate.
\end{subproof}
\begin{subproof}[$\ref{cond:ThereExistsContracting}\Rightarrow\ref{cond:InfiniteLimit}$]
Let $E\in\Exp(P)$ be such that the one-parameter group $\{r^E\}$ is contracting and let $\{x_k\}\subseteq\mathbb{R}^d$ be such that $\lim_{k\to\infty}|x_k|=\infty$. By virtue of Proposition \ref{prop:ScaleFromSphere}, there exist sequences $\{r_k\}\subseteq (0,\infty)$ and $\{\eta_k\}\in\mathbb{S}$ for which $r_k^E\eta_k=x_k$ for all $k$ and $\lim_{k\to\infty}r_k=\infty$. Given that $P$ is continuous and strictly positive on the compact set $\mathbb{S}$, we have $\inf_{\eta\in\mathbb{S}}P(\eta)>0$ and therefore
\begin{equation*}
\liminf_k P(x_k)=\liminf_k r_kP(\eta_k)\geq \liminf_k r_k\left(\inf_{\eta\in\mathbb{S}}P(\eta)\right)=\infty
\end{equation*}
showing that $\lim_{k\to\infty} P(x_k)=\infty$, as desired.
\end{subproof}
\begin{subproof}[$\ref{cond:InfiniteLimit}\Rightarrow\ref{cond:SisCompact}$]
Because $S$ is the preimage of the closed singleton $\{1\}$ under the continuous function $P$, it is closed. By virtue of Condition \ref{cond:InfiniteLimit}, $S$ is also be bounded and thus compact in view of the Heine-Borel theorem.
\end{subproof}
\end{proof}

\begin{proposition}\label{prop:SymCompact}
For each positive homogeneous function $P$, $\Sym(P)$ is a compact subgroup of $\Gl(\mathbb{R}^d)$. In particular, it is a subgroup of the orthogonal group $O(\mathbb{R}^d)$.
\end{proposition}
\begin{proof}
As discussed in the introduction, it is clear that $\Sym(P)$ is a subgroup of $\Gl(\mathbb{R}^d)$. To see that $\Sym(P)$ is compact, it suffices to prove that $\Sym(P)$ is closed and bounded by virtue of the Heine-Borel theorem. To this end, let $\{O_n\}\subseteq\Sym(P)$ be a sequence converging to $O\in \Gl(\mathbb{R}^d)$. For each $x\in\mathbb{R}^d$, the continuity of $P$ guarantees that
\begin{equation*}
P(Ox)=P\left(\lim_{n\to\infty}O_nx\right)=\lim_{n\to\infty}P(O_nx)=\lim_{n\to\infty}P(x)=P(x).
\end{equation*}
Hence, $O\in\Sym(P)$ and so $\Sym(P)$ is closed.

We assume, to reach a contradiction, that $\Sym(P)$ is not bounded. In this case, there is a sequence $\{\eta_n\}\subseteq \mathbb{S}$ for which $\lim_{n\to\infty}|O_n\eta_n|=\infty$. Given that $\mathbb{S}$ is compact, by passing to a subsequence if needed, we may assume without loss in generality that $\lim_{n\to\infty}\eta_n=\eta\in\mathbb{S}$. By virtue of Proposition \ref{prop:PositiveHomogeneousCharacterization} and the continuity of $P$,
\begin{equation*}
P(\eta)=\lim_{n\to\infty}P(\eta_n)=\lim_{n\to\infty}P(O_n\eta_n)=\infty
\end{equation*}
which is impossible. Hence $\Sym(P)$ is bounded.
\end{proof}

\begin{corollary}\label{cor:TraceisInvariant}
Let $P$ be a positive homogeneous function, then
\begin{equation*}
\tr E=\tr E'>0
\end{equation*}
for all $E,E'\in\Exp(P)$.
\end{corollary}
\begin{proof}
By virtue of Propositions \ref{prop:PositiveHomogeneousCharacterization} and \ref{prop:ContractingTrace}, $\tr E>0$ for all $E\in\Exp(P)$. It remains to show that the trace map is constant on $\Exp(P)$. To this end, let $E,E'\in\Exp(P)$. Then, for all $r>0$ and $x\in\mathbb{R}^d$,
\begin{equation*}
P(x)=r(1/r)P(x)=rP((1/r)^{E'}x)=P(r^E(1/r)^{E'}x)=P(r^{E}r^{-E'}x).
\end{equation*}
Thus $O_r=r^{E}r^{-E'}\in\Sym(P)$ for each $r>0$. In view of Proposition \ref{prop:SymCompact}, Proposition \ref{prop:ContinuousGroupProperties} and the homomorphism property of the determinant,
\begin{equation*}
1=\det(O_r)=\det(r^{E}r^{E'})=\det(r^{E})\det(r^{-E'})=r^{\tr E}r^{\tr E}=r^{\tr E-\tr E'}
\end{equation*}
for all $r>0$ and therefore $\tr E=\tr E'$.
\end{proof}

\noindent For a continuous, positive definite function $P$ for which $\Exp(P)$ is non-empty, the following proposition gives a sufficient condition for $P$ to be positive homogeneous. As discussed in Example \ref{exp:Weierstrass} above, it is this condition that was used to define ``positive homogeneous polynomial" in \cite{randles_convolution_2017}.

\begin{proposition}\label{prop:PosHomSufficientCondition}
If $P$ is continuous, positive definite and $\Exp(P)$ contains an $E\in\End(\mathbb{R}^d)$ with real spectrum, then $\{r^E\}$ is contracting and hence all of the above conditions are (simultaneously) met. 
\end{proposition}
\begin{proof}
Since $\Spec(E)$ is real, the characteristic polynomial of $E$ factors completely over $\R$ and so we may apply the Jordan-Chevalley decomposition to write $E=D+N$ where $D$ is diagonalizable, $N$ is nilpotent, and $DN=ND$. Let $v_1,v_2,\dots,v_d \in \R^d$ be an eigenbasis of $D$ whose corresponding eigenvalues $\lambda_1,\lambda_2,\dots,\lambda_d$ satisfy $\lambda_k\leq \lambda_{k+1}$ for all $k=1,2\dots,d-1$.

Let us assume, to reach a contradiction, 
that $\{ r^E \}$ is not contracting. Repeating the same argument given in $\ref{cond:PisAboveOne}\Rightarrow\ref{cond:Contracting}$ in the proof of Proposition \ref{prop:PositiveHomogeneousCharacterization}, leaves us with only one possibility: There is a non-zero $x = \sum^d_{i=1}\alpha_i v_i \in\mathbb{R}^d$, and a sequence $r_k\to 0$ for which $|r_k^E x|\to\infty$. Let $n+1$ denote the index of $N$, then we have
\begin{equation*}
r^E_k x = r_k^{N+D} x 
= r_k^N r_k^D x 
= \sum_{j=0}^n\sum_{i=1}^d \f{r_k^{\lambda_i}(\log r_k)^j}{j!}   \alpha_iN^j v_i
\end{equation*}
for all $k$. Since $\abs{r_k^E x} \to \infty$ and $r_k \to 0$, at least one eigenvalue of $D$ must be non-positive. To see this, suppose $\lambda_i > 0$ for all $i = 1,2,\dots,d$, then in view of L'H\^{o}pital's rule we have
\begin{equation*}
    \lim_{r_k \to 0}(\log r_k)^j r_k^{\lambda_i} = 0  
\end{equation*}
for any $j =0, 1,2,\dots,n$ and $i =1,2,\dots,d$, which implies that $\abs{r^E_k x} \not\to \infty$ as $r_k \to 0$, contradicting our assumption. Thus, $\lambda_1 = \min\{ \Spec(D)\} \leq 0$. Let $k$ be such that $N^k v_1 \neq 0$ but $N^{k+1} v_1 = 0$, then
\begin{equation*}
    r^E N^k v_1 = r^D r^N N^k v_1 = r^D \sum_{j=0}^\infty \f{(\log r)^j}{j!}N^j N^k v_1 = r^D N^k v_1 = N^k r^D  v_1 =  r^{\lambda_{1}} N^k  v_1
\end{equation*}
where we have used the fact that $DN = ND$. If $\lambda_1= 0$, then 
\begin{equation*}
    \infty =  \lim_{r\to \infty} rP(N^k v_1)  = \lim_{r\to \infty} P( r^E N^k v_1) =  \lim_{r\to \infty}P(r^{0} N^k v_1)= \lim_{r\to \infty}P( N^k v_1) = P(N^k v_1)
\end{equation*}
which is impossible since $P$ is continuous at $N^k v_1$. On the other hand, if $\lambda_1 < 0$, then
\begin{equation*}
    \infty = \lim_{r\to \infty} rP(N^k v_1) = \lim_{r\to \infty} P(r^E N^k v_1) = \lim_{r\to \infty}P(r^{\lambda_1} N^k v_1) = P(0) = 0
\end{equation*}
which is also impossible. 
\end{proof}

\begin{proposition}\label{prop:ExpP}
For any  $O \in \Sym{(P)} $
\begin{equation*}
    \Exp(P) = O^* \Exp(P) O
\end{equation*}
where $O^*$ is the transpose of $O$. In other words, the set $\Exp(P)$ is invariant under conjugation by $\Sym(P)$.
\end{proposition}

\begin{proof}
Since $\Sym(P)$ is a subgroup of the orthogonal group $O(\mathbb{R}^d)$, $O^* = O^{-1} \in \Sym{P}$ whenever $O\in\Sym(P)$. For a fixed $O\in\Sym(P)$, it is easy to see that the map $\Exp(P)\ni E\mapsto  O^* E O\in\Exp(P)$ is a bijection and hence $\Exp(P)=O^* \Exp(P) O$. 
\end{proof}

\subsection{Subhomogeneous functions}\label{subsec:SubhomogeneousFunctions}
In this subsection, we introduce the notions of subhomogeneous functions and strongly subhomogeneous functions with respect to a given endomorphism $E\in\End(\mathbb{R}^d)$. As discussed in the following section, these notions appear naturally in the study of convolution powers of complex-valued functions on $\mathbb{Z}^d$.

% \begin{comment}
% \begin{definition}\label{def:homogeneous_types}
% Let $Q$ be a complex-valued function defined on an open neighborhood of $0$ in $\mathbb{R}^d$ and let $E\in\End(\mathbb{R}^d)$ be such that $\{r^E\}$ is a contracting group.
% \begin{enumerate}
% \item We say that $Q$ is \textbf{subhomogeneous} with respect to $E$ if, for each $\epsilon>0$ and compact set $K$, there is a $\delta>0$ for which
% \begin{equation*}
% \abs{Q(r^E\xi)}\leq \epsilon r
% \end{equation*}
% for all $0<r<\delta$ and $\xi\in K$.
% \item In the case that $Q$ is differentiable on its domain, we say that $Q$ is \textbf{strongly subhomogeneous} with respect to $E$ if, for each $\epsilon>0$ and compact set $K$, there is a $\delta>0$ for which
% \begin{equation*}
% \abs{\partial_r Q(r^E\xi)}\leq \epsilon
% \end{equation*}
% for all $0<r<\delta$ and $\xi\in K$.
% \end{enumerate}
% \end{definition}
% \end{comment}

\begin{definition}\label{def:homogeneous_types}
Let $Q$ be a continuous and complex-valued function defined on an open neighborhood $\mathcal{O}$ of $0$ in $\mathbb{R}^d$ and let $E\in\End(\mathbb{R}^d)$ be such that $\{r^E\}$ is a contracting group.
\begin{enumerate}
\item We say that $Q$ is \textbf{subhomogeneous with respect to $E$} if, for each $\epsilon>0$ and compact set $K\subseteq\mathbb{R}^d$, there is a $\delta>0$ for which
\begin{equation*}
\abs{Q(r^E\xi)}\leq \epsilon r
\end{equation*}
for all $0<r<\delta$ and $\xi\in K$.
\item Given $l\geq 1$, we say that $Q$ is \textbf{strongly subhomogeneous with respect to $E$ of order $l$} if $Q\in C^l(\mathcal{O})$ and, for each $\epsilon>0$ and compact set $K\subseteq\mathbb{R}^d$, there is a $\delta>0$ for which
\begin{equation*}
    \abs{r^j\partial_r^j Q(r^E\xi)}\leq \epsilon r
\end{equation*}
for all $j=1,2,\dots,l$, $0<r<\delta$ and $\xi\in K$.
\end{enumerate}
When the endomorphism $E$ is understood (and fixed), we will say that $Q$ is subhomogeneous if it is subhomogeneous with respect to $E$. Also, we will say that $Q$ is $l$-strongly subhomogeneous if it is strongly subhomogeneous with respect to $E$ of order $l$.
\end{definition}
\noindent The following proposition, in particular, justifies our choice of vocabulary and give credence to the interpretation that $0$-strongly subhomogeneous is synonymous with subhomogeneous.

\begin{proposition}\label{prop:supersub_implies_sub}
Let $E\in\End(\mathbb{R}^d)$ be for which $\{r^E\}$ is contracting and suppose that, for some $l\geq 1$, $Q$ is strongly subhomogeneous with respect to $E$ of order $l$. If $Q(0)=0$, then $Q$ is subhomogeneous with respect to $E$.
\end{proposition}
\begin{proof}
Let $\epsilon>0$ and $K$ be a compact set.  In view of our supposition that $Q$ is strongly subhomogeneous with respect to $E$ or order $k$, let $\delta>0$ be given so that $\abs{r\partial_r Q(r^E\xi)}\leq \epsilon r$ for all $\xi\in K$ and $0<r< \delta$. Given that $r^E$ is a contracting group and $Q(0)=0$, it follows that, for each $\xi\in K$, $f_{\xi}:[0,\delta)\to\mathbb{C}$ defined by
\begin{equation*}
f_{\xi}(r)=\begin{cases}
Q(r^E\xi) & 0<r<\delta\\
0 & r=0
\end{cases}
\end{equation*}
is differentiable on $(0,\delta)$ and continuous on $[0,\delta)$, for each $\xi\in K$. Consequently, for every $0<r<\delta$ and $\xi\in K$, the mean value theorem guarantees a $c=c_{\xi,r}\in (0,r)$ for which 
\begin{equation*}
\abs{f_{\xi}(r)-f_{\xi}(0)}\leq r\abs{f_{\xi}'(c)}=r\abs{\partial_r Q(r^E\xi)\Big\vert_{r=c}}\leq r\epsilon.
\end{equation*}
Consequently, for all $0<r<\delta$ and $\xi\in K$,
\begin{equation*}
\abs{Q(r^E\xi)}=\abs{f_{\xi}(r)-f_{\xi}(0)}\leq r\epsilon.
\end{equation*}
\end{proof}

\begin{proposition}\label{prop:Subhomequivtolittleoh}
Let $P$ be positive homogeneous and $\widetilde{P}$ be complex-valued and continuous on a neighborhood of $0$ in $\mathbb{R}^d$. The following are equivalent:
\begin{enumerate}[label=(\alph*), ref=(\alph*)]
    \item\label{item:Subhomequivtolittleoh1} $\widetilde{P}(\xi)=o(P(\xi))$ as $\xi\to 0$.
    \item\label{item:Subhomequivtolittleoh2} For every $E\in\Exp(P)$, $\widetilde{P}$ is subhomogeneous with respect to $E$.
    \item\label{item:Subhomequivtolittleoh3} There exists $E\in\Exp(P)$ for which $\widetilde{P}$ is subhomogeneous with respect to $E$.
\end{enumerate}
\end{proposition}
\begin{proof}
\begin{subproof}[\ref{item:Subhomequivtolittleoh1} $\Rightarrow$ \ref{item:Subhomequivtolittleoh2}] Let $\epsilon>0$, $K$ be a compact set and choose $E\in \Exp(P)$. Given our supposition that $\widetilde{P}(\xi)=o(P(\xi))$ as $\xi\to 0$, we can find an open neighborhood $\mathcal{O}$ of $0$ for which 
\begin{equation*}
\abs{\widetilde{P}(\xi)}\leq \frac{\epsilon}{1+\sup_{\eta\in K}P(\eta)}P(\xi)
\end{equation*}
for all $\xi\in \mathcal{O}$. Now, because $r^E$ is contracting in view of Proposition \ref{prop:PositiveHomogeneousCharacterization}, we can find a $\delta>0$ for which $r^E\xi\in \mathcal{O}$ for all $0<r<\delta$ and $\xi\in K$ by virtue of Proposition \ref{prop:ContractingCapturesCompact}. Consequently, for all $0<r<\delta$ and $\xi\in K$,
\begin{equation*}
\abs{\widetilde{P}(r^E\xi)}\leq \frac{\epsilon}{1+\sup_{\eta\in K}P(\eta)}P(r^E\xi)=\epsilon r\frac{P(\xi)}{1+\sup_{\eta\in K}P(\eta)}\leq r\epsilon.
\end{equation*}
\end{subproof}

\begin{subproof} [\ref{item:Subhomequivtolittleoh2} $\Rightarrow$ \ref{item:Subhomequivtolittleoh3}] This implication is trivial.
\end{subproof}
 
 \begin{subproof}[\ref{item:Subhomequivtolittleoh3} $\Rightarrow$ \ref{item:Subhomequivtolittleoh1}]  Let $\epsilon>0$. Choose $E\in \Exp(P)$ and let $S=\{\eta\in\mathbb{R}^d:P(\eta)=1\}$. Using the supposition that $\widetilde{P}$ is subhomogeneous with respect to $E$, we may choose $\delta>0$ for which
\begin{equation*}
\abs{\widetilde{P}(r^E\eta)}\leq \epsilon r 
\end{equation*}
 for all $0<r<\delta$ and $\eta\in S$. We remark that, in view of the continuity of $\widetilde{P}$ and the fact that $r^E$ is contracting, this inequality ensures that $\widetilde{P}(0)=0$. We fix $\mathcal{O}$ to be the open set $B_\delta=\{\xi\in\mathbb{R}^d:P(\xi)<\delta\}$. For each non-zero $\xi\in\mathcal{O}$, we observe that $\xi=r^E\eta$ where $0<r=P(\xi)<\delta$ and $\eta=P(\xi)^{-E}\xi\in S$ and therefore
\begin{equation*}
\abs{\widetilde{P}(\xi)}=\abs{\widetilde{P}(r^E\eta)}\leq r\epsilon=\epsilon  P(\xi)
\end{equation*}
If $\xi=0$, obviously, $\abs{\widetilde{P}(\xi)}=0=\epsilon P(0)=\epsilon P(\xi)$. Thus, for all $\xi\in\mathcal{O}$,
\begin{equation*}
\abs{\widetilde{P}(\xi)}\leq\epsilon P(\xi),
\end{equation*}
as desired.
\end{subproof}
\end{proof}

\begin{proposition}\label{prop:2StronglySubhomogeneous}
Let $E\in\End(\mathbb{R}^d)$ be for which $\{r^E\}$ is contracting and suppose that $Q$ is strongly subhomogeneous with respect to $E$ of order $2$. Given $\alpha>0$, set $F=\alpha E$. Then, for any $\epsilon>0$ and compact set $K$,
\begin{equation*}
    \abs{\theta\partial_\theta Q(\theta^F\eta)}\leq \epsilon \theta^\alpha
\end{equation*}
and
\begin{equation*}
    \abs{\theta^2\partial_\theta^2 Q(\theta^F\eta)}\leq \epsilon \theta^\alpha
\end{equation*}
for all $0<\theta\leq \delta^{1/\alpha}$ and $\eta\in K$.
\end{proposition}
\begin{proof}
Let $\epsilon>0$ and $K\subseteq\mathbb{R}^d$ be a compact set. By virtue of the strong subhomogeneity of $Q$, let $\delta>0$ be such that
\begin{equation*}
    \abs{r\partial_r Q(r^E\eta)}\leq \epsilon' r\hspace{1cm}\mbox{and}\hspace{1cm}\abs{r^2\partial_r^2Q(r^E\eta)}\leq \epsilon' r
\end{equation*}
for $0<r<\delta$ and $\eta\in K$ where $\epsilon'=\epsilon/(2\alpha^2+\alpha)$. We set $r=\theta^\alpha$ so that $r^E=\theta^F$ and observe that
\begin{equation*}
    \abs{\theta\partial_\theta Q(\theta^F\eta)}=\abs{\theta\partial_rQ(r^E\eta)\frac{\partial r}{\partial\theta}}= \abs{\theta\partial_rQ(r^E\eta)\alpha \theta^{\alpha-1}}=\alpha r\abs{\partial_rQ(r^E\eta)}\leq \alpha\epsilon' r<\epsilon \theta^\alpha
\end{equation*}
for all $0<\theta\leq \delta^{1/\alpha}$ and $\eta\in K$. Further, we have
\begin{eqnarray*}
    \abs{\theta^2\partial_\theta^2Q(\theta^F\eta)}&=&\theta^2\abs{\partial_r^2Q(r^E\eta)\left(\frac{\partial r}{\partial\theta}\right)^2+\partial_r Q(r^E\eta)\frac{\partial^2 r}{\partial \theta^2}}\\
    &=&\theta^2\abs{\partial_r^2 Q(r^E\eta) \alpha^2\theta^{2\alpha-2}+\partial_r Q(r^E\eta)\alpha(\alpha-1)\theta^{\alpha-2}}\\
    &\leq &\alpha^2\abs{r^2Q(r^E\eta)}+\abs{\alpha(\alpha-1)}\abs{r\partial_r Q(r^E\eta)}\\
    &< &\alpha^2 \epsilon' r+\abs{\alpha^2-\alpha}\epsilon' r\\
    &<&\epsilon\theta^\alpha
\end{eqnarray*}
for all $0<\theta<\delta^{1/\alpha}$ and $\eta\in K$.
\end{proof}

\section{An application: Estimates for convolution powers}\label{sec:ConvolutionPowers}

We denote by $\ell^1(\mathbb{Z}^d)$ the set of functions $\phi:\mathbb{Z}^d\to\mathbb{C}$ for which
\begin{equation*}
\|\phi\|_1:=\sum_{x\in\mathbb{Z}^d}\abs{\phi(x)}<\infty.
\end{equation*}
Given $\psi,\phi\in \ell^1(\mathbb{Z}^d)$, the convolution product $\psi\ast\phi\in\ell^1(\mathbb{Z}^d)$ is defined by
\begin{equation*}
\left(\psi\ast\phi\right)(x)=\sum_{y\in\mathbb{Z}^d}\psi(x-y)\phi(y)
\end{equation*}
for $x\in\mathbb{Z}^d$. For a given $\phi\in\ell^1(\mathbb{Z}^d)$, we are interested in the asymptotic behavior of its convolution powers $\phi^{(n)}\in\ell^1(\mathbb{Z}^d)$ defined iteratively by $\phi^{(n)}=\phi^{(n-1)}\ast\phi$ for $n\geq 1$ where $\phi^{(0)}=\phi$. In the special case that $\phi$ is non-negative (and satisfies mild summability conditions), the asymptotic behavior of $\phi^{(n)}$ is well understood and is the subject of the local (central) limit theorem. For accounts of this story and its connection to probability and random walk theory, we encourage the reader to see the excellent texts \cite{lawler_random_2010} and \cite{spitzer_principles_1964} (see also Section 7.6 of \cite{randles_convolution_2017}). When $\phi$ is generally complex-valued (or simply a real-valued function taking on both positive and negative values), its convolution powers can exhibit exotic behaviors not seen in the probabilistic setting. The problem of describing these behaviors dates back to Erastus L. De Forest in his study of data smoothing in the nineteenth century and was further pursued by Isaac J. Schoenberg and Thomas N. E.  Greville. In the 1960's, spurred by advancements in scientific computing, the study was reinvigorated by its application to numerical solutions to partial differential equations. The article \cite{diaconis_convolution_2014} provides a full account of this history and references to the literature. Regarding recent developments, mostly in the context of one spatial dimension, we encourage the reader to see the articles \cite{diaconis_convolution_2014,randles_convolution_2015, coulombel2020generalized, randles_convolution_2017}. Concerning global space-time estimates, the article \cite{diaconis_convolution_2014} establishes Gaussian and sub-Gaussian estimates for the convolution powers of finitely-supported complex-valued functions on $\mathbb{Z}$ whose Fourier transform (characteristic function) satisfies certain hypotheses and the article \cite{coulombel2020generalized} focuses on Gaussian estimates and extends results of \cite{diaconis_convolution_2014} and \cite{randles_convolution_2017}. The articles \cite{diaconis_convolution_2014} and \cite{randles_convolution_2015} treat local limit theorems and sup norm estimates for the convolution powers of complex-valued functions on $\mathbb{Z}$; the latter provides a complete description of local limit theorems and sup norm estimates for the class of finitely-supported complex-valued functions on $\mathbb{Z}$, essentially resolving \textit{De Forest's problem}. In the general context of $\mathbb{Z}^d$, \cite{randles_convolution_2017} treats local limit theorems, global space-time estimates and sup norm estimates for the convolution powers of complex-valued functions whose Fourier transform satisfies certain assumption discussed below.  \\

\noindent In this section, we focus on sup-norm-type estimates for convolution powers of complex-valued functions on $\mathbb{Z}^d$. Our main result, Theorem \ref{thm:ConvolutionPowerEstimate}, partially extends results of \cite{randles_convolution_2015} and \cite{randles_convolution_2017}, and its proof makes use of Theorem \ref{thm:BestIntegrationFormula} and the Van der Corput lemma. A forthcoming article will present a theory of local limit theorems for complex-valued functions on $\mathbb{Z}^d$ satisfying the hypotheses of Theorem \ref{thm:ConvolutionPowerEstimate}. The Fourier transform is essential to our analysis and is defined as follows: Given $\phi\in\ell^1(\mathbb{Z}^d)$, the Fourier transform of $\phi$ is the function $\widehat{\phi}:\mathbb{R}^d\to\mathbb{C}$ defined by
\begin{equation*}
    \widehat{\phi}(\xi)=\sum_{x\in\mathbb{Z}^d}\phi(x)e^{ix\cdot\xi}
\end{equation*}
for $\xi\in\mathbb{R}^d$. As in \cite{randles_convolution_2017}, we shall focus on the subspace $\mathcal{S}_d$ of $\ell^1(\mathbb{Z}^d)$ consisting of those $\phi:\mathbb{Z}^d\to\mathbb{C}$ for which
\begin{equation*}
    \|x^\beta \phi(x)\|_1=\sum_{x\in\mathbb{Z}^d}\abs{x^\beta\phi(x)}=\sum_{x\in\mathbb{Z}^d}\abs{(x^1)^{\beta_1}(x^2)^{\beta_2}\cdots(x^d)^{\beta_d}\phi(x^1,x^2,\dots,x^d)}<\infty
\end{equation*}
for each multi-index $\beta=(\beta_1,\beta_2,\dots,\beta_d)\in\mathbb{N}^d$; we remark that $\mathcal{S}_d$ contains all finitely supported complex-valued functions on $\mathbb{Z}^d$.  It is easy to see that $\widehat{\phi}\in C^\infty(\mathbb{R}^d)$ whenever $\phi\in \mathcal{S}_d$. As discussed in \cite{thomee_stability_1965,diaconis_convolution_2014,randles_convolution_2015,randles_convolution_2017}, the asymptotic behavior of the iterative convolution powers $\phi^{(n)}$ of $\phi\in\mathcal{S}_d$ is characterized by the local behavior of $\widehat{\phi}$ near points at which $\widehat{\phi}$ is maximized in absolute value. For simplicity of our analysis, we shall focus on those $\phi\in\mathcal{S}_d$ which have been suitably normalized so that $\sup_{\xi}|\widehat{\phi}(\xi)|=1$ and, in this case, we define
\begin{equation*}
    \Omega(\phi)=\left\{\xi\in \mathbb{T}^d:\abs{\widehat{\phi}(\xi)}=1\right\}
\end{equation*}
where $\mathbb{T}^d=(-\pi,\pi]^d$. For each $\xi_0\in \Omega(\phi)$, consider $\Gamma_{\xi_0}:\mathcal{U}\to\mathbb{C}$ defined by
\begin{equation*}\Gamma_{\xi_0}(\xi)=\log\left(\frac{\widehat{\phi}(\xi+\xi_0)}{\widehat{\phi}(\xi_0)}\right)
\end{equation*}
for $\xi\in \mathcal{U}$ where $\mathcal{U}\subseteq\mathbb{R}^d$ is a convex open neighborhood of $0$ which is small enough to ensure that $z\mapsto\log(z)$, the principal branch of the logarithm, is defined and continuous on $\{\widehat{\phi}(\xi+\xi_0)/\widehat{\phi}(\xi_0):\xi\in\mathcal{U}\}$. Because $\widehat{\phi}$ is smooth, $\Gamma_{\xi_0}\in C^{\infty}(\mathcal{U})$ and so we can use Taylor's theorem to approximate $\Gamma_{\xi_0}$ near $0$. More precisely, we can write
\begin{equation}\label{eq:GammaExpansion}
    \Gamma_{\xi_0}(\xi)=i\alpha_{\xi_0}\cdot\xi -i\left(Q_{\xi_0}(\xi)+\widetilde{Q}_{\xi_0}(\xi)\right)-\left(R_{\xi_0}(\xi)+\widetilde{R}_{\xi_0}(\xi)\right)
\end{equation}
where $\alpha_{\xi_0}\in\mathbb{R}^d$, $Q_{\xi_0}$ and $R_{\xi_0}$ are real-valued polynomials which vanish at $0$ and contain no linear terms, and $\widetilde{Q}_{\xi_0}$ and $\widetilde{R}_{\xi_0}$ are real-valued smooth functions on $\mathcal{U}$ which vanish at $0$. The fact that this expansion contains no real linear part is seen necessary because $\xi_0$ is a local maximum for $|\widehat{\phi}|$. The vector $\alpha_{\xi_0}\in\mathbb{R}^d$ is said to be the \textbf{drift}\footnote{In the case that $\phi$ defines a probability measure and a $\mathbb{Z}^d$-valued random vector $X$ has this measure as its distribution, then $\alpha_{\xi_0}$ is $X$'s mean. For a precise statement and details, see Proposition 7.4 of \cite{randles_convolution_2017}.} associated to $\xi_0$. Motivated by Thom\'{e}e \cite{thomee_stability_1965}, we introduce the following definition.

\begin{definition}\label{def:Types}
Let $\phi\in\mathcal{S}_d$ with $\sup_{\xi}|\widehat{\phi}(\xi)|=1$ and, given $\xi_0\in\Omega(\phi)$, consider the expansion \eqref{eq:GammaExpansion} above.
\begin{enumerate}
    \item We say that $\xi_0$ is of \textbf{positive homogeneous type} for $\widehat{\phi}$ if $R_{\xi_0}$ is positive homogeneous and, there exists $E\in \Exp(R_{\xi_0})$ for which $Q_{\xi_0}$ is homogeneous with respect to $E$ and both $\widetilde{R}_{\xi_0}$ and $\widetilde{Q}_{\xi_0}$ are subhomogeneous with respect to $E$. In this case, we will write  $\mu_{\xi_0}=\mu_{R_{\xi_0}}$.
\item We say that $\xi_0$ is of \textbf{imaginary homogeneous type} for $\widehat{\phi}$ if $|Q_{\xi_0}|$ and $R_{\xi_0}$ are both positive homogeneous and, there exists $E\in\Exp(|Q_{\xi_0}|)$ and $k>1$ for which $R_{\xi_0}$ is homogeneous with respect to $E/k$, $\widetilde{Q}_{\xi_0}$ is strongly subhomogeneous with respect to $E$ of order $2$, and $\widetilde{R}_{\xi_0}$ is strongly subhomogeneous with respect to $E/k$ of order $1$. In this case, we write $\mu_{\xi_0}=\mu_{|Q_{\xi_0}|}$.
\end{enumerate}
In either case, $\mu_{\xi_0}$ is said to be the homogeneous order associated to $\xi_0$.
\end{definition}

\noindent In his study of approximation schemes to solutions of parabolic partial differential equations, V. Thom\'{e}e introduced the notions of points of type $\gamma$ and $\beta$, arising in local approximations of the Fourier transforms of (schemes) $\phi:\mathbb{Z}\to\mathbb{C}$, to dichotomize the stability of approximation schemes in the $\ell^\infty$ norm\cite{thomee_stability_1965}. Thom\'{e}e's definition provided a key insight which led to the complete description of the asymptotic behavior of convolution powers of finitely supported functions on $\mathbb{Z}$ given in \cite{randles_convolution_2015}. The points of positive homogeneous type and imaginary homogeneous type in the definition above parallel (and generalize) Thom\'{e}e's points of type $\gamma$ and $\beta$ (and points of type 1 and type 2 of \cite{randles_convolution_2015}), respectively.\\

\noindent In Definition 1.3 of \cite{randles_convolution_2017}, a point $\xi_0\in\Omega(\phi)$ is said to be of positive homogeneous type for $\widehat{\phi}$ provided that the expansion for $\Gamma_{\xi_0}$ is of the form
\begin{equation}\label{eq:GammaInTermsofP}
    \Gamma_{\xi_0}(\xi)=i\alpha_{\xi_0}\cdot\xi-P_{\xi_0}(\xi)-\widetilde{P}_{\xi_0}(\xi)
\end{equation}
for $\xi\in\mathcal{U}$ where $P_{\xi_0}$ is a positive homogeneous polynomial in the sense of Example \ref{exp:Polynomial} and $\widetilde{P}_{\xi_0}(\xi)=o(R_{\xi_0}(\xi))$ as $\xi\to 0$ where $R_{\xi_0}=\Re P_{\xi_0}$. To put this into context with our definition above, let's write $P_{\xi_0}(\xi)=R_{\xi_0}(\xi)+iQ_{\xi_0}(\xi)$ and $\widetilde{P}_{\xi_0}(\xi)=\widetilde{R}_{\xi_0}(\xi)+i\widetilde{Q}_{\xi_0}(\xi)$ in which case \eqref{eq:GammaExpansion} coincides with \eqref{eq:GammaInTermsofP}. If $\xi_0$ is of positive homogeneous type for $\widehat{\phi}$ in the sense of the definition above, it follows that $P_{\xi_0}$ is a complex-valued polynomial which is homogeneous with respect to $E$ (and so $\Exp(P_{\xi_0})$ contains $E\in\End(\mathbb{R}^d)$ for which $\{r^E\}$ is contracting) and $R_{\xi_0}=\Re P_{\xi_0}$ is positive definite. In view of Remark \ref{rmk:PositiveHomogeneousPolynomialsVSFunctions}, this is consistent with (and perhaps generalizes) the assumption in which $P_{\xi_0}$ is a positive homogeneous polynomial. Further, the assumption that $\widetilde{Q}_{\xi_0}$ and $\widetilde{R}_{\xi_0}$ are subhomogeneous with respect to $E$ guarantees that $\widetilde{P}_{\xi_0}(\xi)=o(R_{\xi_0}(\xi))$ as $\xi\to 0$ by virtue of Proposition \ref{prop:Subhomequivtolittleoh}. With these two observations, we see that our definition, which is stated in terms of subhomogeneity, is consistent with that of \cite{randles_convolution_2017}. \\

\noindent The essential difference between the cases in Definition \ref{def:Types} 
concerns the nature of the dominant (at low order) term in the expansion. When $\xi_0$ is of positive homogeneous type for $\widehat{\phi}$, the dominant term $P_{\xi_0}$ contains the real-valued positive definite polynomial $R_{\xi_0}$. In this case, local limit theorems for $\phi^{(n)}(x)$ contains attractors/approximants of the form
\begin{equation*}
    H^n_{P_{\xi_0}}(x)=\frac{1}{(2\pi)^d}\int_{\mathbb{R}^d}e^{-nP_{\xi_0}(\xi)-ix\cdot\xi}\,d\xi
\end{equation*}
which can be seen, for example, in Theorem 1.5 of \cite{randles_convolution_2017}. These are necessarily Schwartz functions and appear as fundamental solutions to the higher-order partial differential equations discussed in \cite{randles_positive-homogeneous_2017}. When $\xi_0$ is of imaginary homogeneous type for $\widehat{\phi}$, the dominant term in the expansion is the purely imaginary polynomial $iQ_{\xi_0}(\xi)$ and its existence (without a real counterpart) profoundly affects the asymptotic behavior of $\phi^{(n)}(x)$ (e.g., see \cite{randles_convolution_2015}). In fact, as will be shown in a forthcoming article, local limit theorems for $\phi^{(n)}(x)$ will contain approximants/attractors which are (formally) given by the oscillatory integral
\begin{equation*}
    H_{iQ_{\xi_0}}^{n}(x)=\frac{1}{(2\pi)^d}\int_{\mathbb{R}^d}e^{-inQ_{\xi_0}(\xi)-ix\cdot \xi}\,d\xi
\end{equation*}
whose convergence is a delicate matter.\\

\noindent Our theorem will be stated under the assumption that, for $\phi\in\mathcal{S}_d$ with $\sup_\xi|\widehat{\phi}(\xi)|=1$, each $\xi_0\in\Omega(\phi)$ is either of positive homogeneous type or of imaginary homogeneous type for $\widehat{\phi}$. In both cases, the positive definiteness of $R_{\xi_0}$ guarantees that each $\xi_0\in\Omega(\phi)$ is an isolated point of $\mathbb{T}^d$. Consequently, if each $\xi_0\in\Omega(\phi)$ is of positive homogeneous or imaginary homogeneous type for $\widehat{\phi}$, the set $\Omega(\phi)$ is finite and we set
\begin{equation*}
    \mu_{\phi}=\min_{\xi\in\Omega(\phi)}\mu_{\xi}.
\end{equation*}

\begin{theorem}\label{thm:ConvolutionPowerEstimate}
Let $\phi\in\mathcal{S}_d$ be such that $\sup |\widehat{\phi}|=1$ and suppose that each $\xi_0\in\Omega(\phi)$ is of positive homogeneous or imaginary homogeneous type for $\widehat{\phi}$. If $\alpha_{\xi_0}=0$ and $\mu_{\xi_0}<1$ for each $\xi_0\in\Omega(\phi)$ which is of imaginary homogeneous type for $\widehat{\phi}$, then, for any compact set $K$, there is a constant $C_K>0$ for which
\begin{equation*}
    \left|\phi^{(n)}(x)\right|\leq\frac{C_K}{n^{\mu_\phi}}
\end{equation*}
for all $x\in K$ and $n\in\mathbb{N}_+$.
\end{theorem}

\noindent The above result partially extends the analogous one-dimensional results of \cite{randles_convolution_2015} into the $d$-dimensional setting. Specifically, Theorem 3.6 of \cite{randles_convolution_2015} guarantees that, for each $\phi:\mathbb{Z}\to\mathbb{C}$ for which $\sup_{\xi\in\mathbb{T}}|\widehat{\phi}(\xi)|=1$ and whose support is finite and contains more than one point, there is a constant $C$ and a positive integer $m$ for which
\begin{equation*}
    \abs{\phi^{(n)}(x)}\leq Cn^{-1/m}
\end{equation*}
for all $x\in\mathbb{Z}$ and $n\in\mathbb{N}_+$. Under these hypotheses concerning $\phi$'s support, it follows from a basic result of complex analysis that every point $\xi_0$ of $\Omega(\phi)$ is necessarily\footnote{This is Proposition 2.2 of \cite{randles_convolution_2015}. It is easy to see that the assertion fails when $d>1$.} of positive homogeneous type or imaginary homogeneous type for $\widehat{\phi}$ with $\mu_{\xi_0}\leq 1/2<1$. In this way, we see that Theorem \ref{thm:ConvolutionPowerEstimate} partially extends Theorem 3.6 of \cite{randles_convolution_2015} in the sense that it guarantees a spatially uniform estimate over compact sets and is stated under the additional hypotheses that the drift is zero for each point of imaginary homogeneous type for $\widehat{\phi}$. Though we expect that this is not the final result on the matter, the more limited scope of Theorem \ref{thm:ConvolutionPowerEstimate} is not surprising in light of the natural complexity of $\mathbb{R}^d$ (and $\mathbb{Z}^d$).\\

\noindent Concerning the existing theory in $\mathbb{Z}^d$, Theorem \ref{thm:ConvolutionPowerEstimate} is stated under weaker hypotheses than is the analogous result in \cite{randles_convolution_2017}. Specifically, Theorem 1.4 of \cite{randles_convolution_2017} is stated under the assumption that, given $\phi\in\mathcal{S}_d$ with $\sup_{\xi}|\widehat{\phi}(\xi)|=1$, every point $\xi_0\in\Omega(\phi)$ is of positive homogeneous type for $\widehat{\phi}$ and, in this case, the theorem gives positive constants $C$ and $C'$, for which
\begin{equation}\label{eq:SupNormResultforPosHom}
    C'n^{-\mu_\phi}\leq \abs{\phi^{(n)}(x)}\leq C n^{-\mu_\phi}
\end{equation}
for all $x\in\mathbb{Z}^d$ and $n\in\mathbb{N}$. Though we have not stated it this way, our proof of Theorem \ref{thm:ConvolutionPowerEstimate} (see Lemma \ref{lem:EstPosHom}), guarantees the upper estimate in \eqref{eq:SupNormResultforPosHom} (with $K=\mathbb{Z}^d$) in the case that there are no points $\xi_0\in\Omega(\phi)$ of imaginary homogeneous type for $\widehat{\phi}$. It is the presence of points $\xi_0\in\Omega(\phi)$ of imaginary homogeneous type that makes the analysis significantly more difficult (even in one dimension) and leads to the slightly weaker conclusion. It is our belief that a uniform estimate of the type \eqref{eq:SupNormResultforPosHom} is valid when all points are either of positive homogeneous or imaginary homogeneous type for $\widehat\phi$ (perhaps still with some restriction on homogeneous order) but a resolution of such a conjecture will require further analysis and a thorough study of local limits. In Subsection \ref{subsec:Examples}, we give several examples illustrating the conclusion of Theorem \ref{thm:ConvolutionPowerEstimate}, none of which satisfy the hypotheses of Theorem 1.4 of \cite{randles_convolution_2017}.\\

\noindent Our proof of Theorem \ref{thm:ConvolutionPowerEstimate} will make use of the Fourier inversion formula
\begin{equation}\label{eq:FourierInversionConvolutionPower}
\phi^{(n)}(x)=\frac{1}{(2\pi)^d}\int_{\mathbb{T}_\phi^d}\widehat{\phi}^n(\xi)e^{-ix\cdot\xi}\,d\xi
\end{equation}
which is valid for all $n\in\mathbb{N}_+$ and $x\in\mathbb{R}^d$; here $\mathbb{T}_\phi^d=\mathbb{T}^d+\xi_\phi\subseteq\mathbb{R}^d$
is a representation of the $d$-dimensional torus chosen so that $\Omega(\phi)\subseteq \Interior(\mathbb{T}_{\phi}^d)$; this can always be arranged, i.e., some $\xi_\phi\in\mathbb{R}^d$ can be selected, because $\Omega(\phi)$ is a finite set (see Remark 3 of \cite{randles_convolution_2017}). As discussed in \cite{randles_convolution_2017}, the asymptotic behavior of $\phi^{(n)}$ is characterized by the contributions to the above integral produced by integration over neighborhoods of points $\xi_0\in\Omega(\phi).$ Specifically, we shall study integrals of the form
\begin{equation}\label{eq:LocalizedFourierInversionConvolutionPower}
\frac{1}{(2\pi)^d}\int_{\mathcal{O}_{\xi_0}}\widehat{\phi}^n(\xi)e^{-ix\cdot\xi}\,d\xi
\end{equation}
where $\mathcal{O}_{\xi_0}$ is some (small and to be determined) neighborhood of $\xi_0\in\Omega(\phi)$. When $\xi_0$ is of positive homogeneous type for $\widehat{\phi}$, such integrals are very well behaved (the integrand is dominated uniformly by $e^{-nR_{\xi_0}(\xi)/2}$, a member of the Schwartz class). When $\xi_0$ is of imaginary homogeneous type, such integrals are oscillatory in nature and therefore much more difficult to handle. Our first lemma below handles the ``easy" case in which $\xi_0$ is of positive homogeneous type for $\widehat{\phi}$. This lemma appears, essentially, as Lemma 4.3 of \cite{randles_convolution_2017}. For illustrative purposes, we have decided to present a distinct proof here which makes use of the polar coordinate integration formula in Theorem \ref{thm:BestIntegrationFormula}. 
\begin{lemma}\label{lem:EstPosHom}
Let $\xi_0\in\Omega(\phi)$ be of positive homogeneous type for $\widehat{\phi}$ with homogeneous order $\mu_{\xi_0}$. Then, there exists an open neighborhood $\mathcal{O}_{\xi_0}\subseteq\Interior(\mathbb{T}^d_\phi)$ of $\xi_0$, which can be taken as small as desired, and a constant $C=C_{\xi_0}$ for which
\begin{equation*}
    \left|\frac{1}{(2\pi)^d}\int_{\mathcal{O}_{\xi_0}}\widehat{\phi}^n(\xi)e^{-ix\cdot\xi}\,d\xi\right|\leq 
    C_{\xi_0} n^{-\mu_{\xi_0}}
\end{equation*}
for all $n\in\mathbb{N}_+$ and $x\in\mathbb{R}^d$.
\end{lemma}

\begin{proof}
For simplicity, we write $R=R_{\xi_0}$, $\widetilde{R}=\widetilde{R}_{\xi_0}$ and $\mu=\mu_{\xi_0}$. Given that $\xi_0$ is of positive homogeneous type for $\widehat{\phi}$, there is an open neighborhood $\mathcal{U}$ of $0$ for which
\begin{equation*}
    \left|\widehat{\phi}(\xi+\xi_0)\right|=\left|\widehat{\phi}(\xi_0)e^{\Gamma_{\xi_0}(\xi)}\right|=e^{-\left(R(\xi)+\widetilde{R}(\xi)\right)}
\end{equation*}
for $\xi\in \mathcal{U}$. Using the fact that $\widetilde{R}(\xi)=o(R(\xi))$ as $\xi\to 0$ in view of Proposition \ref{prop:Subhomequivtolittleoh}, we can further restrict $\mathcal{U}$ so that
\begin{equation*}
    \left|\widehat{\phi}(\xi+\xi_0)\right|\leq e^{-R(\xi)/2}
\end{equation*}
for all $\xi\in\mathcal{U}$. Take $E\in\Exp(R)$ and let $\sigma_R$ be the surface measure on  $S=\{\eta\in\mathbb{R}^d:R(\eta)=1\}$ guaranteed by Theorem \ref{thm:BestIntegrationFormula}. We fix an open neighborhood $\mathcal{O}_{\xi_0}$ of $\xi_0$ which is as small as desired and has the property that
\begin{equation*}
    \mathcal{O}:=\mathcal{O}_{\xi_0}-\xi_0\subseteq\mathcal{U}.
\end{equation*}
With this, we observe that
\begin{equation}\label{eq:WlogCenterAtZero}
\int_{\mathcal{O}_{\xi_0}}\widehat{\phi}^n(\xi)e^{-ix\cdot\xi}\,d\xi=\int_{\mathcal{O}}\widehat{\phi}^n(\xi+\xi_0)e^{-ix\cdot(\xi+\xi_0)}\,d\xi
\end{equation}
and therefore
\begin{eqnarray*}
    \left|\frac{1}{(2\pi)^d}\int_{\mathcal{O}_{\xi_0}}\widehat{\phi}^n(\xi)e^{-ix\cdot\xi}\,d\xi\right|&\leq& \frac{1}{(2\pi)^d}\int_{\mathcal{O}}\left|\widehat{\phi}^n(\xi+\xi_0)e^{-i x\cdot(\xi+\xi_0)}\right|\,d\xi\\
    &\leq& \frac{1}{(2\pi)^d}\int_{\mathcal{U}}e^{-nR(\xi)/2}\,d\xi\\
    &\leq&\frac{1}{(2\pi)^d}\int_{\mathbb{R}^d} e^{-(n/2)R(\xi)}\,d\xi
\end{eqnarray*}
for all $x\in\mathbb{R}^d$ and $n\in\mathbb{N}_+$. By virtue of Theorem \ref{thm:BestIntegrationFormula}, we have
\begin{equation*}
\int_{\mathbb{R}^d}e^{-(n/2)R(\xi)}\,d\xi=\int_S \int_0^\infty e^{-(n/2)r}r^{\mu-1}\,dr\,\sigma_R(d\eta)=\int_S \frac{2^\mu\Gamma(\mu)}{n^{\mu}}\,\sigma_R(d\eta)=2^\mu\Gamma(\mu)\sigma_R(S)n^{-\mu}
\end{equation*}
where $\Gamma$ denotes the Gamma function. Consequently,
\begin{equation*}
    \left|\frac{1}{(2\pi)^d}\int_{\mathcal{O}_{\xi_0}}\widehat{\phi}^n(\xi)e^{-ix\cdot\xi}\,d\xi\right|\leq C n^{-\mu}
\end{equation*}
for all $x\in\mathbb{R}^d$ and $n\in\mathbb{N}_+$ where $C=2^\mu \Gamma(\mu)\sigma_R(S)/(2\pi)^d.$
\end{proof}
\noindent We shall now focus on the case in which $\xi_0\in\Omega(\phi)$ is of imaginary homogeneous type for $\widehat{\phi}$. As discussed above,  \eqref{eq:LocalizedFourierInversionConvolutionPower} is oscillatory in nature; this is due to the fact that the ``principal" behavior of $\Gamma_{\xi_0}(\xi)$, for small $\xi$, is characterized by the purely imaginary polynomial $iQ_{\xi_0}$. Our main estimate is presented in Lemma \ref{lem:EstImagHom} and its proof makes use of \eqref{eq:BestIntegrationFormula} and the following version of the Van der Corput lemma.

\begin{proposition}\label{prop:VanderCorput}
Let $g\in C^1([a,b])$ be complex-valued and $f\in C^2([a,b])$ be real-valued and such that $f''(x)\neq 0$ for all $x\in [a,b]$. Then
\begin{equation*}
\abs{\int_a^b e^{if(x)}g(x)\,dx}\leq \min\left\{\frac{4}{\lambda_1},\frac{8}{\sqrt{\lambda_2}}\right\}\left(\|g\|_{L^\infty[a,b]}+\|g'\|_{L^1[a,b]}\right).
\end{equation*}
where $\lambda_1=\inf_{x\in[a,b]}\abs{f'(x)}$ and $\lambda_2=\inf_{x\in [a,b]}\abs{f''(x)}$.
\end{proposition}

\noindent For a proof of the above proposition, we refer the reader to Chapter 8 of \cite{stein_harmonic_1993} or Section 3 of \cite{randles_convolution_2015} (see Lemma 3.4 therein).  To effectively make use of the proposition above to estimate \eqref{eq:LocalizedFourierInversionConvolutionPower} in the case that $\xi_0$ is of imaginary homogeneous type, we first treat two preliminary lemmas. 

\begin{lemma}\label{lem:PhaseDerivativeEstimate}
Let $Q:\mathbb{R}^d\to\mathbb{R}$ be a continuous function for which $\abs{Q}$ is positive homogeneous with $\mu:=\mu_{\abs{Q}}<1$. Given a compact subset $S$ of $\mathbb{R}^d$ for which $0\notin S$, set
\begin{equation*}
    \rho=\inf_{\eta\in S}|Q(\eta)|/3>0.
\end{equation*}
For an open neighborhood $\mathcal{O}$ of $0$ in $\mathbb{R}^d$, suppose that $\widetilde{Q}:\mathcal{O}\to\mathbb{R}$ is a twice continuously differentiable function which is strongly subhomogeneous with respect to $E$ of order $2$, set $F=E/\mu$, and define
\begin{eqnarray*}
f_{n,\eta,x}(\theta)&=&-nQ(\theta^F\eta)-n\widetilde{Q}(\theta^F\eta)-x\cdot \theta^F\eta\\
&=&-n\theta^{1/\mu}Q(\eta)-n\widetilde{Q}(\theta^{F}\eta)-x\cdot \theta^{F}\eta
\end{eqnarray*}
for $n\in\mathbb{N}_+$, $\eta\in S$, $x\in\mathbb{R}^d$ and $\theta>0$ sufficiently small so that $\theta^F\eta\in\mathcal{O}$. Then, given any compact set $K$, there is a $\delta>0$ for which $\partial_\theta^2 f_{n,x,\eta}(\theta)\neq 0$ and
\begin{equation*}
    |\partial_\theta f_{n,x,\eta}(\theta)|\geq \frac{\rho}{\mu} n^{\mu}
\end{equation*}
for all $n\in\mathbb{N}_+$, $\eta\in S$,  $x\in K$, and $\theta>0$ for which $n^{-\mu}\leq \theta\leq \delta^\mu$.
\end{lemma}

\begin{proof}
Let $E$ and $S$ be as in the statement of the lemma and write $f=f_{n,\eta,x}$. Because ${\theta^F}$ is contracting, let $\delta_1>0$ be such that
\begin{equation}\label{eq:PhaseDerivativeEstimate1}
    \abs{x\cdot \theta^F F \eta}\leq \frac{\rho}{\mu}\hspace{.5cm}\mbox{and}\hspace{.5cm}\abs{x\cdot \theta^F (F-I)F\eta}\leq \frac{\rho}{\mu}\left(\frac{1}{\mu}-1\right)
\end{equation}
for all $0<\theta<\delta_1^\mu$, $x\in K$ and $\eta\in S$. By virtue of Proposition \ref{prop:2StronglySubhomogeneous}, there exists $\delta_2>0$ such that
\begin{equation}\label{eq:PhaseDerivativeEstimate2}
    \abs{\theta\partial_\theta \widetilde{Q}(\theta^F\eta)}
    \leq\frac{\rho}{\mu}\theta^{1/\mu}
    \hspace{0.5cm}\mbox{and}\hspace{0.5cm}
    \abs{\theta^2\partial_{\theta}^2\widetilde{Q}(\theta^F\eta)}
    \leq\frac{\rho}{\mu}\left(\frac{1}{\mu}-1\right)\theta^{1/\mu}
\end{equation}
for all $0<\theta<\delta_2^{\mu}$ and $\eta\in S$. Set $\delta=\min\{\delta_1,\delta_2\}$. By virtue \eqref{eq:PhaseDerivativeEstimate1} and \eqref{eq:PhaseDerivativeEstimate2}, we have
\begin{eqnarray*}
    \abs{\theta^2\partial_{\theta}^2 f(\theta)}&=&\theta^2\abs{n\partial_\theta^2\left( \theta^{1/\mu}Q(\eta)+\widetilde{Q}(\theta^F\eta)\right)+\partial_\theta^2\left(x\cdot \theta^F\eta\right)}\\
    &=&\abs{\frac{n}{\mu}\left(\frac{1}{\mu}-1\right)\theta^{1/\mu}Q(\eta)+n\theta^2\partial_\theta^2\widetilde{Q}(\theta^F\eta)+\left(x\cdot \theta^F(F-I)F\eta\right)}\\
    &\geq& \frac{n}{\mu}\left(\frac{1}{\mu}-1\right)\theta^{1/\mu}\abs{Q(\eta)}-n\abs{\theta^2\partial_\theta^2\widetilde{Q}(\theta^F\eta)}-\abs{x\cdot\theta^F(F-I)F\eta}\\
    &\geq&\frac{3\rho n}{\mu}\left(\frac{1}{\mu}-1\right)\theta^{1/\mu}-\frac{\rho n}{\mu}\left(\frac{1}{\mu}-1\right)\theta^{1/\mu}-\frac{\rho}{\mu}\left(\frac{1}{\mu}-1\right)\\
    &\geq &\frac{\rho}{\mu}\left(\frac{1}{\mu}-1\right)\left(2n\theta^{1/\mu}-1\right)
\end{eqnarray*}
for all $n\in\mathbb{N}_+$, $\eta\in S$, $x\in K$ and $0<\theta<\delta^\mu$. Given that $\theta\mapsto \theta^{1/\mu}$ is increasing, it follows that
\begin{equation*}
    \abs{\theta^2\partial_\theta^2 f(\theta)}\geq \frac{\rho}{\mu}\left(\frac{1}{\mu}-1\right)(2n\theta^{1/\mu}-1)\geq \frac{\rho}{\mu}\left(\frac{1}{\mu}-1\right)(2n (n^{-\mu})^{1/\mu}-1)=\frac{\rho}{\mu}\left(\frac{1}{\mu}-1\right)>0
\end{equation*}
and, in particular, $\partial_\theta^2 f(\theta)\neq 0$ for all $n\in\mathbb{N}_+$, $\eta\in S$, $x\in K$ and $\theta>0$ for which $n^{-\mu}\leq\theta\leq\delta^{\mu}$. By another appeal to \eqref{eq:PhaseDerivativeEstimate1} and \eqref{eq:PhaseDerivativeEstimate2}, we find
\begin{eqnarray*}
    \abs{\partial_\theta f(\theta)} &=& \abs{n\partial_\theta \left(\theta^{1/\mu}Q(\eta)+\widetilde{Q}(\theta^F\eta)\right)+\partial_\theta\left(x\cdot \theta^F \eta\right)}\\
    &=&\abs{\frac{n}{\mu}\theta^{1/\mu-1}Q(\eta)+n\partial_\theta\widetilde{Q}(\theta^F\eta)+\theta^{-1}\left(x\cdot\theta^F F\eta\right)}\\
    &\geq &\frac{3\rho n}{\mu}\theta^{1/\mu-1}-n\theta^{-1}\abs{\theta\partial_\theta \widetilde{Q}(\theta^F\eta)}-\theta^{-1}\abs{x\cdot\theta^F F\eta}\\
    &\geq& \frac{3\rho n}{\mu}\theta^{1/\mu-1}-\frac{\rho n}{\mu}\theta^{1/\mu-1}-\frac{\rho}{\mu}\theta^{-1}\\
    &\geq &\frac{\rho}{\mu}\left(2n\theta^{1/\mu-1}-\theta^{-1}\right)
\end{eqnarray*}
for all $n\in\mathbb{N}_+$, $\eta\in S$, $x\in K$ and $0<\theta<\delta^{\mu}$. Given our supposition that $\mu<1$, $\theta\mapsto\left( 2n\theta^{1/\mu-1}-\theta^{-1}\right)$ is increasing for $\theta>0$ and therefore
\begin{equation*}
 \abs{\partial_\theta f(\theta)}\geq\frac{\rho}{\mu}\left(2n\theta^{1/\mu-1}-\theta^{-1}\right)\geq \frac{\rho}{\mu}\left(2n(n^{-\mu})^{1/\mu-1}-(n^{-\mu})^{-1}\right)=\frac{\rho}{\mu}n^{\mu}
\end{equation*}
for all for all $n\in\mathbb{N}_+$, $\eta\in S$, $x\in K$ and $\theta>0$ for which $n^{-\mu}\leq\theta\leq\delta^{\mu}$, as was asserted.
\end{proof}

\begin{lemma}\label{lem:AmplitudeSobolevEstimates}
Let $R$ be a positive homogeneous function with $G\in\Exp(R)$, $\widetilde{R}:\mathcal{O}\to\mathbb{R}$ be once continuously differentiable on a neighborhood $\mathcal{O}$ of $0$ which has $\widetilde{R}(0)=0$ and is strongly subhomogeneous with respect to $G$ of order $1$, and let $k$ and $\mu$ be positive real numbers. Set $F=(k/\mu) G$ and, for each $\eta\in S=\{\eta:R(\eta)=1\}$ and $n\in\mathbb{N}_+$, set
\begin{equation*}
    g_{n,\eta}(\theta)=e^{-n\left(R\left(\theta^F\eta\right)+\widetilde{R}\left(\theta^F\eta\right)\right)}
\end{equation*}
for $\theta>0$ which is sufficiently small so that $\theta^F\eta\in\mathcal{O}$. Then, for each $\beta>1$, there is $\delta>0$ for which 
\begin{equation*}
    \|g_{n,\eta}\|_{L^\infty[\theta_1,\theta_2]}\leq 1
\end{equation*}
and
\begin{equation*}
    \|\partial_\theta g_{n,\eta}\|_{L^1[\theta_1,\theta_2]}\leq \beta
\end{equation*}
uniformly for $\eta\in S$, $n\in\mathbb{N}$ and $0<\theta_1\leq\theta_2\leq \delta^{\mu}$.
\end{lemma}
\begin{proof}
By virtue of the strong subhomogeneity of $\widetilde{R}$, Proposition \ref{prop:supersub_implies_sub}, and the fact that $r^G$ is a contracting group, we may choose $\delta>0$ for which 
\begin{equation}\label{eq:AmplitudeSobolevEstimates1}
    R(r^G\eta)+\widetilde{R}\left(r^G\eta\right)=r+\widetilde{R}\left(r^G\eta\right)\geq (1-\epsilon)r>0
\end{equation}
and
\begin{equation}\label{eq:AmplitudeSobolevEstimates2}
    \abs{\partial_r\left(R(r^G\eta)+\widetilde{R}(r^G\eta)\right)}\leq 1+\epsilon
\end{equation}
for all $0<r<\delta^k$ and $\eta\in S$ where
\begin{equation*}
    \epsilon=\frac{\beta-1}{\beta+1}\in(0,1).
\end{equation*}
In view of \eqref{eq:AmplitudeSobolevEstimates1}, for any $0<\theta_1\leq\theta_2<\delta^{\mu}$, $\eta\in S$ and $n\in\mathbb{N}_+$,
\begin{equation*}
\|g_{n,\eta}\|_{L^\infty[\theta_1,\theta_2]}\leq\sup_{0<\theta\leq\delta^{\mu}}\abs{g_{n,\eta}(\theta)}=\sup_{0<r\leq\delta^k}\abs{g_{n,\eta}(r^{\mu/k})}\leq \sup_{0<r\leq \delta^k}e^{-nr(1-\epsilon)}= 1
\end{equation*}
where we have used the fact that $(r^{\mu/k})^F=r^{(\mu/k)F}=r^G$ for $r>0$. By virtue of \eqref{eq:AmplitudeSobolevEstimates1} and \eqref{eq:AmplitudeSobolevEstimates2}, we find that
\begin{eqnarray*}
\|\partial_{\theta}g_{n,\eta}\|_{L^1[\theta_1,\theta_2]}
&=&\int_{\theta_1}^{\theta_2}\abs{\p_\theta g_{n,\eta}(\theta)}\,d\theta\\
&=&\int_{\theta_1^{k/\mu}}^{\theta_2^{k/\mu}}\abs{\p_r \lp g_{n,\eta}(r^{\mu/k})\rp}\,dr\\
&=&\int_{\theta_1^{k/\mu}}^{\theta_2^{k/\mu}}\abs{\p_r \left(e^{-n\left(R(r^G\eta)+\widetilde{R}(r^G\eta)\right)}\right) }\,dr\\
&=&\int_{\theta_1^{k/\mu}}^{\theta_2^{k/\mu}}n\abs{\partial_r\left(R(r^G\eta)+\widetilde{R}(r^G\eta)\right)}\abs{e^{-nR(r^G\eta)+\widetilde{R}(r^G\eta)}}\,dr\\
&\leq&\int_{\theta_1^{k/\mu}}^{\theta_2^{k/\mu}}n(1+\epsilon)e^{-n(1-\epsilon)r}\,dr\\
&\leq &\frac{1+\epsilon}{1-\epsilon}\int_0^\infty e^{-r}\,dr=\beta
\end{eqnarray*}
for all $n\in\mathbb{N}_+$, $\eta\in S$ and $0<\theta_1\leq\theta_2\leq\delta^{\mu}$, as desired.
\end{proof}

\begin{lemma}\label{lem:EstImagHom}
Suppose that $\xi_0$ is of imaginary homogeneous type for $\widehat{\phi}$ with associated drift $\alpha_{\xi_0}$ and homogeneous order $\mu_{\xi_0}$. If $\alpha_{\xi_0}=0$ and $\mu_{\xi_0}<1$, then, for each compact set $K$, there is an open neighborhood $\mathcal{O}_{\xi_0}\subseteq\Interior(\mathbb{T}_\phi^d)$ of $\xi_0$, which can be taken as small as desired, and a constant $C_{\xi_0}$ for which
\begin{equation*}
    \abs{\f{1}{(2\pi)^d}\int_{\mathcal{O}_{\xi_0}}\widehat{\phi}^n(\xi)e^{-ix\cdot\xi}\,d\xi}\leq C_{\xi_0}n^{-\mu_{\xi_0}}
\end{equation*}
for all $x\in K$ and $n\in\mathbb{N}_+$.
\end{lemma}
\begin{proof}
For simplicity of notation, we will write $Q=Q_{\xi_0}$, $\widetilde{Q}=\widetilde{Q}_{\xi_0}$, $R=R_{\xi_0}$, $\widetilde{R}=\widetilde{R}_{\xi_0}$ and $\mu=\mu_{\xi_0}$. We fix a compact set $K\subseteq\mathbb{R}^d$ and let $E$ and $k$ as given in Definition \ref{def:Types}. In studying the proof of Lemma \ref{lem:EstPosHom} and \eqref{eq:WlogCenterAtZero}, in particular, it is evident that we may assume $\xi_0=0$ and $\widehat{\phi}(0)=1$ without loss of generality. Given that $G:=E/k\in\Exp(R)$, set
\begin{equation*}
    F=(k/\mu)G=E/\mu.
\end{equation*} Using the positive homogeneous structure of $R$, let $\sigma_R$ be the measure on $S=\{\eta\in \mathbb{R}^d:R(\eta)=1\}$ as guaranteed by Theorem \ref{thm:BestIntegrationFormula}. By setting
\begin{equation*}
    \rho=\inf_{\eta\in S}|Q(\eta)|/3,
\end{equation*}
an appeal to Lemma \ref{lem:PhaseDerivativeEstimate} guarantees a $\delta_1>0$ for which 
\begin{equation}\label{eq:EstImagHom1}
    \abs{\partial_{\theta}f_{n,\eta,x}(\theta)}\geq \frac{\rho}{\mu}n^{\mu}\hspace{1cm}\mbox{and}\hspace{1cm}\partial_\theta^2 f_{n,\eta,x}(\theta)\neq 0
\end{equation} for all $n\in\mathbb{N}_+$, $\eta\in S$, $x\in K$ and $\theta>0$ for which $n^{-\mu}\leq \theta\leq \delta_1^\mu$. An appeal to Lemma \ref{lem:AmplitudeSobolevEstimates} guarantees $\delta_2>0$ for which
\begin{equation}\label{eq:EstImagHom2}
    \|g_{n,\eta}\|_{L^\infty[\theta_1,\theta_2]}
    +
    \|\p_\theta g_{n,\eta} \|_{ L^1[\theta_1,\theta_2]}
    \leq 3
\end{equation}
for all $n\in\mathbb{N}_+$, $\eta\in S$ and $0<\theta_1\leq\theta_2\leq\delta_2^{\mu}$. We set $\mathcal{O}=\{\eta\in\mathbb{R}^d:R(\eta)<\delta^k\}$
where $0<\delta\leq \min\{\delta_1,\delta_2\}$ is as small as desired; this is necessarily an open neighborhood of $0$. We have
\begin{eqnarray*}
    \int_{\mathcal{O}}\widehat{\phi}^n(\xi)e^{-ix\cdot\xi}\,d\xi
    &=&
    \int_S\int_0^{\delta^{k}}\widehat{\phi}^n(r^G\eta)e^{-ix\cdot r^G\eta}r^{\mu/k-1}\,dr \sigma_R(d\eta)\\
    &=&
    \frac{k}{\mu}\int_S \int_0^{\delta^{\mu}} \widehat{\phi}^n(\theta^{F} \eta) e^{-i x\cdot\theta^F \eta}  \,d\theta \,\sigma_R(d\eta)\\
    &=&
    \frac{k}{\mu}\int_S I_{n,x}(\eta)\,\sigma_R(d\eta)
\end{eqnarray*}
where we have made the change of variables $\theta=r^{\mu/ k}$ and set
\begin{eqnarray*}
    I_{n,x}(\eta)&=&\int_0^{\delta^{\mu}}\widehat{\phi}^n(\theta^F\eta)e^{-ix\cdot\theta^F\eta}\,d\theta.
\end{eqnarray*}
For each $n\in\mathbb{N}_+$, $\eta\in S$, and $x\in\mathbb{R}^d$, we have
\begin{eqnarray*}
\abs{I_{n,x}(\eta)}
&\leq & 
\abs{\int_{n^{-\mu}}^{\delta^{\mu}}\widehat{\phi}^n(\theta^F\eta)e^{-ix\cdot\theta^F\eta}\,d\theta} +\int_{0}^{n^{-\mu}}\abs{\widehat{\phi}^n(\theta^F\eta)}\,d\theta\\
&\leq& \abs{\int_{n^{-\mu}}^{\delta^\mu} e^{-i\left(nQ\left(\theta^F\eta\right)+n\widetilde{Q}\left(\theta^F\eta\right)+x\cdot\theta^F\eta\right)}e^{-n\left(R\left(\theta^F\eta\right)+\widetilde{R}\left(\theta^F\eta\right)\right)}\,d\theta}+n^{-\mu}\\
& &\hspace{1cm}=\abs{\int_{n^{-\mu}}^{\delta^\mu} e^{i f_{n,\eta,x}(\theta) } g_{n,\eta}(\theta)\,d\theta} 
+ n^{-\mu}.
\end{eqnarray*}
In view of \eqref{eq:EstImagHom1} and \eqref{eq:EstImagHom2}, an appeal to Proposition \ref{prop:VanderCorput} guarantees that, for any $n\in\mathbb{N}_+$, $\eta\in S$ and $x\in K$,
\begin{eqnarray*}
 \hspace{-1cm}\abs{\int_{n^{-\mu}}^{\delta^\mu}e^{if_{n,\eta,x}(\theta)}g_{n,\eta}(\theta)\,d\theta}
    &\leq& 
    4
    \frac{ 
    \|g_{n,\eta}\|_{L^\infty[n^{-\mu},\delta^{\mu}]}
    +
    \|\partial_{\theta}g_{n,\eta}\|_{L^1[n^{-\mu},\delta^{\mu}]}
    }{
    \inf_{n^{-\mu}\leq\theta\leq \delta^{\delta}}|\partial_{\theta}f_{n,x,\eta}(\theta)|
    }\\
    &\leq& 4\frac{3}{(\rho/\mu) n^{\mu}}=\frac{12\mu}{\rho}n^{-\mu}
\end{eqnarray*}
and so
\begin{equation*}
    \abs{I_{n,x}(\eta)}\leq \left(\frac{12\mu}{\rho}\right)n^{-\mu}+n^{-\mu}\leq \left(\frac{12\mu}{\rho}+1\right)n^{-\mu}.
\end{equation*}
Thus, for all $n\in\mathbb{N}_+$ and $x\in K$,
\begin{eqnarray*}
\abs{\f{1}{(2\pi)^d}\int_{\mathcal{O}}\widehat{\phi}^n(\xi)e^{-i\xi\cdot x}\,d\xi}
&=&\frac{1}{(2\pi)^d}\f{k}{\mu}\abs{\int_S I_{n,x}(\eta)\,\sigma_R(d\eta)} \\
&\leq& \frac{1}{(2\pi)^d}\f{k}{\mu}\int_S \abs{I_{n,x}(\eta)}\,\sigma_R(d\eta)\\
&\leq& Cn^{-\mu}
\end{eqnarray*}
where
\begin{equation*}
    C=\f{1}{(2\pi)^d} \f{k}{\mu} \left(\frac{12\mu}{\rho}+1\right)\sigma_R(S).
\end{equation*}
\end{proof}

\begin{proof}[Proof of Theorem \ref{thm:ConvolutionPowerEstimate}]
Let $K\subseteq\mathbb{R}^d$ be a compact set. As we discussed in the paragraph preceding the theorem, the set $\Omega(\phi)$ is finite and so we may write
\begin{equation*}
    \Omega(\phi)=\{\xi_1,\xi_2,\dots,\xi_N,\xi_{N+1},\xi_{N+2},\dots,\xi_M\}
\end{equation*}
where our labeling assumes that the points $\xi_1,\xi_2,\dots,\xi_N$ are of imaginary homogeneous type for $\widehat{\phi}$ and the points $\xi_{N+1},\xi_{N+2},\dots,\xi_M$ are of positive homogeneous type for $\widehat{\phi}$. In view of the theorem's hypotheses, for each $j=1,2,\dots,N$, the point $\xi_j$, which is of imaginary homogeneous type for $\widehat{\phi}$, has drift $\alpha_{\xi_j}=0$ and homogeneous order $\mu_j:=\mu_{\xi_j}<1$. Thus, for each $j=1,2,\dots,N$, an appeal to Lemma \ref{lem:EstImagHom} guarantees an open neighborhood $\mathcal{O}_j=\mathcal{O}_{\xi_j}\subseteq\Interior(\mathbb{T}_\phi^d)$ of $\xi_j$ and a constant $C_j=C_{\xi_j}$ for which
\begin{equation}\label{eq:ConvolutionPowerEstimate1}
    \abs{\frac{1}{(2\pi)^d}\int_{\mathcal{O}_j}\widehat{\phi}^n(\xi)e^{-ix\cdot\xi}\,d\xi}\leq C_j n^{-\mu_j}
\end{equation}
for all $n\in\mathbb{N}_+$ and $x\in K$. For each $j=N+1,N+2,\dots M$, an appeal to Lemma \ref{lem:EstPosHom} guarantees an open neighborhood $\mathcal{O}_j=\mathcal{O}_{\xi_j}\subseteq\Interior(\mathbb{T}_\phi^d)$ of $\xi_j$ and a constant $C_j=C_{\xi_j}$ for which 
\begin{equation}\label{eq:ConvolutionPowerEstimate2}
        \abs{\frac{1}{(2\pi)^d}\int_{\mathcal{O}_j}\widehat{\phi}^n(\xi)e^{-ix\cdot\xi}\,d\xi}\leq C_jn^{-\mu_j}
\end{equation}
for all $n\in\mathbb{N}_+$ and $x\in\mathbb{R}^d$ where $\mu_j:=\mu_{\xi_j}$ is the homogeneous order associated to $\xi_j$. As guaranteed by the lemmas, let us take this collection of open sets $\mathcal{O}_1,\mathcal{O}_2,\dots,\mathcal{O}_M\subseteq\mathbb{T}_{\phi}^d$ to be mutually disjoint and define
\begin{equation}
    \mathcal{G}=\mathbb{T}_{\phi}^d\setminus\left(\bigcup_{j=1}^M \mathcal{O}_j\right).
\end{equation}
Given that $\mathcal{G}$ is a closed set which contains no elements of $\Omega(\phi)$,
\begin{equation*}
s:=\sup_{\xi\in\mathcal{G}}\abs{\widehat{\phi}(\xi)}<1.
\end{equation*}
By virtue of \eqref{eq:FourierInversionConvolutionPower}, \eqref{eq:ConvolutionPowerEstimate1}, \eqref{eq:ConvolutionPowerEstimate2}, and the disjointness of the collection $\mathcal{O}_1,\mathcal{O}_2,\dots,\mathcal{O}_M$, we have
\begin{eqnarray}\label{eq:ConvolutionPowerEstimate3}\nonumber
    \abs{\phi^{(n)}(x)}
    &=&\abs{\lb \sum_{j=1}^M\frac{1}{(2\pi)^d}\int_{\mathcal{O}_j}\widehat{\phi}^n(\xi)e^{-x\cdot\xi}\,d\xi \rb
    +\frac{1}{(2\pi)^d}\int_{\mathcal{G}}\widehat{\phi}^n(\xi)e^{-x\cdot\xi}\,d\xi}\\\nonumber
    &\leq&\sum_{j=1}^M\abs{\frac{1}{(2\pi)^d}\int_{\mathcal{O}_j}\widehat{\phi}^n(\xi)e^{-x\cdot\xi}\,d\xi}+\abs{\frac{1}{(2\pi)^d}\int_{\mathcal{G}}\widehat{\phi}^n(\xi)e^{-x\cdot\xi}\,d\xi}\\
    &\leq&\sum_{j=1}^M C_jn^{-\mu_j}+s^n
\end{eqnarray}
for all $n\in\mathbb{N}_+$ and $x\in K$. Upon noting that $\mu_\phi=\min\{\mu_1,\mu_2,\dots,\mu_M\}$, we have
\begin{equation*}
    n^{-\mu_j}=O(n^{-\mu_\phi})
\end{equation*}
as $n\to\infty$ for each $j=1,2,\dots M$. Also, because $s<1$, $s^n=o(n^{-\mu_\phi})$ as $n\to \infty$. With these two observations, the theorem follows immediately from \eqref{eq:ConvolutionPowerEstimate3}.
\end{proof}

\subsection{Examples}\label{subsec:Examples}

%%%%%%%%%%%%%%%%%%%%%%%%%%%%%%%

In this subsection, we give a number of examples illustrating the results of Theorem \ref{thm:ConvolutionPowerEstimate}, all of which are beyond the scope of validity of the results of \cite{randles_convolution_2017}. First, we treat a useful proposition which gives sufficient conditions for a point $\xi_0\in\Omega(\phi)$ to be of positive homogeneous or imaginary homogeneous type for $\hat{\phi}$ in terms of the Taylor expansion for $\Gamma_{\xi_0}$.

\begin{proposition}\label{prop:ExpandGamma}
Let $\phi\in\mathcal{S}_d$ with $\sup_{\xi}|\widehat{\phi}(\xi)|=1$ and let $\xi_0\in\Omega(\phi)$. Suppose that there exists $\mathbf{m}\in \mathbb{N}^d_+$ and some $k \geq 1$ such that the Taylor expansion of $\Gamma_{\xi_0} : \mathcal{U}\to\mathbb{C}$ centered at $0$ is a series of the form
\begin{eqnarray}\label{eq:SemiEllipticImaginaryExpansion}\nonumber
    \Gamma_{\xi_0}(\xi) 
    &=& i\al_{\xi_0} \cdot \xi - i \left( \sum_{\abs{\be : 2\mathbf{m}} \geq 1} A_\be \xi^\be\right) - \sum_{\abs{\be : 2\mathbf{m}} \geq k} B_\be \xi^\be \\ \nonumber
    &=& i\al_{\xi_0} \cdot \xi - i \lp \sum_{\abs{\be : 2\mathbf{m}} = 1} A_\be \xi^\be + \sum_{\abs{\be : 2\mathbf{m}} > 1} A_\be \xi^\be\rp 
    - \lp \sum_{\abs{\be : 2\mathbf{m}} = k} B_\be \xi^\be + \sum_{\abs{\be : 2\mathbf{m}} > k} B_\be \xi^\be \rp \\
    &=&  i\al_{\xi_0} \cdot \xi - i\lp Q_{\xi_0}(\xi) + \widetilde{Q}_{\xi_0}(\xi)\rp - \lp R_{\xi_0}(\xi) + \widetilde{R}_{\xi_0}(\xi) \rp,
\end{eqnarray}
where $\al_{\xi_0} \in \mathbb{R}^d$;   $Q_{\xi_0}$ and $R_{\xi_0}$ are real-valued polynomials for which $R_{\xi_0}$ is positive definite; and  $\widetilde{Q}_{\xi_0},$ and $\widetilde{R}_{\xi_0}$ are real multivariate power series which are absolutely and uniformly convergent on $\mathcal{U}$. If $k=1$, then $\xi_0$ is of positive homogeneous type for $\widehat{\phi}$. If $k>1$ and $|Q_{\xi_0}|$ is positive definite, then $\xi_0$ is of imaginary homogeneous type for $\hat{\phi}$. In either case, $\xi_0$ has drift $\alpha_{\xi_0}$ and homogeneous order
\begin{equation*}
    \mu_{\xi_0}=\abs{\mathbf{1}:2\mathbf{m}}=\sum_{j=1}^d\frac{1}{2m_j}.
\end{equation*}
\end{proposition}
\noindent Before proving the proposition, we shall first take care of the following useful lemma.
\begin{lemma}
Given an open neighborhood $\mathcal{U}$ of $0$ in $\mathbb{R}^d$, suppose that $Q:\mathcal{U}\to\mathbb{C}$ is real-analytic on $\mathcal{U}$ with absolutely and uniformly convergent series expansion
\begin{equation*}
    Q(\xi)=\sum_{|\beta:\mathbf{n}|>1}A_\beta\xi^\beta
\end{equation*}
for some $\mathbf{n}\in\mathbb{N}_+^d$. Consider $E\in\End(\mathbb{R}^d)$ with standard representation $\diag(1/n_1,1/n_2,\dots,1/n_d)$. Then, for each $l\in\mathbb{N}_+$, $Q$ is strongly subhomogeneous with respect to $E$ of order $l$. 
\end{lemma}
\begin{proof}
It suffices to show that, for each, $j\in\mathbb{N}_+$, $\epsilon>0$ and compact set $K\subseteq\mathbb{R}^d$, there is a $\delta>0$ for which
\begin{equation*}
    \abs{r^j\partial_r^jQ(r^E\eta)}\leq \epsilon r
\end{equation*}
for all $0<r<\delta$ and $\eta\in K$. To this end, we fix $j$, $\epsilon$, and $K$ as above and write $Q=Q_1+Q_2$ where
\begin{equation*}
Q_1(\xi)=\sum_{1+\rho\leq |\beta:\mathbf{n}|\leq 2j+2}A_\beta\xi^\beta
\hspace{0.5cm}\mbox{and}\hspace{0.5cm}
    Q_2(\xi)=\sum_{|\beta:\mathbf{n}|> 2j+2}A_{\beta}\xi^\beta
\end{equation*}
where $\rho:=\min\{|\beta:\mathbf{n}|:A_\beta\neq 0\}-1>0$.
For each $q\geq 1$ and $l\in \mathbb{N}_+$, define
\begin{equation*}
    \mathcal{P}(q,l)=q(q-1)(q-2)\cdots (q-(l-1)).
\end{equation*}
In this notation, we observe that 
\begin{equation*}
    \partial_r^j(r^E\xi)^\beta=\partial_r^j\left(r^{|\beta:\mathbf{n}|}\xi^\beta\right)=\mathcal{P}(|\beta:\mathbf{n}|,j)r^{|\beta:\mathbf{n}|-j}\xi^\beta
\end{equation*}
for $\xi\in\mathbb{R}^d$, $r>0$ and $\beta\in\mathbb{N}^d$.
Because $Q_1$ is a polynomial and $K$ is compact, we have
\begin{equation*}
    M_1:=\sup_{\eta\in K}\lp\sum_{1+\rho\leq |\beta:\mathbf{n}|\leq 2j+2}\abs{A_\beta \mathcal{P}(|\beta:\mathbf{n}|,j)\eta^\beta}\rp<\infty.
\end{equation*}
Given that $Q$ is absolutely and uniformly convergent on $\mathcal{U}$, let $\mathcal{O}\subseteq \overline{\mathcal{O}}\subseteq\mathcal{U}$ be an open neighborhood of $0$ for which
\begin{equation*}
    M_2:=\sup_{\xi\in \mathcal{O}}\lp \sum_{|\beta:\mathbf{n}|>2j+2}\abs{A_\beta\xi^\beta}\rp<\infty.
\end{equation*}
We now specify $\delta$. First, given that $\{r^E\}$ and $\{r^{E/4}\}$ are contracting and the set $K$ is compact, we may find a $0<\delta_1\leq 1$ for which $r^E\eta$ and $r^{E/4}\eta$ belong to $\mathcal{O}$ whenever $0<r<\delta_1$ and $\eta\in K$. Also, there exists $\delta_2>0$ for which
\begin{equation}\label{eq:PermuationEst}
    \abs{\mathcal{P}(q,j)}r^{q/4}\leq 1
\end{equation}
for all $q>j$ and $0<r\leq \delta_2$; it is sufficient to take $\delta_2=e^{-4j}$.  Finally, given that $\rho>0$, let $\delta_3>0$ be such that
\begin{equation*}
    M_1 r^\rho+M_2r<\epsilon
\end{equation*}
for all $0<r<\delta_3$. Set $\delta=\min\{\delta_1,\delta_2,\delta_3\}$ and observe that
for all $\eta\in K$ and $0<r<\delta$, we have
\begin{eqnarray*}
    \abs{r^j\partial_r^jQ_1(r^E\eta)}&=&r^j\abs{\sum_{1+\rho\leq|\beta:\mathbf{n}|\leq 2j+2}A_\beta \partial_r^j\lp r^E\eta\rp^{\beta}}\\
    &\leq&r^j\sum_{1+\rho\leq|\beta:\mathbf{n}|\leq 2j+2}\abs{A_\beta\mathcal{P}(|\beta:\mathbf{n}|,j)r^{|\beta:\mathbf{n}|-j}\eta^\beta}\\
    &\leq&r^{1+\rho}\sum_{1+\rho\leq |\beta:\mathbf{n}|\leq 2j+2}\abs{A_\beta \mathcal{P}(|\beta:\mathbf{n}|,j)\eta^\beta}\\
&\leq&r  M_1r^\rho.
\end{eqnarray*}
By virtue of \eqref{eq:PermuationEst}, for each $q=|\beta:\mathbf{n}|>2j+2$, we have
\begin{eqnarray*}
    \abs{\partial_r^j\lp A_\beta(r^E\eta)^\beta\rp}&=&\abs{A_\beta}\abs{\mathcal{P}(|\beta:\mathbf{n}|,j)}r^{|\beta:\mathbf{n}|-j}\abs{\eta^\beta}\\
    &=&r^{|\beta:\mathbf{n}|/2-j}\abs{A_\beta}\abs{\mathcal{P}(|\beta:\mathbf{n}|,j)r^{|\beta:\mathbf{n}|/4}}\abs{(r^{E/4}\eta)^\beta}\\
    &\leq& r\abs{A_\beta(r^{E/4}\eta)^\beta}
\end{eqnarray*}
for all $0<r<\delta\leq\delta_2$ and $\eta\in K$. It follows that
\begin{eqnarray*}
    \abs{\partial_r^jQ_2(r^E\eta)}&=&\abs{\sum_{|\beta:\mathbf{n}|> 2j+2}\partial_r^j\lp A_\beta(r^E\eta)^\beta\rp}\\
    &\leq & \sum_{|\beta:\mathbf{n}|>2j+2}\abs{\partial_r^j\lp A_\beta(r^E\eta)^\beta\rp}\\
    &\leq&\sum_{|\beta:\mathbf{n}|>2j+2}r\abs{A_\beta(r^{E/4}\eta)^\beta}\\
    &\leq &rM_2
\end{eqnarray*}
for all $0<r<\delta$ and $\eta\in K$. Therefore, for each $0<r<\delta$ and $\eta\in K$, we have
\begin{eqnarray*}
    \abs{r^j\partial_r^jQ(r^E\eta)}&\leq&\abs{r^j\partial_r^jQ_1(r^E\eta)}+\abs{r^j\partial_r^jQ_2(r^E\eta)}\\
    &\leq& rr^\rho M_1+r^{j+1}M_2\\
    &\leq& r(M_1r^\rho+M_2r)\\
    &<&r \epsilon.
\end{eqnarray*}
\end{proof}

\begin{proof}[Proof of Proposition \ref{prop:ExpandGamma}.]
It is easy to see that $E\in \Exp(Q_{\xi_0})\cap\Exp(\abs{Q_{\xi_0}})$ and $E/k\in\Exp(R_{\xi_0})$ for $E\in\End(\mathbb{R}^d)$ with standard matrix representation 
\begin{equation*}
\diag((2m_1)^{-1}, (2m_2)^{-1},\dots, (2m_d)^{-1}).
\end{equation*}
If $k=1$, $R_{\xi_0}$ is positive homogeneous with $E\in\Exp(R_{\xi_0})\cap\Exp(Q_{\xi_0})$. By virtue of the preceding lemma (with $\mathbf{n}=2\mathbf{m}$), $\widetilde{Q}_{\xi_0}$ and $\widetilde{R}_{\xi_0}$ are strongly subhomogeneous with respect to $E$ of order $1$ and so, in view of Proposition \ref{prop:supersub_implies_sub}, both are subhomogeneous with respect to $E$. In this case, we may conclude that $\xi_0$ is of positive homogeneous type for $\widehat{\phi}$ with drift $\alpha_{\xi_0}$ and homogeneous order
\begin{equation*}
    \mu_{\xi_0}=\tr E=|\mathbf{1}:2\mathbf{m}|=\sum_{j=1}^d\frac{1}{2m_j}.
\end{equation*}
If $k>1$, our supposition guarantees that $\abs{Q_{\xi_0}}$ is positive homogeneous with respect to $E$ and $R_{\xi_0}$ is positive homogeneous with respect to $E/k$. By virtue of the preceding lemma, $\widetilde{Q}_{\xi_0}$ is strongly subhomogeneous with respect to $E$ of order $2$ and $\widetilde{R}_{\xi_0}$ is strongly subhomogeneous with respect to $E/k$ of order $1$. Consequently, $\xi_0$ is of imaginary homogeneous type for $\widehat{\phi}$ with drift $\alpha_{\xi_0}$ and homogeneous order $\mu_{\xi_0}=\tr E=|\mathbf{1}:2\mathbf{m}|$ as in the previous case.
\end{proof}
%%%%%%%%%%%%%%%%%%%%%%%%%%%%%%%%%%%%%%

\begin{example}\normalfont

Consider the function $\phi : \mathbb{Z}^2 \to \mathbb{C}$ defined by 
\begin{equation*}
    \phi(x,y) =
    \frac{1}{512}\times
    \begin{cases}
    372 - 96i &(x,y) = (0,0)\\
    56+32i &(x,y) = (\pm 1, 0)\mbox{ or }(0,\pm 1)\\
    -28-8i        &(x,y) = (\pm 2,0)\mbox{ or }(0,\pm 2)\\
    8       &(x,y) = (\pm 3,0)\mbox{ or }(0,\pm 3)\\
    -1        &(x,y) = (\pm 4,0)\mbox{ or }(0,\pm 4)\\
    0& \text{otherwise}.
    \end{cases}
\end{equation*}
It is easily verified that $\sup_{\xi}|\widehat{\phi}|=1$ and $\Omega(\phi) = \{\xi_0 \}$ where $\xi_0=(0,0)$. Since $\phi$ is finitely supported, $\Gamma_{0}=\Gamma_{\xi_0}$ is holomorphic and so its Taylor series converges absolutely and uniformly on an open neighborhood $\mathcal{U}\subseteq \mathbb{R}^2$ of $0$. By a straightforward computation, we find 
\begin{eqnarray*}
\Gamma_{0}(\xi)
&=& 
-i\lp \frac{\tau^4}{64} + \frac{\zeta^4}{64}   \rp
-i \sum_{|\beta:(4,4)|\geq 2}A_\beta \xi^\beta\\
&& 
\hspace{1.3cm}
-\lp \f{15\tau^8}{8192} - \f{\tau^4\zeta^4}{4096} + \f{15\zeta^8}{8192}   \rp 
- \sum_{|\beta:(4,4)|\geq 6}B_\beta \xi^\beta\\
&=&-i\lp Q_{0}(\xi)+\widetilde{Q}_{0}(\xi)\rp-\lp R_0(\xi)+\widetilde{R}_0(\xi)\rp
\end{eqnarray*}
where
\begin{equation*}
Q_{0}(\xi)=\sum_{|\beta:(4,4)|=1}A_\beta \xi^\beta=\frac{\tau^4}{64} + \frac{\zeta^4}{64} ,
\end{equation*}
\begin{equation*}
R_{0}(\xi)=\sum_{|\beta:(4,4)|=2}B_\beta\xi^\beta
=  \f{15\tau^8}{8192} - \f{\tau^4\zeta^4}{4096} + \f{15\zeta^8}{8192}  , 
\end{equation*}
\begin{equation*}
\widetilde{Q}_{0}(\xi)= \sum_{|\beta:(4,4)| \geq 3/2}A_\beta \xi^\beta= 
-\f{\tau^6+\zeta^6}{384}  + \f{\tau^8 + \zeta^8}{5120}+  \f{7(\tau^4\zeta^8+\tau^8\zeta^4)}{262144} +
\cdots, 
\end{equation*}
and
\begin{eqnarray*}
\widetilde{R}_{0}(\xi)=\sum_{|\beta:(4,4)|\geq 5/2}B_\beta \xi^\beta 
=\f{(\tau^4\zeta^6 + \tau^6\zeta^4)}{24576} - \f{\tau^6\zeta^6}{147456} - \f{(\tau^4\zeta^8 + \tau^8\zeta^4)}{327680}
\cdots, 
\end{eqnarray*}
for $\xi=(\tau,\zeta)\in\mathcal{U}$. Observe that this expansion is of the form \eqref{eq:SemiEllipticImaginaryExpansion} with $\alpha_0=(0,0)$, $\mathbf{m}=(2,2)$, and $k=2$. It is readily verified that $|Q_0|=Q_0$ and $R_0$ are positive definite and, by virtue of Proposition \ref{prop:ExpandGamma}, we conclude that $\xi_0=0$ is of imaginary homogeneous type for $\widehat\phi$ with drift $\alpha_0=0$ and homogeneous order
\begin{equation*}
    \mu_{\phi}=\mu_0=|\mathbf{1}:2\mathbf{m}|=\frac{1}{4}+\frac{1}{4}=\f{1}{2}.
\end{equation*}
By an appeal to Theorem \ref{thm:ConvolutionPowerEstimate} we obtain, to each compact set $K\subseteq\mathbb{R}^2$, a positive constant $C$ for which
\begin{equation}\label{eq:SecondOrderExampleDecay}
    |\phi^{(n)}(x,y)|\leq \frac{C}{n^{\mu_\phi}}=\frac{C}{n^{1/2}}
\end{equation}
for all $n\in\mathbb{N}_+$ and $(x,y)\in K$. To illustrate this result, we consider the compact set $K = [-700, 700] \times [-700, 700]$ and define $f(n)=f_{\phi,K}(n)=\max_{(x,y)\in K}\abs{\phi^{(n)}(x,y)}$. Figure \ref{fig:Conv_Pwr_0_new} illustrates this result by capturing the decay of $f(n)=f_{\phi,K}(n)=\max_{(x,y)\in K}\abs{\phi^{(n)}(x,y)}$ relative to that of $n^{-\mu_\phi}$. Also, Figure \ref{fig:Conv_Pwr_00_new} illustrates the graph of $\Re \phi^{(n)}(x,y)$ for $(x,y)\in K$ and $n=200$ and $n=1000$.

\begin{figure}[!htb]
    \begin{subfigure}{0.49\textwidth}
    \centering
    \includegraphics[scale=0.58]{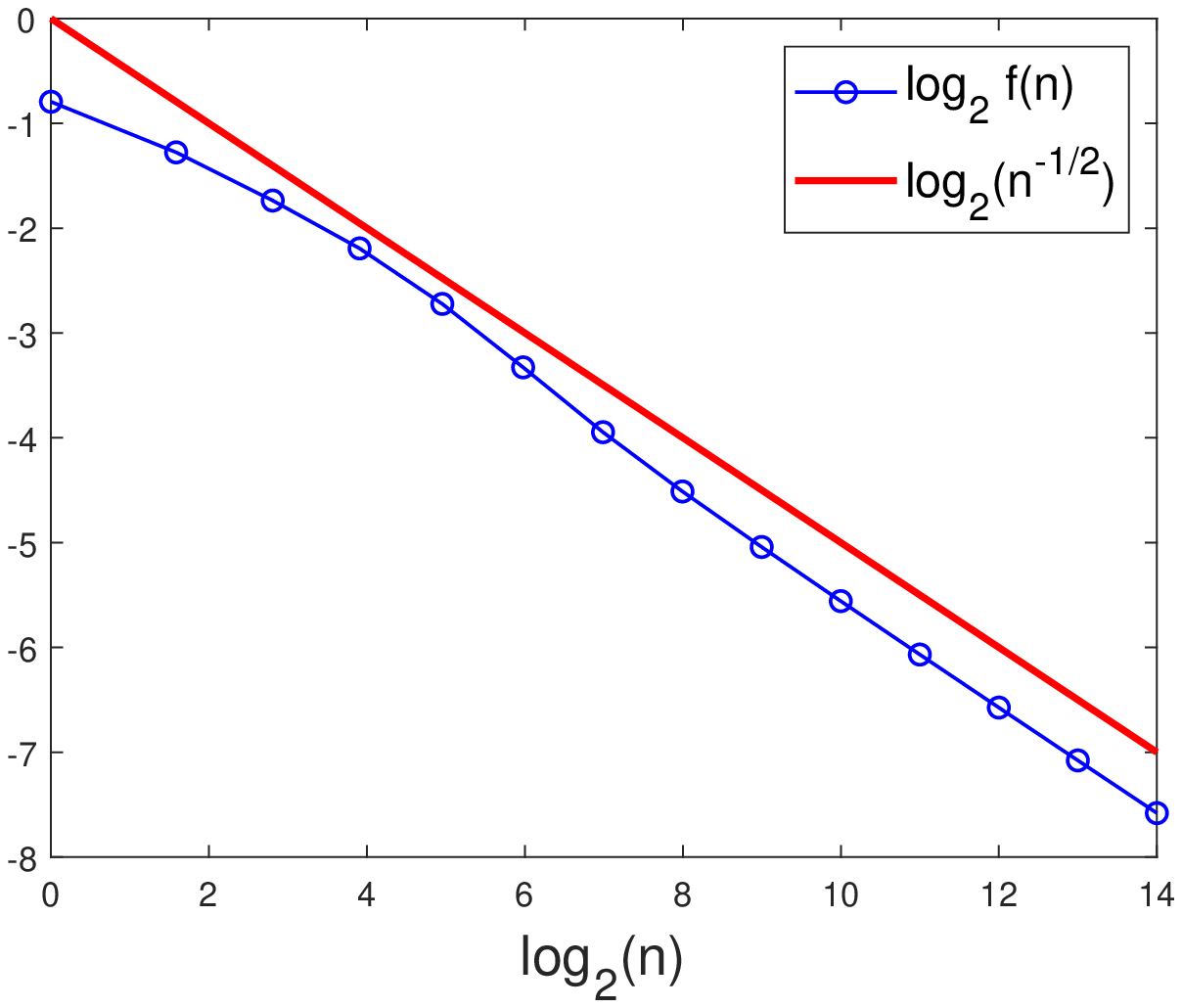}
    \caption{$\log_2 f(n)$ and $\log_2 n^{-1/2}$ versus $\log_2 n$}
    \end{subfigure}
    \begin{subfigure}{0.49\textwidth}
    \centering
    \includegraphics[scale=0.58]{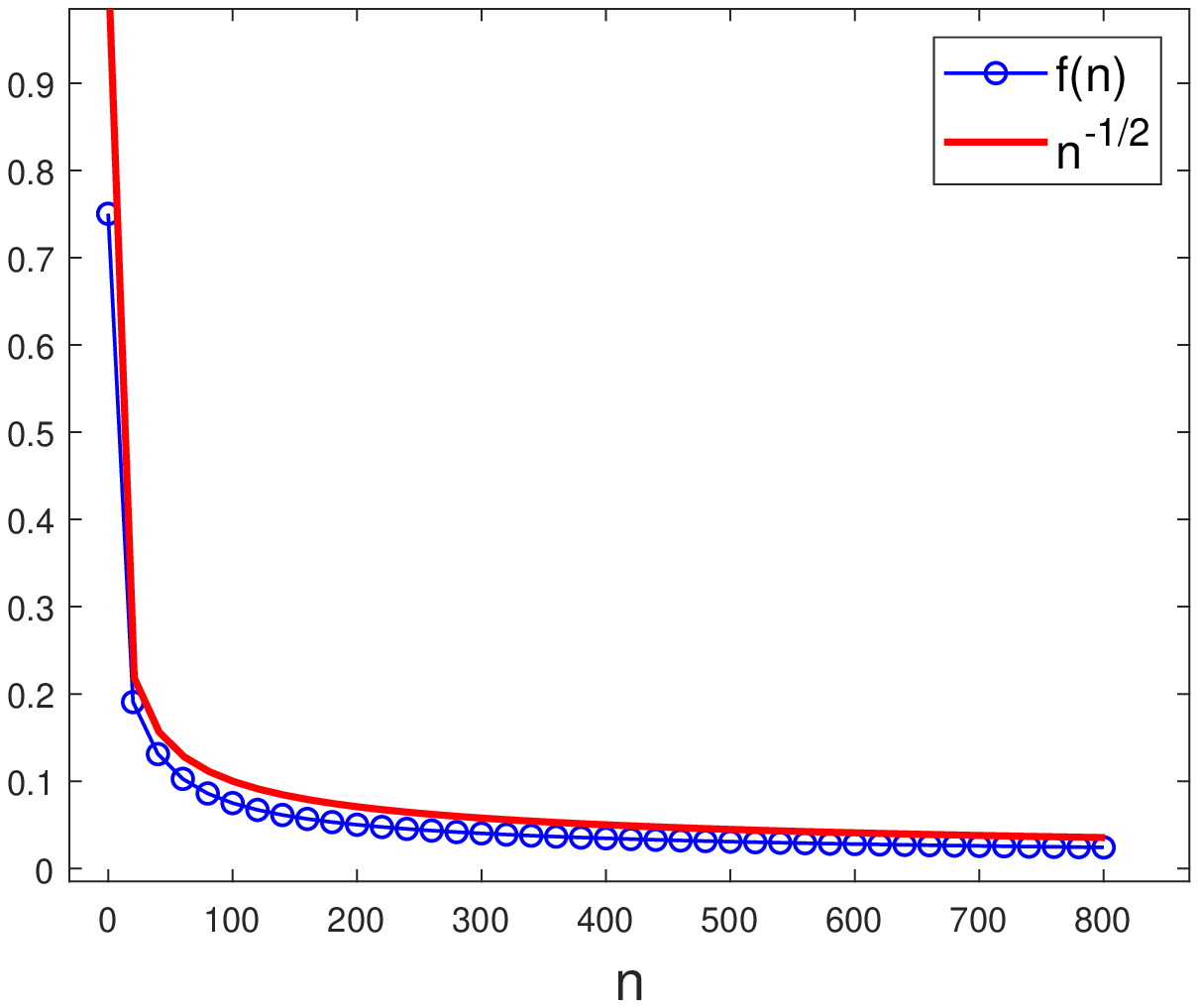}
    \caption{$f(n)$ and $n^{-1/2}$ versus $n$}
    \end{subfigure}
    \caption{Behavior of $f(n)=f_{\phi,K}(n)=\max_{(x,y)\in K}\abs{\phi^{(n)}(x,y)}$}
    \label{fig:Conv_Pwr_0_new}
\end{figure}

\begin{figure}[!htb]
    \begin{subfigure}{0.49\textwidth}
    \centering
    \includegraphics[scale=0.58]{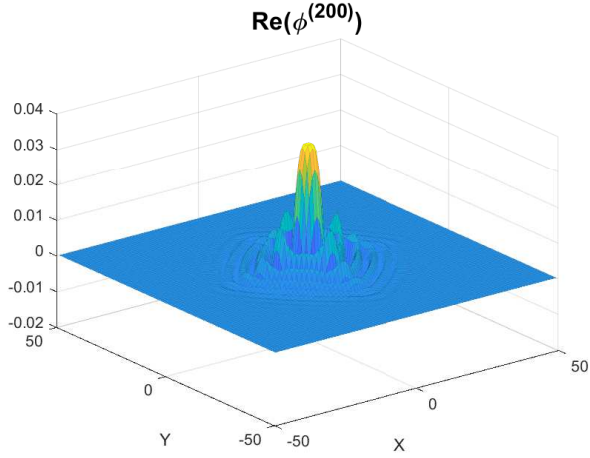}
    \caption{$n = 100$.}
    \label{fig:Conv_Pwr_00a}
    \end{subfigure}
    \begin{subfigure}{0.49\textwidth}
    \centering
    \includegraphics[scale=0.58]{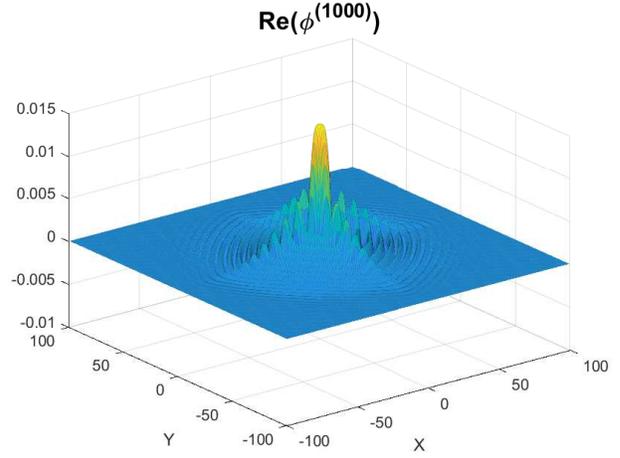}
    \caption{$n = 1000$.}
    \label{fig:Conv_Pwr_00b}
    \end{subfigure}
    \caption{$\Re\phi^{(n)}(x,y)$ for $n=200$ and $1000$}
    \label{fig:Conv_Pwr_00_new}
\end{figure}
\end{example}

%%% EXAMPLE 7
\begin{example}\normalfont
Consider the function $\phi : \mathbb{Z}^2 \to \mathbb{C}$ defined by 
\begin{equation*}
    \phi(x,y) = 
    \f{1}{768}\times
    \begin{cases}
    602 - 112i &(x,y) = (0,0)\\
    56 + 32i   &(x,y) = (0,\pm 1)\mbox{ or }(-1,0)\\
    72 + 32i   &(x,y) = (1,0)\\
    -28 - 8i   &(x,y) = (0,\pm 2)\\
    -16        &(x,y) = (\pm 2,0)\\
    56         &(x,y) = (0,\pm 3)\\
    -1         &(x,y) = (0,\pm 4)\\
    4          &(x,y) = (-1,\pm 1)\\
    -4         &(x,y) = (1,\pm 1)\\
    0          &\text{otherwise}.
    \end{cases}
\end{equation*}
As with the preceding examples, it is easy to see that $\sup_{\xi}|\widehat{\phi}(\xi)|=1$, $\Omega(\phi)=\{\xi_0\}=\{(0,0)\}$ and $\Gamma_{0}=\Gamma_{\xi_0}$ has the absolutely and uniformly convergent Taylor expansion
\begin{equation*}
    \Gamma_{0}(\xi)=-i\left(Q_0(\xi)+\widetilde{Q}_0(\xi)\right)-\left(R_0(\xi)+\widetilde{R}_0(\xi)\right)
\end{equation*}
where
\begin{equation*}
    Q_0(\xi)=\sum_{|\beta:(2,4)|=1}A_\beta \xi^\beta=\frac{\tau^2}{24}-\frac{\tau\zeta^2}{96} +\frac{ \zeta^4}{96},
\end{equation*}
\begin{equation*}
    R_0(\xi)=\sum_{|\beta:(2,4)|=2}B_\beta \xi^\beta=\frac{23\tau^4}{1152}  + \frac{\tau^3\zeta^2}{2304}  - \frac{\tau^2\zeta^4}{2048} + \f{\tau\zeta^6}{9216}+ \frac{23\zeta^8}{18432},
\end{equation*}
\begin{equation*}
    \widetilde{Q}_0(\xi)=\sum_{|\beta:(2,4)|\geq 3/2}A_\beta \xi^\beta=- \frac{\tau^4}{288}+\frac{\tau\zeta^4}{1152} +\frac{\tau^3\zeta^2}{576}  -\frac{\tau^3\zeta^4}{6912} %+ \frac{43\tau^4\zeta^4}{221184} 
    + \cdots,
\end{equation*}
and
\begin{equation*}
    \widetilde{R}_0(\xi)=\sum_{|\beta:(2,4)|\geq 5/2}B_\beta \xi^\beta
    =  - \frac{\tau^3\zeta^4}{27648} + \frac{\tau^4\zeta^4}{18432} +  \frac{\tau^4\zeta^4}{18432} + \cdots
\end{equation*}
for $\xi=(\tau,\zeta)\in\mathcal{U}$ where $\mathcal{U}\subseteq\mathbb{R}^2$ is a neighborhood of $0$. In this case, the above expansion is of the form \eqref{eq:SemiEllipticImaginaryExpansion} with $\alpha_0=(0,0)$, $\mathbf{m}=(1,2)$ and $k=2$. Here, as with the previous examples, it is readily verified that $|Q_0|=Q_0$ and $R_0$ are positive definite and so an appeal to Proposition \ref{prop:ExpandGamma} guarantees that $\xi_0=(0,0)$ is of imaginary homogeneous type for $\widehat{\phi}$ with drift $\alpha_0=(0,0)$ and homogeneous order
\begin{equation*}
    \mu_\phi=\mu_0=|\mathbf{1}:2\mathbf{m}|=
    \frac{1}{2}+\frac{1}{4}=\frac{3}{4}.
\end{equation*}
By an appeal to Theorem \ref{thm:ConvolutionPowerEstimate}, we obtain, to each compact set $K$, a positive constant $C$ for which
\begin{equation}\label{eq:SecondOrderExampleDecay2}
    \abs{\phi^{(n)}(x,y)}\leq\frac{C}{n^{\mu_\phi}}=\frac{C}{n^{3/4}}
\end{equation}
for all $n\in\mathbb{N}$ and $(x,y)\in K$. Figure \ref{fig:Conv_Pwr_2} illustrates this result by capturing the decay of $f(n)=f_{\phi,K}(n)=\max_{(x,y)\in K}\abs{\phi^{(n)}(x,y)}$ where %$K=[-50,50]\times [-50,50]$
$K=[-300,300]\times [-300,300]$ relative to that of $n^{-\mu_\phi}$. Also, Figure \ref{fig:Conv_Pwr_20} illustrates the graph of $\Re \phi^{(n)}(x,y)$ for $(x,y)\in K$ and $n=300$ and $n=600$.

\begin{figure}[!htb]
    \begin{subfigure}{0.49\textwidth}
    \centering
    \includegraphics[scale=0.58]{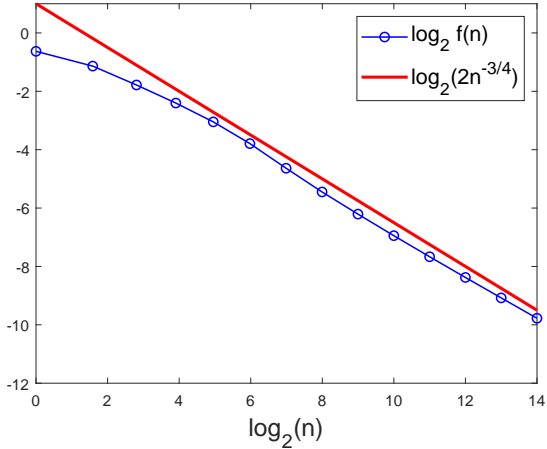}
    \caption{$\log_2 f(n)$ and  $\log_2 2n^{-3/4}$ versus $\log_2 n$.}
    \end{subfigure}
    \begin{subfigure}{0.49\textwidth}
    \centering
    \includegraphics[scale=0.58]{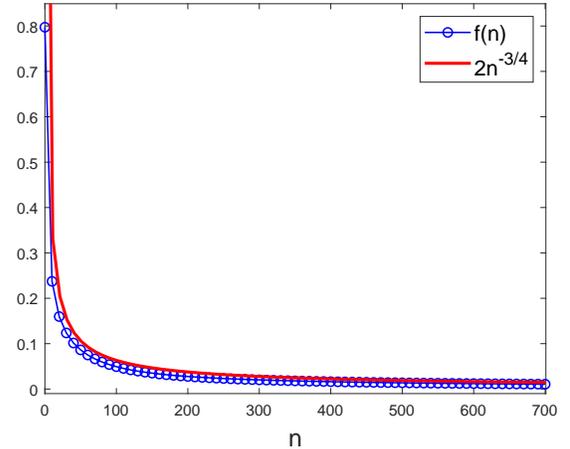}
    \caption{$f(n)$ and $2n^{-3/4}$ versus $n$}
    \end{subfigure}
    \caption{Behavior of $f(n)=f_{\phi,K}(n)=\max_{(x,y)\in K}\abs{\phi^{(n)}(x,y)}$.}
    \label{fig:Conv_Pwr_2}
\end{figure}

\begin{figure}[!htb]
    \begin{subfigure}{0.49\textwidth}
    \centering
    \includegraphics[scale=0.58]{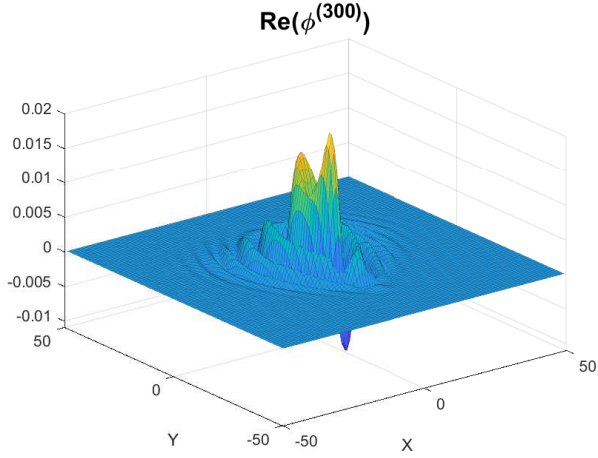}
    \caption{$n=300$.}
    \end{subfigure}
    \begin{subfigure}{0.49\textwidth}
    \centering
    \includegraphics[scale=0.58]{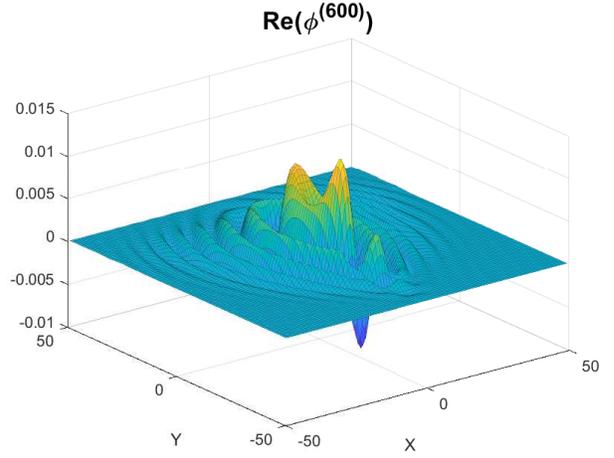}
    \caption{$n = 600$.}
    \label{fig:Conv_Pwr_2b}
    \end{subfigure}
    \caption{$\Re{\phi^{(n)}}$ for $n = 300$ and $n=600$}
    \label{fig:Conv_Pwr_20}
\end{figure}
\end{example}

%%%%%%%%%%%%%%%%%%%%%%%%%%%%%%%%%%%%%

%%%%% INTERESTING EXAMPLE %%%%%%%%

\begin{example}\normalfont
This example illustrates a complex-valued function $\phi$ on $\mathbb{Z}^2$ whose Fourier transform is maximized in absolute value at two distinct points in $\mathbb{T}^2$, one of which is a point of imaginary homogeneous type for $\widehat{\phi}$ with homogeneous order $2/3$ and the other is a point of positive homogeneous type $\widehat{\phi}$ of homogeneous order $1$. We define $\phi: \mathbb{Z}^2 \to \mathbb{C}$ by $\phi=2^{-7}\phi_1-i2^{-11}\phi_2+2^{-21}\phi_3$ where
\begin{equation*}
    \phi_1(x,y)=\begin{cases}
    15 + 15i &(x,y) = (\pm 1,0)\\
    16 + 16i &(x,y) = (0, \pm 1)\\
     1 + 1i &(x,y) = (\pm 3,0)\\
    0 &\mbox{otherwise}
    \end{cases},
    \hspace{1cm}
    \phi_2(x,y) = 
    \begin{cases}
    682 &(x,y) = (0,0)\\
    152  &(x,y) = (\pm 2,0)\\
    -28  &(x,y) = (\pm 4,0)\\
    8 &(x,y) = (\pm 6, 0)\\
    -1 &(x,y) = (\pm 8, 0)\\
    60  &(x,y) = (0, \pm 2)\\
    -24 &(x,y) = (0,\pm 4)\\
    4 &(x,y) = (0,\pm 6)\\
    0 &\mbox{otherwise}
    \end{cases},
\end{equation*}
and
\begin{equation*}
    \phi_3(x,y) = 
    \begin{cases}
    1387004 &(x,y) = (0,0)\\
    -106722 &(x,y) = (\pm 2,0)\\
    3960 &(x,y) = (\pm 4,0)\\
    -1045 &(x,y) = (\pm 6, 0)\\
    138  &(x,y) = (\pm 8, 0)\\
    -9 &(x,y) = (\pm 10, 0)\\
    -131072 &(x,y) = (0, \pm 2)\\
    0 &\mbox{otherwise}
    \end{cases}
\end{equation*}
for $(x,y)\in\mathbb{Z}^2$. Though this example is slightly more complicated than the previous ones considered, it is straightforward to verify that $\sup_{\xi}|\widehat{\phi}(\xi)|=1$ and, in this case, the supremum is attained at two points in $\mathbb{T}^2$. Specifically, $\Omega(\phi)=\{\xi_1,\xi_2\}$ where $\xi_1=(0,0)$ and $\xi_2=(\pi,\pi)$. For $\xi_1$, $\Gamma_1=\Gamma_{\xi_1}$ has an absolutely and uniformly convergent Taylor series of the form
\begin{equation*}
    \Gamma_{1}(\xi)=-i\left(Q_{1}(\xi)-\widetilde{Q_1}(\xi)\right)-\left(R_1(\xi)+\widetilde{R_1}(\xi)\right)
\end{equation*}
for $\xi=(\tau,\zeta)\in\mathcal{U}_1$ where $\mathcal{U}_1\subseteq\mathbb{R}^2$ is an open neighborhood of $(0,0)$ and
\begin{equation*}
    Q_{1}(\xi)=\sum_{|\beta:(6,2)|=1}A_\beta \xi^\beta=\frac{\tau^6}{128}+\frac{\zeta^2}{8},
\end{equation*}
\begin{equation*}
    R_1(\xi)=\sum_{|\beta:(6,2)|=2}B_\beta\xi^\beta=\frac{111\tau^{12}}{32768}-\frac{\tau^6 \zeta^2}{1024}+\frac{3\zeta^4}{128},
\end{equation*}
\begin{equation*}
    \widetilde{Q_1}(\xi)=\sum_{|\beta:(6,2)|\geq 4/3}A_\beta \xi^\beta=-\f{65\tau^8}{512} -\f{\zeta^4}{96} + \frac{\tau^6\zeta^4}{8192} 
    +\cdots,
\end{equation*}
and
\begin{equation*}
    \widetilde{R_1}(\xi)=\sum_{|\beta:(6,2)|\geq 7/3}B_\beta\xi^\beta=\f{65\tau^8\zeta^2}{4096}  + \f{\tau^6 \zeta^4}{12288} + \cdots
\end{equation*}
for $\xi=(\tau,\zeta)\in\mathcal{U}_1$. It is straightforward to verify that $Q_1=\abs{Q_1}$ and $R_1$ are positive definite and so Proposition \ref{prop:ExpandGamma} guarantees that $\xi_1=(0,0)$ is of imaginary homogeneous type for $\widehat{\phi}$ with $\mathbf{m}_1=(3,1)$, $k_1=2$, drift $\alpha_{\xi_1}=(0,0)$ and homogeneous order
\begin{equation*}
    \mu_{\xi_1}=|\mathbf{1}:2\mathbf{m}_1|=\frac{1}{6}+\frac{1}{2}=\frac{2}{3}.
\end{equation*}
For $\xi_2 = (\pi,\pi)$, $\Gamma_2=\Gamma_{\xi_2}$ has an absolutely and uniformly convergent Taylor series of the form
\begin{equation*}
    \Gamma_2(\xi)=-i\left(Q_2(\xi)+\widetilde{Q_2}(\xi)\right)-\left(R_2(\xi)+\widetilde{R_2}(\xi)\right)
\end{equation*}
for $\xi=(\tau,\zeta)\in\mathcal{U}_2$ where $\mathcal{U}_2\subseteq\mathbb{R}^2$ is an open neighborhood of $(0,0)$ and 
\begin{equation*}
    Q_2(\xi)=\sum_{|\beta:(2,2)|=1}A_\beta \xi^\beta=-\left(\frac{3\tau^2}{8}+\frac{\zeta^2}{4}\right),
\end{equation*}
\begin{equation*}
    R_2(\xi)=\sum_{|\beta:(2,2)|=1}B_\beta\xi^\beta=\frac{\tau^2}{8}+\frac{3\zeta^2}{8},
\end{equation*}
\begin{equation*}
    \widetilde{Q_2}(\xi)=\sum_{|\beta:(2,2)|\geq 2}A_\beta\xi^\beta=\frac{\tau^4}{64}-\frac{9\tau^2\zeta^2}{64}+\frac{\zeta^4}{48}+\cdots,
\end{equation*}
and
\begin{equation*}
    \widetilde{R_2}(\xi)=\sum_{|\beta:(2,2)|\geq 2}B_\beta\xi^\beta=-\frac{\tau^4}{8}-\frac{3\tau^2\zeta^2}{64}-\frac{13\zeta^4}{384}+\cdots,
\end{equation*}
for $\xi=(\tau,\zeta)\in\mathcal{U}_2$. Thus, the expansion is of the form \eqref{eq:SemiEllipticImaginaryExpansion} with $\mathbf{m}_2=(1,1)$ and $k_2=1$. Since $R_2$ is clearly positive definite, Proposition \ref{prop:ExpandGamma} guarantees that $\xi_2=(\pi,\pi)$ is of positive homogeneous type for $\widehat{\phi}$ with drift $\alpha_{\xi_2}=(0,0)$ and homogeneous order
\begin{equation*}
    \mu_{\xi_2}=|\mathbf{1}:2\mathbf{m}_2|=\frac{1}{2}+\frac{1}{2}=1.
\end{equation*}
Upon noting that $\mu_{\phi}=\min\{\mu_{\xi_1},\mu_{\xi_2}\}=2/3$, an appeal to Theorem \ref{thm:ConvolutionPowerEstimate} guarantees, to each compact set $K$, a constant $C$ for which
\begin{equation}
    \abs{\phi^{(n)}(x,y)}\leq\frac{C}{n^{\mu_\phi}}=\frac{C}{n^{2/3}}
\end{equation}
for all $n\in\mathbb{N}$ and $(x,y)\in K$. Figure \ref{fig:Conv_Pwr_5} illustrates this result by capturing the decay of $f(n)=f_{\phi,K}(n)=\max_{(x,y)\in K}\abs{\phi^{(n)}(x,y)}$ where $K=[-500,500]\times [-500,500]$ relative to that of $n^{-\mu_\phi}$. Also, Figure \ref{fig:Conv_Pwr_50} illustrates the graph of $\Re \phi^{(n)}(x,y)$ for $(x,y)\in K$ and $n=200$ and $n=700$.

\begin{figure}[!htb]
    \begin{subfigure}{0.49\textwidth}
    \centering
    \includegraphics[scale=0.58]{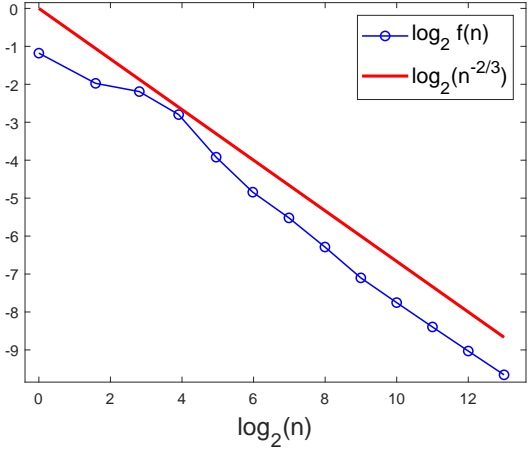}
    \caption{$\log_2 f(n)$ and $\log_2 n^{-2/3}$ versus $\log_2 n$}
    \end{subfigure}
    \begin{subfigure}{0.49\textwidth}
    \centering
    \includegraphics[scale=0.58]{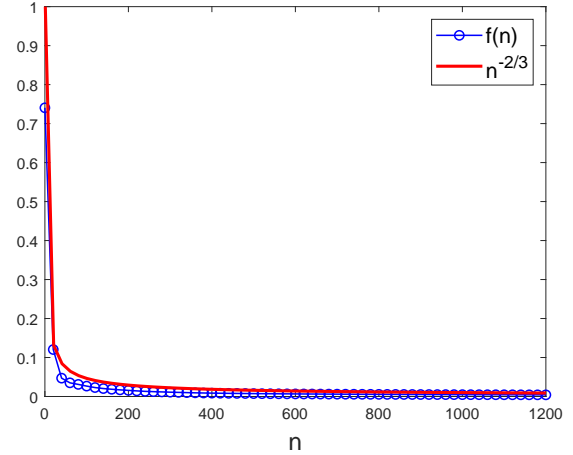}
    \caption{$f(n)$ and $n^{-2/3}$ versus $n$}
    \end{subfigure}
    \caption{Behavior of $f(n)=f_{\phi,K}(n)=\max_{(x,y)\in K}\abs{\phi^{(n)}(x,y)}$.}
    \label{fig:Conv_Pwr_5}
\end{figure}

\begin{figure}[!htb]
    \begin{subfigure}{0.49\textwidth}
    \centering
    \includegraphics[scale=0.58]{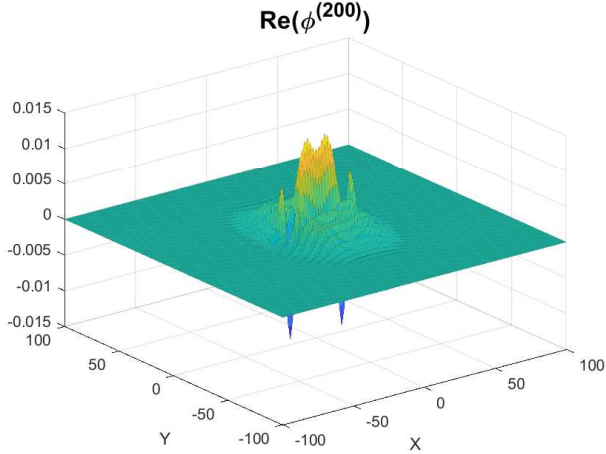}
    \caption{$n = 200$.}
    \label{fig:Conv_Pwr_5a}
    \end{subfigure}
    \begin{subfigure}{0.49\textwidth}
    \centering
    \includegraphics[scale=0.58]{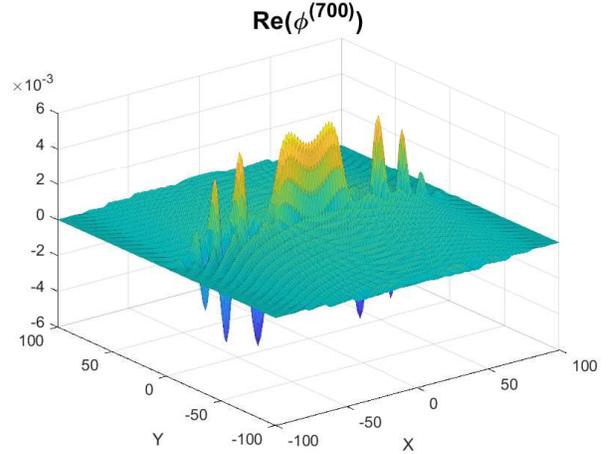}
    \caption{$n = 700$.}
    \label{fig:Conv_Pwr_5b}
    \end{subfigure}
    \caption{$\Re{\phi^{(n)}}$ for $n = 200$ and $n=700$.}
    \label{fig:Conv_Pwr_50}
\end{figure}

\end{example}

%%%%%%%%%%%%%%%%%%%%%%%%%%

\section{Proof of Theorem \ref{thm:BestIntegrationFormula}}\label{sec:ProofofBest}

\subsection{Construction of $\sigma_{P,E}$}\label{subsec:ConstructionofSigma}

Throughout this section, we fix $E\in\Exp(P)$. Define $\psi_E:(0,\infty)\times S\to\mathbb{R}^d\setminus\{0\}$ by
\begin{equation}\label{eq:Homeomorphism}
\psi_E(r,\eta)=r^E\eta
\end{equation}
for $r>0$ and $\eta\in S$. As $\psi_E$ is the restriction of the continuous function $(0,\infty)\times \mathbb{R}^d\ni (r,x)\mapsto r^E x\in\mathbb{R}^d$ to $(0,\infty)\times S$, it is necessarily continuous. As the following proposition shows, $\psi_E$ is, in fact, a homeomorphism.

\begin{proposition}\label{prop:PsiHomeomorphism}
The map $\psi_E:(0,\infty)\times S\to\mathbb{R}^d\setminus\{0\}$ is a homeomorphism with continuous inverse $\psi_E^{-1}:\mathbb{R}^d\setminus\{0\}\to (0,\infty)\times S$ given by
\begin{equation*}
\psi_E^{-1}(x)=(P(x),(P(x))^{-E}x)
\end{equation*}
for $x\in\mathbb{R}^d\setminus\{0\}$.
\end{proposition}

\begin{proof}
Given that $P$ is continuous and positive definite, $P(x)>0$ for each $x\in \mathbb{R}^d\setminus\{0\}$ and the map $\mathbb{R}^d\setminus\{0\}\ni x \mapsto (P(x))^{-E}x\in \mathbb{R}^d$ is continuous. Further, in view of the homogeneity of $P$,
\begin{equation*}
P\left((P(x))^{-E} x \right)=P(x)^{-1}P(x)=1
\end{equation*}
for all $x\in\mathbb{R}^d\setminus\{0\}$. It follows from these two observations that
\begin{equation*}
\rho(x)=(P(x),(P(x))^{-E}x),
\end{equation*}
defined for $x\in\mathbb{R}^d\setminus\{0\}$, is a continuous function taking $\mathbb{R}^d\setminus\{0\}$ into $(0,\infty)\times S$. We have
\begin{equation*}
(\psi_E\circ \rho)(x)=\psi_E(P(x),(P(x))^{-E}x)=(P(x))^{E}(P(x))^{-E}x=x
\end{equation*}
for every $x\in \mathbb{R}^d\setminus \{0\}$ and
\begin{equation*}
(\rho\circ\psi_E)(r,\eta)=\rho(r^E\eta)=(P(r^{E}\eta),(P(r^{E}\eta))^{-E}(r^E\eta))=(rP(\eta),(rP(\eta))^{-E}(r^{E}\eta))=(r,\eta)
\end{equation*}
for every $(r,\eta)\in (0,\infty)\times S$. Thus $\rho$ is a (continuous) inverse for $\psi_E$ and so it follows that $\psi_E$ is a homeomorphism and $\rho=\psi_E^{-1}$.
\end{proof}

\noindent We shall now construct the $\sigma$-algebra $\Sigma_{P,E}$ on $S$; later, we will show that it is independent of our choice of $E$. As in the statement of Theorem \ref{thm:BestIntegrationFormula}, for each $F\subseteq S$, define
\begin{equation*}
\widetilde{F_E}=\bigcup_{0<r<1}\left(r^E F\right)=\{r^E\eta:0<r<1,\eta\in F\}. 
\end{equation*}
We shall denote by $\Sigma_{P,E}$ the collection of subsets $F$ of $S$ for which $\widetilde{F_E}\in\mathcal{M}_d$, i.e.,  
\begin{equation*}
\Sigma_{P,E}=\{F\subseteq S:\widetilde{F_E}\in\mathcal{M}_d\}.
\end{equation*}

\begin{proposition}\label{prop:BorelContainment}
$\Sigma_{P,E}$ is a $\sigma$-algebra on $S$ containing the Borel $\sigma$-algebra on $S$, $\mathcal{B}(S)$.
\end{proposition}

\begin{proof}
Throughout the proof, we write $\Sigma=\Sigma_{P,E}$ and $\widetilde{F}=\widetilde{F_E}$ for each $F\subseteq S$.
We first show that $\Sigma$ is a $\sigma$-algebra. Since $\widetilde S=B\setminus\{0\}$, it is open in $\mathbb{R}^d\setminus\{0\}$ and therefore Lebesgue measurable. Hence $S\in \Sigma$. Let $G, F\in \Sigma$ be such that $G\subseteq F$. Then,
\begin{equation*}
\widetilde{F\setminus G}=\bigcup_{0<r<1}r^E\left(F\setminus G\right)=\bigcup_{0<r<1}\left(r^EF\setminus r^E G\right)=\left(\bigcup_{0<r<1}r^E F\right)\setminus\left(\bigcup_{0<r<1}r^E G\right)=\widetilde F\setminus \widetilde G
\end{equation*}
where we have used the fact that the collection $\{r^E F\}_{0<r<1}$ is mutually disjoint to pass the union through the set difference. Consequently $\widetilde F\setminus \widetilde{G}$ is Lebesgue measurable and therefore $F\setminus G\in \Sigma$.  Now, given a countable collection $\{F_n\}\subseteq \Sigma$, observe that
\begin{equation*}
    \widetilde{\bigcup_{n=1}^\infty F_n}= \bigcup_{0<r<1}r^E \left(\bigcup_{n=1}^\infty F_n\right)= \bigcup_{0 <r< 1}  \bigcup_{n=1}^\infty  r^E F_n =\bigcup_{n=1}^\infty \bigcup_{0 <r < 1}  r^E F_n =\bigcup_{n=1}^\infty \widetilde{F_n} \in \mathcal{M}_d
\end{equation*}
whence $\cup_n F_n\in \Sigma$. Thus $\Sigma$ is a $\sigma$-algebra. 

Finally, we show that
\begin{equation*}
\mathcal{B}(S)\subseteq\Sigma.
\end{equation*}
As the Borel $\sigma$-algebra is the smallest $\sigma$-algebra containing the open subsets of $S$, it suffices to show that $\mathcal{O}\in \Sigma$ whenever $\mathcal{O}$ is open in $S$. Armed with Proposition \ref{prop:PsiHomeomorphism}, this is an easy task: Given an open set $\mathcal{O}\subseteq S$, observe that
\begin{equation*}
\widetilde{\mathcal{O}}=\{r^E\eta:0<r<1,\eta\in\mathcal{O}\}=\psi_E((0,1)\times\mathcal{O}).
\end{equation*}
Upon noting that $(0,1)\times\mathcal{O}$ is an open subset of $(0,\infty)\times S$, Proposition \ref{prop:PsiHomeomorphism} guarantees that $\widetilde{\mathcal{O}}=\psi_E((0,1)\times\mathcal{O})\subseteq\mathbb{R}^d\setminus\{0\}$ is open and therefore $\widetilde{\mathcal{O}} \in \mathcal{M}_d$. Thus, $\mathcal{O}\in \Sigma$.
\end{proof}

\noindent We are now ready to specify a measure on the measurable space $(S,\Sigma_{P,E})$. For each $F\in \Sigma_{P,E}$, we define
\begin{equation*}
\sigma_{P,E}(F)=\mu_P\cdot m(\widetilde{F_E})
\end{equation*}
where $m$ is the Lebesgue measure on $\mathbb{R}^d$ and $\mu_P=\tr E>0$ is the homogeneous order associated to $P$.

\begin{proposition}\label{prop:sigmaisameaure}
$\sigma_{P,E}$ is a finite measure on $(S,\Sigma_{P,E})$.
\end{proposition}
\begin{proof}

\noindent Throughout the proof, we will write $\sigma=\sigma_{P,E}$, $\Sigma=\Sigma_{P,E}$, and, $\widetilde{F}=\widetilde{F_E}$ for each $F\subseteq S$. It is clear that $\sigma$ is non-negative and $\sigma(\varnothing)=0$. Let $\{ F_n  \}^\infty_{n=1} \subseteq \Sigma $ be a mutually disjoint collection. We claim that $\{ \widetilde{F_n} \}_{n=1}^\infty\subseteq\mathcal{M}_d$ is also a mutually disjoint collection. To see this, suppose that $x = r_n^E \eta_n = r_m^E \eta_m\in \widetilde{F_n}\cap\widetilde{F_m}$, where $r_n,r_m \in (0,1)$, $\eta_n \in F_n$, and $\eta_m \in F_m $. Then
\begin{equation*}
    r_n = P(r_n^E \eta_n) = P(x) = P(r_m^E \eta_m) = r_m,
\end{equation*}
implying that $\eta_n = \eta_m\in F_n\cap F_m$. Because $\{F_n\}_{n=1}^\infty$ is mutually disjoint, we must have $n=m$ which verifies our claim. By virtue of the countable additivity of Lebesgue measure, we therefore have
\begin{equation*}
\sigma\left(\bigcup_{n=1}^\infty F_n\right)
    = \mu_P\cdot m\left( \widetilde{\bigcup^\infty_{n=1} F_n } \right)=\mu_P\cdot m\left( \bigcup^\infty_{n=1}\widetilde{F_n} \right)
    = \mu_P\sum^\infty_{n=1} m(\widetilde{F_n})
    = \sum^\infty_{n=1}\sigma(F_n).
\end{equation*}
Therefore $\sigma$ is a measure on $(S,\Sigma)$. In view of Condition \ref{cond:PisAboveOne} of Proposition \ref{prop:PositiveHomogeneousCharacterization}, $\widetilde{S}=B\setminus\{0\}$ is a bounded subset of $\mathbb{R}^d\setminus\{0\}$ and hence $\sigma(S)=\mu_P\cdot m(B\setminus\{0\})<\infty$ showing that $\sigma$ is finite.
\end{proof}

\noindent By virtue of the two preceding propositions, $\sigma_{P,E}$ is a finite Borel measure on $S$. In fact, as a consequence of next subsection's main result, Theorem \ref{thm:MainIntegrationFormula}, we will see that $\sigma_{P,E}$ is independent of our choice of $E\in\Exp(P)$ and is a Radon measure; see Subsection \ref{subsec:IndependentofE}.

\subsection{Product Measure and Point Isomorphism}\label{subsec:ProductMeasure}

Throughout this subsection, $E\in\Exp(P)$ will remain fixed and $(S,\Sigma_{P,E},\sigma_{P,E})$ will denote the finite measure space of Proposition \ref{prop:sigmaisameaure}. We recall that $\mathcal{L}$ denotes the $\sigma$-algebra of Lebesgue measurable subsets of $(0,\infty)$ and  $\lambda_P$ denotes the measure on $(0,\infty)$ with $\lambda_P(dr)=r^{\mu_P-1}\,dr$, i.e., for each $L\in\mathcal{L}$,
\begin{equation*}
\lambda_P(L)=\int_0^\infty \chi_L(r)r^{\mu_P-1}\,dr.
\end{equation*}
It is easy to see that $\lambda_P$ is $\sigma$-finite and so, in view of the finiteness of the measure $\sigma_{P,E}$, there exists a unique product measure $\lambda_P\times\sigma_{P,E}$ on $(0,\infty)\times S$ equipped with the product $\sigma$-algebra $\mathcal{L}\times\Sigma_{P,E}$ which satisfies
\begin{equation*}
    (\lambda_P\times\sigma_{P,E})(L\times F)=\lambda_P(L)\sigma_{P,E}(F)
\end{equation*}
for all $L\in\mathcal{L}$ and $F\in\Sigma_{P,E}$. We shall denote by $((0,\infty)\times S,(\mathcal{L}\times\Sigma_{P,E})',\lambda_P\times\sigma_{P,E})$ the completion of the measure space $((0,\infty)\times S,\mathcal{L}\times\Sigma_{P,E},\lambda_P\times\sigma_{P,E})$. Our primary goal in this subsection is to prove the theorem below. We note that Properties \ref{item:MainIntegrationFormula1} and \ref{item:MainintegrationFormula2} in Theorem \ref{thm:MainIntegrationFormula} differ only from Properties \ref{property:BestPointIsomorphism} and \ref{property:BestIntegrationFormula} in Theorem \ref{thm:BestIntegrationFormula} in that, a priori, the $\sigma$-algebra $\Sigma_{P,E}$ and the measure $\sigma_{P,E}$ in Theorem \ref{thm:MainIntegrationFormula} depends on our choice of $E\in\Exp(P)$. As a consequence of Theorem \ref{thm:MainIntegrationFormula}, we shall see in Subsection \ref{subsec:IndependentofE} that $\Sigma_{P,E_1}=\Sigma_{P,E_2}$ and $\sigma_{P,E_1}=\sigma_{P,E_2}$ for all $E_1,E_2\in\Exp(P)$ and so this apparent dependence is superficial; this is Proposition \ref{prop:Endependence}. As a consequence of the proposition, we shall obtain Properties \ref{property:BestPointIsomorphism} and \ref{property:BestIntegrationFormula} of Theorem \ref{thm:BestIntegrationFormula} immediately from Properties of \ref{item:MainIntegrationFormula1} and \ref{item:MainintegrationFormula2} in Theorem \ref{thm:MainIntegrationFormula}.

\begin{theorem}\label{thm:MainIntegrationFormula}
Let $((0,\infty)\times S,(\mathcal{L}\times\Sigma_{P,E})',\lambda_P\times\sigma_{P,E})$ be as above.
\begin{enumerate}
\item\label{item:MainIntegrationFormula1} The map $\psi_E: (0,\infty)\times S\to\mathbb{R}^d\setminus\{0\}$, defined by \eqref{eq:Homeomorphism}, is a point isomorphism of the measure spaces $((0,\infty)\times S,(\mathcal{L}\times\Sigma_{P,E})',\lambda_P\times\sigma_{P,E})$ and $(\mathbb{R}^d\setminus\{0\},\mathcal{M}_d,m)$. That is
\begin{equation*}
\mathcal{M}_d=\{A\subseteq \mathbb{R}^d\setminus\{0\}:\psi_E^{-1}(A)\in(\mathcal{L}\times\Sigma_{P,E})'\}
\end{equation*}
and, for each $A\in\mathcal{M}_d$,
\begin{equation*}
m(A)=(\lambda_P\times\sigma_{P,E})(\psi_E^{-1}(A)).
\end{equation*}
\item\label{item:MainintegrationFormula2} If $f:\mathbb{R}^d\to\mathbb{C}$ is Lebesgue measurable, then $f\circ \psi_E$ is $(\mathcal{L}\times\Sigma_{P,E})'$-measurable and the following statements hold:
\begin{enumerate}
\item If $f\geq 0$, then
\begin{equation}\label{eq:MainIntegrationFormula}
\int_{\mathbb{R}^d}f(x)\,dx=\int_0^\infty\left(\int_S f(r^E\eta)\,\sigma_{P,E}(d\eta)\right)r^{\mu_P-1}\,dr=\int_S\left(\int_0^\infty f(r^E\eta)r^{\mu_P-1}\,dr\right)\,\sigma_{P,E}(d\eta).
\end{equation}
\item When $f$ is complex-valued, we have 
\begin{equation*}f\in L^1(\mathbb{R}^d)\hspace{.5cm}\mbox{if and only if}\hspace{0.5cm}f\circ\psi_E\in L^1((0,\infty)\times S,(\mathcal{L}\times \Sigma_{P,E})',\lambda_P\times\sigma_{P,E})
\end{equation*}
and, in this case, \eqref{eq:MainIntegrationFormula} holds.
\end{enumerate}
\end{enumerate}
\end{theorem}

\noindent To prove Theorem \ref{thm:MainIntegrationFormula}, we shall first treat several lemmas. These lemmas isolate and generalize several important ideas used in standard proofs of \eqref{eq:StandardPolarIntegrationFormula} (see, e.g., \cite{folland_real_2013} and \cite{stein_real_2009}). 

\begin{lemma}\label{lemma:Scaling}
Let $A\subseteq\mathbb{R}^d$ and $r>0$.  $A$ is Lebesgue measurable if and only if $r^E A=\{x=r^E a:a\in A\}$ is Lebesgue measurable and, in this case,
\begin{equation*}
m(r^E A)=r^{\mu_P}m(A).
\end{equation*}
\end{lemma}

\begin{proof}
Because $x\mapsto r^E x$ is a linear isomorphism, $r^E A$ is Lebesgue measurable if and only if $A$ is Lebesgue measurable. Observe that $x\in r^E A$ if and only if $r^{-E}x\in A$ and therefore
\begin{equation*}
m(r^E A)=\int_{\mathbb{R}^d}\chi_{r^E A}(x)\,dx=\int_{\mathbb{R}^d}\chi_{A}(r^{-E}x)\,dx
\end{equation*}
where $\chi_{r^EA}$ and $\chi_{A}$ respectively denote the indicator functions of the sets $r^EA$ and $A$. Now, by making the linear change of variables $x\mapsto r^E x$, we have
\begin{equation*}
m(r^E A)=\int_{\mathbb{R}^d}\chi_A(x)|\det(r^E)|\,dx=r^{\mu_P}m(A),
\end{equation*}
because $\det(r^E)=r^{\tr E}=r^{\mu_P}>0$ by virtue of Proposition \ref{prop:ContinuousGroupProperties} and Corollary \ref{cor:TraceisInvariant}.
\end{proof}

\begin{lemma}\label{lem:SpecialRectangle}
Let $F\in\Sigma_{P,E}$. If $I\subseteq (0,\infty)$ is open, closed, $G_\delta$, or $F_\sigma$, then $\psi_E(I\times F)\in\mathcal{M}_d$ and
\begin{equation}\label{eq:SpecialRectangle}
m(\psi_E(I\times F))=(\lambda_P\times\sigma_{P,E})(I\times F)=\lambda_P(I)\sigma_{P,E}(F).
\end{equation}
\end{lemma}
\begin{proof}
To simplify notation, we shall write $\lambda=\lambda_P$ and $\sigma=\sigma_{P,E}$ throughout the proof. We fix $F\in\Sigma_{P,E}$ and consider several cases for $I$. If $I=(0,b)$ for $0<b<\infty$, we have
\begin{equation*}
\psi_E(I\times F)=\{r^E\eta:0<r<b,\eta\in F\}=b^E\{r^E\eta:0<r<1,\eta \in F\}=b^E\widetilde{F_E}.
\end{equation*}
By virtue of Lemma \ref{lemma:Scaling}, it follows that $\psi_E(I\times F)\in\mathcal{M}_d$ and
\begin{eqnarray*}
(\lambda\times\sigma)(I\times F)&=&\lambda(I)\sigma(F)\\
&=&\left(\int_0^b r^{\mu_P-1}\,dr\right)\left(\mu_P\cdot m(\widetilde{F_E})\right)\\
&=&b^{\mu_P}m(\widetilde{F_E})\\
&=&m(b^{E}\widetilde{F_E})\\
&=&m(\psi_E(I\times F)).
\end{eqnarray*}
Using this result and the continuity of the measures $\lambda\times\sigma$ and $m$, standard arguments guarantee that $\psi_E(I\times F)\in\mathcal{M}_d$ and
\begin{equation*}
    (\lambda\times\sigma)(I\times F)=\left(b^{\mu_P}-a^{\mu_P}\right)\sigma(F)=m(\psi_E(I\times F))
\end{equation*}
whenever $I$ is an interval of the form $(a,b),\,(a,b],\,[a,b),$ and $[a,b]$ for $0\leq a\leq b\leq \infty$. Further, using the fact that every open subset of $(0,\infty)$ is countable union of disjoint open intervals, another standard argument using the continuity of the measures $\lambda\times \sigma$ and $m$ guarantees that $\psi_E(I\times F)\in\mathcal{M}_d$ and
\begin{equation*}
    (\lambda\times\sigma)(I\times F)=m(\psi_E(I\times F))
\end{equation*}
whenever $I$ is an open set. We then extend this to the case in which $I$ is a $G_\delta$ by virtue of the continuity of measure. Finally, by taking complements and using the continuity of measure, we find that the assertion holds whenever $I$ is an $F_\sigma$ set. 
\end{proof}

\begin{lemma}\label{lem:AllMeasurableRectangles} For any $L\in\mathcal{L}$ and $F\in \Sigma_{P,E}$, $\psi_E(L\times F)\in\mathcal{M}_d$ and 
\begin{equation*}
m(\psi_E(L\times F))=(\lambda_P\times\sigma_{P,E})(L\times F).
\end{equation*}
\end{lemma}
\begin{proof}
Fix $L\in\mathcal{L}$ and $F\in\Sigma_{P,E}$. It is easy to see that $\lambda_P$ and the Lebesgue measure $dr$ on $(0,\infty)$ are mutually absolutely continuous. It follows that $((0,\infty), \mathcal{L},\lambda_P)$ is a complete measure space and, further, that there exists an $F_\sigma$ set $L_\sigma\subseteq (0,\infty)$ and a $G_\delta$ set $L_\delta\subseteq (0,\infty)$ for which $L_\sigma\subseteq L\subseteq L_\delta$ and $\lambda_P(L_\delta\setminus L_\sigma)=0$. Note that, necessarily, $\lambda_P(L)=\lambda_P(L_\sigma)=\lambda_P(L_\delta)$. We have
\begin{equation}\label{eq:AllMeasurableRectangles1}
\psi_E(L\times F)=\psi_E( L_\sigma\times F)\cup\psi_E((L\setminus L_\sigma)\times F)
\end{equation}
where, by virtue of the preceding lemma, $\psi_E(L_\sigma\times F)\in \mathcal{M}_d$ and
\begin{equation}\label{eq:AllMeasurableRectangles2}
m(\psi_E(L_{\sigma}\times F))=(\lambda_P\times\sigma_{P,E})( L_\sigma\times F)=\lambda_P(L_\sigma)\sigma_{P,E}(F)=\lambda_P(L)\sigma_{P,E}(F)=(\lambda_P\times\sigma_{P,E})(L\times F).
\end{equation}
Observe that
\begin{equation*}
\psi_E((L\setminus L_\sigma)\times F)\subseteq \psi_E((L_{\delta}\setminus L_\sigma)\times F)
\end{equation*}
where, because $L_\delta\setminus L_\sigma$ is an $G_{\delta}$ set, the latter set is a member of $\mathcal{M}_d$ and
\begin{equation*}
m(\psi_E((L_\delta\setminus L_\sigma)\times F))=(\lambda_P\times\sigma_{P,E})((L_\delta\setminus L_\sigma)\times F)=\lambda_P(L_\delta\setminus L_\sigma)\sigma_{P,E}(F)=0
\end{equation*}
by virtue of the preceding lemma. Using the fact that $(\mathbb{R}^d\setminus\{0\},\mathcal{M}_d,m)$ is complete, we conclude that $\psi_E((L\setminus L_\sigma)\times F)\in \mathcal{M}_d$ and $m(\psi_E((L\setminus L_\sigma)\times F))=0$. It now follows from \eqref{eq:AllMeasurableRectangles1} and \eqref{eq:AllMeasurableRectangles2} that $\psi_E(L\times F)\in\mathcal{M}_d$ and
\begin{equation*}
m(\psi_E(L\times F))=m(\psi_E(L_\sigma\times F))+m(\psi_E((L\setminus L_\sigma)\times F))=(\lambda_P\times\sigma_{P,E})(L\times F),
\end{equation*}
as desired.
\end{proof}

\begin{lemma}\label{lem:OpenRectangle}
Every open subset $U\subseteq \mathbb{R}^d\setminus\{0\}$ can be written as a countable union of open sets of the form $\psi_E(\mathcal{U})$ where $\mathcal{U}=I\times\mathcal{O}$ is an open rectangle in $(0,\infty)\times S$.
\end{lemma}

\begin{proof}
Let $\{r_k\}_{k=1}^\infty$ and $\{\eta_j\}_{j=1}^\infty$ be countably dense subsets of $(0,\infty)$ and $S$, respectively.  For each triple of natural numbers $j,l,n\in\mathbb{N}_+$, consider the open set
\begin{equation*}
\mathcal{U}_{j,l,n}=\{ \vert r - r_j \vert < 1/n \}\times \mathcal{O}_{l,n}\subseteq (0,\infty)\times S
\end{equation*}
where
\begin{equation*}
\mathcal{O}_{l,n}=\{\eta\in S: |\eta-\eta_l|<1/n\}.
\end{equation*}
Let $U\subseteq \mathbb{R}^d\setminus \{0\}$ be open. We will show that
\begin{equation}\label{eq:OpenRectangle}
U=\bigcup_{\substack{j,l,n\\ \psi_E(\mathcal{U}_{j,l,n})\subseteq U}}\psi_E(\mathcal{U}_{j,l,n}),
\end{equation}
where, in view of Proposition \ref{prop:PsiHomeomorphism}, each $\psi_E(\mathcal{U}_{j,l,n})$ is open. It is clear that any element of the union on the right hand side of \eqref{eq:OpenRectangle} belongs to some $\psi_E(\mathcal{U}_{j,l,n}) \subseteq U$ and so the union is a subset of $U$. To prove \eqref{eq:OpenRectangle}, it therefore suffices to prove that, for each $x\in U$, there exists a triple $j,l,n$ with
\begin{equation*}
x\in\psi_E(\mathcal{U}_{j,l,n})\subseteq U.
\end{equation*}
To this end, fix $x\in U$ and let $\delta>0$ be such that $\mathbb{B}_\delta(x)\subseteq U$. Consider $(r_x,\eta_x)=\psi_E^{-1}(x)\in (0,\infty)\times S$ and set $M=\|r_x^E\|>0$ and $C=\|E\|>0$. Observe that 
\begin{equation*}
\|I-\alpha^E\|=\left\|\sum_{k=1}^\infty \frac{(\ln \alpha)^k}{k!} E^k\right\|\leq \sum_{k=1}^\infty \frac{|\ln \alpha|^k}{k!} \|E\|^k=e^{(C|\ln \alpha|)}-1
\end{equation*}
for all $\alpha>0$. Since $\alpha\mapsto e^{(C|\ln \alpha|)}-1$ is continuous and $0$ at $\alpha=1$, we can choose $\delta'>0$ for which
\begin{equation*}
\|I-\alpha ^E\|< \frac{\delta}{2M (  |\eta_x|+2)}
\end{equation*}
whenever $|\alpha-1|<\delta'$. Fix an integer $n>\max \left\{1/\delta'r_x, 4M/\delta \right\}$ and, using the density of the collections $\{r_j\}$ and $\{\eta_l\}$, let $r_j$ and $ \eta_l$ be such that $\abs{r_j - r_x } < 1/n$ and $\abs{\eta_l - \eta_x} < 1/n$. It follows that the corresponding open set $\mathcal{U}_{j,l,n}$ contains $\psi_E^{-1}(x)$, or, equivalently, $x\in \psi_E(\mathcal{U}_{j,l,n})$. Thus, it remains to show that $\psi_E(\mathcal{U}_{j,l,n}) \subseteq \mathbb{B}_\delta(x)$. To this end, let $y=\psi_E(r_y,\eta_y)\in\psi_E(\mathcal{U}_{j,l,n})$ and observe that
\begin{eqnarray*}
| x - y | &\leq& \vert \psi_E(r_x,\eta_x,) - \psi_E(r_x,\eta_y) \vert 
    + \vert \psi_E(r_x,\eta_y) - \psi_E(r_y,\eta_y) \vert\\
    &=&  \vert r_x^E (\eta_x - \eta_y) \vert + \vert (r_x^E - r_y^E) \eta_y \vert\\
    &\leq& M\vert \eta_x - \eta_y \vert + \|{r_x^E - r_y^E}\|  \vert \eta_y \vert.
\end{eqnarray*}
Since both $(\eta_x,r_x),(\eta_y,r_y) \in \mathcal{U}_{j,l,n}$, we have
\begin{equation*}
    \vert \eta_x - \eta_y \vert \leq \vert \eta_x - \eta_j \vert + \vert \eta_j - \eta_y \vert < \frac{2}{n}
\hspace{1cm}\mbox{and}\hspace{1cm}
    \vert \eta_y \vert \leq \vert \eta_y - \eta_x \vert + \vert \eta_x \vert < \vert \eta_x \vert + \frac{2}{n}.
\end{equation*}
Also, since $|r_x-r_y|<1/n$, it follows that $r_y=\alpha r_x$ where $|1-\alpha|<1/nr_x < \delta'$ by our choice of $n$. Consequently,
\begin{eqnarray*}
    \vert x - y \vert 
    &< & \frac{2}{n} M+ \left( \vert \eta_x \vert + \frac{2}{n} \right) \|{r_x^E -  r_x^E \alpha^E}\|   \\ 
    &<& \frac{2}{n}M + \left( \vert \eta_x \vert + 2 \right)M\| I - \alpha^E\| \\
    &<&  \frac{2}{n}M +  \frac{\delta M \left( \vert \eta_x \vert + 2\right) }{2M (| \eta_x | + 2)}  \\
    &<& \frac{\delta}{2} + \frac{\delta}{2}=\delta 
\end{eqnarray*}
and so we have established \eqref{eq:OpenRectangle}. Finally, upon noting that $\{\mathcal{U}_{j,l,n}\}_{(j,l,n)\in\mathbb{N}_+^3}$ is a countable collection of open rectangles, the union in \eqref{eq:OpenRectangle} is necessarily countable and we are done with the proof.
\end{proof}

\noindent In our final lemma preceding the proof of Theorem \ref{thm:MainIntegrationFormula}, we treat a general measure-theoretic statement which gives sufficient conditions concerning two measure spaces to ensure that their completions are isomorphic. Though we suspect that this result is well-known, we present its proof for completeness.

\begin{lemma}\label{lem:PushforwardLemma}
Let $(X_1,\Sigma_1,\nu_1)$ and $(X_2,\Sigma_2,\nu_2)$ be measure spaces, let $\varphi:X_1\to X_2$ be a bijection and denote by $(X_i,\Sigma_i',\nu_i')$ the completion of the measure space $(X_i,\Sigma_i,\nu_i)$ for $i=1,2$. Assume that the following two properties are satisfied:
\begin{enumerate}
\item\label{property:PushforwardLemma1} For each $A_1\in\Sigma_1$, $\varphi(A_1)\in\Sigma_2'$ and $\nu_2'(\varphi(A_1))=\nu_1(A_1).$
\item\label{property:PushforwardLemma2} For each $A_2\in\Sigma_2$, $\varphi^{-1}(A_2)\in \Sigma_1'$ and $\nu_1'(\varphi^{-1}(A_2))=\nu_2(A_2)$.
\end{enumerate}
Then the measure spaces $(X_1,\Sigma_1',\nu_1')$ and $(X_2,\Sigma_2',\nu_2')$ are isomorphic with point isomorphism $\varphi$. Specifically,
\begin{equation}\label{eq:PushforwardLemma1}
\Sigma_2'=\{A_2\subseteq X_2: \varphi^{-1}(A_2)\in\Sigma_1'\}
\end{equation}
and
\begin{equation}\label{eq:PushforwardLemma2}
\nu_2'(A_2)=\nu_1'(\varphi^{-1}(A_2))
\end{equation}
for all $A_2\in\Sigma_2'$.
\end{lemma}
\begin{proof}
Let us first assume that $A_2\in\Sigma_2'$. By definition, $A_2=G_2\cup H_2$ where $G_2\in\Sigma_2$ and $H_2\subseteq G_{2,0}\in \Sigma_2$ with $\nu_2'(A_2)=\nu_2(G_2)$ and $\nu_2'(H_2)=\nu_2(G_{2,0})=0$. Consequently, $\varphi^{-1}(A_2)=\varphi^{-1}(G_2)\cup\varphi^{-1}(H_2)$ and $\varphi^{-1}(H_2)\subseteq \varphi^{-1}(G_{2,0})$. In view of Property \ref{property:PushforwardLemma2}, $\varphi^{-1}(G_2),\varphi^{-1}(G_{2,0})\in \Sigma_1'$ and we have
\begin{equation*}
\nu_1'(\varphi^{-1}(G_2))=\nu_2(G_2)=\nu_2'(A_2)\hspace{1cm}\mbox{and}\hspace{1cm}\nu_1'(\varphi^{-1}(G_{2,0}))=\nu_2(G_{2,0})=0.
\end{equation*}
In view of the fact that $(X_1',\Sigma_1',\nu_1')$ is complete, $\varphi^{-1}(H_2)\in\Sigma_1'$ and $\nu_1'(\varphi^{-1}(H_2))=0$. Consequently, we obtain $\varphi^{-1}(A_2)=\varphi^{-1}(G_2)\cup\varphi^{-1}(H_2)\in\Sigma_1'$ and
\begin{equation*}
\nu_2'(A_2)=\nu_1'(\varphi^{-1}(G_2))\leq\nu_1'(\varphi^{-1}(A_2))\leq\nu_1'(\varphi^{-1}(G_2))+\nu_1'(\varphi^{-1}(H_2))=\nu_2(G_2)+0=\nu_2'(A_2).
\end{equation*}
From this we obtain that $\Sigma_2'\subseteq \{A_2\subseteq X_2:\varphi^{-1}(A_2)\in\Sigma_1'\}$ and, for each $A_2\in\Sigma_2'$, $\nu_2'(A_2)=\nu_1'(\varphi^{-1}(A_2))$. It remains to prove that
\begin{equation*}
\{A_2\subseteq X_2:\varphi^{-1}(A_2)\in\Sigma_1'\}\subseteq \Sigma_2'.
\end{equation*}
To this end, let $A_2$ be a subset of $X_2$ for which $\varphi^{-1}(A_2)\in\Sigma_1'$. By the definition of $\Sigma_1'$, we have $\varphi^{-1}(A_2)=G_1\cup H_1$ where $G_1\in\Sigma_1$, $H_1\subseteq G_{1,0}\in\Sigma_1$ and $\nu_1'(H_1)=\nu_1(G_{1,0})=0$. In view of Property \ref{property:PushforwardLemma1}, $\varphi(G_1)\in\Sigma_2'$, $\varphi(H_1)\subseteq\varphi(G_{1,0})\in\Sigma_2'$ and $\nu_2'(\varphi(G_{1,0}))=\nu_1(G_{1,0})=0$. Because $(X_2',\Sigma_2',\nu_2')$ is complete, we have $\varphi(H_1)\in\Sigma_2'$ and so
\begin{equation*}
A_1=\varphi(\varphi^{-1}(A_2))=\varphi(G_1)\cup\varphi(H_1)\in \Sigma_2',
\end{equation*}
as desired.
\end{proof}

\noindent We are finally in a position to prove Theorem \ref{thm:MainIntegrationFormula}.

\begin{proof}[Proof of Theorem \ref{thm:MainIntegrationFormula}]
Denote by $\mathcal{C}$ the collection of sets $G\subseteq (0,\infty)\times S$ for which $\psi_E(G)\in \mathcal{M}_d$ and $m(\psi_E(G))=(\lambda_P\times\sigma_{P,E})(G).$ By virtue of Lemma \ref{lem:AllMeasurableRectangles}, it follows that $\mathcal{C}$ contains all elementary sets, i.e., finite unions of disjoint measurable rectangles. Using the continuity of measure (applied to the measures $m$ and $\lambda_P\times\sigma_{P,E}$) and the fact that $\psi_E$ is a bijection, it is straightforward to verify that $\mathcal{C}$ is a monotone class. By the monotone class lemma (Theorem 8.3 of \cite{rudin_real_1987}), it immediately follows that $\mathcal{L}\times\Sigma_{P,E}\subseteq\mathcal{C}$. In other words, for each $G\in\mathcal{L}\times\Sigma_{P,E}$,
\begin{equation}\label{eq:Good1}
\psi_E(G)\in\mathcal{M}_d\hspace{1cm}\mbox{and}\hspace{1cm}m(\psi_E(G))=(\lambda_P\times\sigma_{P,E})(G).
\end{equation}
We claim that, for each Borel subset $A$ of $\mathbb{R}^d\setminus\{0\}$, $\psi_E^{-1}(A)\subseteq \mathcal{L}\times\Sigma_{P,E}$. To this end, we write
\begin{equation*}
\psi_E(\mathcal{L}\times\Sigma_{P,E})=\{\psi_E(G):G\in\mathcal{L}\times\Sigma_{P,E}\}
\end{equation*}
for the $\sigma$-algebra on $\mathbb{R}^d\setminus\{0\}$ induced by $\psi_E$. In view of Lemma \ref{lem:OpenRectangle}, $\psi_E(\mathcal{L}\times\Sigma_{P,E})$ contains every open subset of $\mathbb{R}^d\setminus\{0\}$ and therefore
\begin{equation*}
\mathcal{B}(\mathbb{R}^d\setminus\{0\})\subseteq\psi_E(\mathcal{L}\times\Sigma_{P,E}).
\end{equation*}
where $\mathcal{B}(\mathbb{R}^d\setminus\{0\})$ denotes the $\sigma$-algebra of Borel subsets of $\mathbb{R}^d\setminus\{0\}$ thus proving our claim. 

Together, the results of the two preceding paragraphs show that, for each $A\in\mathcal{B}(\mathbb{R}^d\setminus\{0\})$, $\psi_E^{-1}(A)\subseteq \mathcal{L}\times\Sigma_{P,E}$ and $m(A)=(\lambda_P\times\sigma_{P,E})(\psi_E^{-1}(A))$. Upon noting that $\mathcal{L}\times\Sigma_{P,E}\subseteq (\mathcal{L}\times\Sigma_{P,E})'$, we immediately obtain the following statement: For each $A\in\mathcal{B}(\mathbb{R}^d\setminus\{0\})$,
\begin{equation}\label{eq:Good2}
\psi_E^{-1}(A)\in (\mathcal{L}\times\Sigma_{P,E})'\hspace{1cm}\mbox{and}\hspace{1cm}m(A)=(\lambda_P\times\sigma_{P,E})(\psi_E^{-1}(A)).
\end{equation}
In comparing \eqref{eq:Good1} and \eqref{eq:Good2} with Properties \ref{property:PushforwardLemma1} and \ref{property:PushforwardLemma2} of Lemma \ref{lem:PushforwardLemma} and, upon noting that $((0,\infty)\times S,(\mathcal{L}\times\Sigma_{P,E})',\lambda_P\times\sigma_{P,E})$ is the completion of $((0,\infty)\times S,\mathcal{L}\times\Sigma_{P,E},\lambda_P\times\sigma_{P,E})$ and $(\mathbb{R}^d\setminus\{0\},\mathcal{M}_d,m)$ is the completion of $(\mathbb{R}^d\setminus\{0\},\mathcal{B}(\mathbb{R}^d\setminus\{0\}),m)$, Property \ref{item:MainIntegrationFormula1} of Theorem \ref{thm:MainIntegrationFormula} follows immediately from Lemma \ref{lem:PushforwardLemma}.

It remains to prove Property \ref{item:MainintegrationFormula2}. To this end, let $f:\mathbb{R}^d\to\mathbb{C}$ be Lebesgue measurable. Because $\mathcal{M}_d=\{A\subseteq \mathbb{R}^d\setminus\{0\}:\psi_E^{-1}(A)\in(\mathcal{L}\times\Sigma_{P,E})'\}$, it follows that $f\circ\psi_E$ is $(\mathcal{L}\times\Sigma_{P,E})'$-measurable. In the case that $f\geq 0$, we may approximate $f$ monotonically by simple functions and, by invoking Property \ref{item:MainIntegrationFormula1} and the monotone convergence theorem, we find that
\begin{equation}\label{eq:ChangeofMeasure}
\int_{\mathbb{R}^d}f(x)\,dx=\int_{\mathbb{R}^d\setminus \{0\}}f(x)\,dx=\int_{(0,\infty)\times S}f\circ \psi_E\, d(\lambda_P\times\sigma_{P,E}).
\end{equation}
From this, \eqref{eq:MainIntegrationFormula} follows from Fubini's theorem (see, e.g., Part (a) of Theorem 8.8 and Theorem 8.12 of \cite{rudin_real_1987}). Finally, by applying the above result to $|f|\geq 0$, we obtain $f\in L^1(\mathbb{R}^d)$ if and only if $f\circ \psi_E\in L^1((0,\infty)\times S,(\mathcal{L}\times\Sigma_{P,E})',\lambda_P\times\sigma_{P,E})$. In this case, by applying \eqref{eq:ChangeofMeasure} to $\Re(f)_+,\Re(f)_-,\Im(f)_+$ and $\Im(f)_-$, we find that \eqref{eq:ChangeofMeasure} holds for our integrable $f$ and, by virtue Fubini's theorem (see, e.g., Part (b) of Theorem 8.8 and Theorem 8.12 of \cite{rudin_real_1987}), the desired result follows.
\end{proof}

\noindent Our next result, Proposition \ref{prop:Regular}, guarantees that, in particular, $\sigma_{P,E}$ is a Radon measure. 

\begin{proposition}\label{prop:Regular}
We have:
\begin{enumerate}
    \item\label{item:Complete} $(S,\Sigma_{P,E},\sigma_{P,E})$ is the completion of the measure space $(S,\mathcal{B}(S),\sigma_{P,E})$. In particular, the measure space $(S,\Sigma_{P,E},\sigma_{P,E})$ is complete and every $F\in \Sigma_{P,E}$ is of the form $F=G\cup H$ where $G$ is a Borel set and $H$ is a subset of a Borel set $Z$ with $\sigma_{P,E}(Z)=0$.
\item\label{item:Regular} For each $F\in\Sigma_{P,E}$,
\begin{equation}\label{eq:OuterRegular}
\sigma_{P,E}(F)=\inf\{\sigma_{P,E}(\mathcal{O}):F\subseteq\mathcal{O}\subseteq S\mbox{ and $\mathcal{O}$ is open}\}
\end{equation}
and
\begin{equation}
\sigma_{P,E}(F)=\sup\{\sigma_{P,E}(K):K\subseteq F\subseteq S\mbox{ and $K$ is compact}\}.
\end{equation}
\end{enumerate} 
\end{proposition}
\begin{remark}
This proposition can be seen as an application of Proposition \ref{prop:BorelContainment} and Theorem 2.18 of \cite{rudin_real_1987}. The proof we give here is distinct and, we believe, nicely illustrates the utility of \eqref{eq:MainIntegrationFormula} of Theorem \ref{thm:MainIntegrationFormula}.
\end{remark}
\begin{proof}
Throughout the proof, we shall write $\sigma=\sigma_{P,E}$, $\Sigma=\Sigma_{P,E}$ and, for each $F\subseteq S$, $\widetilde{F}=\widetilde{F_E}$. We remark that, by standard arguments using $G_\delta$ and $F_\sigma$ sets, Item \ref{item:Complete} follows immediately from Item \ref{item:Regular}. Also, given that $S$ is compact and $\sigma$ is finite, it suffices to prove \eqref{eq:OuterRegular}, i.e., it suffices to prove the statement: For each $F\in \Sigma$ and $\epsilon>0$, there is an open subset $\mathcal{O}$ of $S$ containing $F$ for which 
\begin{equation*}
\sigma(\mathcal{O}\setminus F)<\epsilon.
\end{equation*}
To this end, let $F\in \Sigma$ and $\epsilon>0$. Given that $\widetilde{F}\in\mathcal{M}_d$ and $m$ is outer regular, there exists an open set $U\subseteq \mathbb{R}^d\setminus\{0\}$ for which $\widetilde{F}\subseteq U$ and $m(U\setminus\widetilde{F})<\epsilon/(2\mu_P)$. Since $\widetilde{F}$ is a subset of the open set $B\setminus\{0\}$, we may assume without loss of generality that $U\subseteq B\setminus\{0\}$ and so $m(\widetilde{F})\leq m(U)<\infty$ and
\begin{equation}\label{eq:LebesgueOuter}
m(U\setminus \widetilde{F})=m(U)-m(\widetilde{F})<\epsilon/(2\mu_P).
\end{equation}
For each $0<r<1$, consider the open set
\begin{equation*}
\mathcal{O}_r=S\cap\left( r^{-E}U\right)
\end{equation*}
in $S$. Observe that, for each $x\in F$, $r^E x\in \widetilde{F}\subseteq U$ and therefore $x\in \mathcal{O}_r$. Hence, for each $0<r<1$, $\mathcal{O}_r$ is an open subset of $S$ containing $F$. 

We claim that there is at least one $r_0\in (0,1)$ for which 
\begin{equation}\label{eq:GoodIneq}
m(\widetilde{\mathcal{O}_{r_0}})< m(U)+\epsilon/(2\mu_P).
\end{equation}
To prove the claim, we shall assume, to reach a contradiction, that 
\begin{equation*}
m(\widetilde{\mathcal{O}_{r}})\geq m(U)+\epsilon/(2\mu_P)
\end{equation*}
for all $0<r<1$. By virtue of \eqref{eq:MainIntegrationFormula} of Theorem \ref{thm:MainIntegrationFormula},
\begin{equation*}
m(U)=\int_{0}^\infty\left(\int_S \chi_{U}(r^E\eta)\,\sigma(d\eta)\right)r^{\mu_P-1}\,dr.
\end{equation*}
Upon noting that $U\subseteq B\setminus\{0\}$, it is easy to see that
\begin{equation*}
U=\bigcup_{0<s<1}s^E\mathcal{O}_s
\hspace{1cm}\mbox{and}\hspace{1cm}
r^E\eta\in \bigcup_{0<s<1}s^E\mathcal{O}_s
\end{equation*}
if and only if $0<r<1$ and $\eta\in \mathcal{O}_r$. Consequently,
\begin{equation*}
m(U)=\int_0^1\left(\int_S\chi_{\mathcal{O}_r}(\eta)\,\sigma(d\eta)\right)\,r^{\mu_P-1}\,dr=\int_0^1\sigma(\mathcal{O}_r)r^{\mu_P-1}\,dr=\int_0^1 \mu_P\cdot m(\widetilde{\mathcal{O}_r})\,r^{\mu_P-1}\,dr.
\end{equation*}
Upon making use of our supposition, we have
\begin{equation*}
\int_0^1\mu_P\cdot m(\widetilde{\mathcal{O}_r})r^{\mu_P-1}\,dr\geq \int_0^1\mu_P\cdot (m(U)+\epsilon/(2\mu_P))r^{\mu_P-1}\,dr=m(U)+\epsilon/(2\mu_P)
\end{equation*}
and so
\begin{equation*}
m(U)\geq m(U)+\epsilon/(2\mu_P),
\end{equation*}
which is impossible. Thus, the stated claim is true.

Given any such $r_0$ for which \eqref{eq:GoodIneq} holds, set $\mathcal{O}=\mathcal{O}_{r_0}$. As previously noted, $\mathcal{O}$ is an open subset of $S$ which contains $F$. In view of \eqref{eq:LebesgueOuter} and \eqref{eq:GoodIneq}, we have
\begin{equation*}
m(\widetilde{\mathcal{O}})-m(\widetilde{F})<m(U)-m(\widetilde{F})+\epsilon/(2\mu_P)<\epsilon/(2\mu_P)+\epsilon/(2\mu_P)=\epsilon/\mu_P
\end{equation*}
and therefore
\begin{equation*}
\sigma(\mathcal{O}\setminus F)=\sigma(\mathcal{O})-\sigma(F)=\mu_P(m(\widetilde{\mathcal{O}})-m(\widetilde{F}))<\epsilon,
\end{equation*}
as desired.
\end{proof}

\subsection{The construction is independent of $E\in\Exp(P)$.}\label{subsec:IndependentofE}

\noindent In this subsection, we show that the Radon measure $\sigma_{P,E}$ is independent of the choice of $E\in\Exp(P)$ and complete the proof of Theorem \ref{thm:BestIntegrationFormula}. To set the stage for our first result, let $E_1,E_2\in\Exp(P)$ and consider the associated (respective) measure spaces $(S,\Sigma_{P,E_1},\sigma_{P,E_1})$ and $(S,\Sigma_{P,E_2},\sigma_{P,E_2})$ produced via the construction in Subsection \ref{subsec:ConstructionofSigma}. 

\begin{proposition}\label{prop:Endependence}
These measure spaces are the same, i.e., $\Sigma_{P,E_1}=\Sigma_{P,E_2}$ and $\sigma_{P,E_1}=\sigma_{P,E_2}$.
\end{proposition}
\begin{proof}
Throughout the proof, we will write $\Sigma_i=\Sigma_{P,E_i}$ and $\sigma_i=\sigma_{P,E_i}$ for $i=1,2$. In view of the Proposition \ref{prop:Regular}, it suffices to show that 
\begin{equation*}
\sigma_1(F)=\sigma_2(F)
\end{equation*}
for all $F\in \mathcal{B}(S)\subseteq \Sigma_{1}\cap\Sigma_{2}$. To this end, we let $F\in\mathcal{B}(S)$ be arbitrary but fixed. 

Given $n\in\mathbb{N}_+$, using the regularity of the measures $\sigma_1$ and $\sigma_2$, select open sets $\mathcal{O}_{n,1},\mathcal{O}_{n,2}$ and compact sets $K_{n,1},K_{n,2}$ for which
\begin{equation*}
K_{n,j}\subseteq F\subseteq \mathcal{O}_{n,j}\hspace{1cm}\mbox{and}\hspace{1cm}\sigma_j(\mathcal{O}_{n,j}\setminus K_{n,j})<1/n
\end{equation*}
for $j=1,2$. Observe that $K_n=K_{n,1}\cup K_{n,2}$ is a compact set, $\mathcal{O}_n=\mathcal{O}_{n,1}\cap\mathcal{O}_{n,2}$ is an open set and $K_n\subseteq F\subseteq \mathcal{O}_n$. Furthermore, 
\begin{equation*}
\sigma_j(\mathcal{O}_n\setminus K_n)\leq \sigma_j(\mathcal{O}_{n,j}\setminus K_{n,j})<1/n
\end{equation*}
for $j=1,2$. Given that $\mathcal{O}_n$ is open in $S$, $\mathcal{O}_n=S\cup U_n$ where $U_n$ is an open subset of $\mathbb{R}^d$ and, because that $S$ is compact, $K_n=K_n\cap S$ is a compact subset of $\mathbb{R}^d$. By virtue of Urysohn's lemma, let $\phi_n:\mathbb{R}^d\to [0,1]$ be a continuous function which is compactly supported in $U_n$ and for which $\phi_n(x)=1$ for all $x\in K_n$. Using this sequence of functions $\{\phi_n\}$, we establish the following useful lemma.

\begin{lemma}\label{lem:IndepProof}
For $j=1,2$ and $n\in\mathbb{N}$, define $g_{n,j}:(0,\infty)\to\mathbb{R}$ by
\begin{equation*}
g_{n,j}(r)=\int_S\phi_n(r^{E_j}\eta)\,\sigma_j(d\eta).
\end{equation*}
for $r>0$. Then $g_{n,j}$ is continuous for each $n\in\mathbb{N}$ and $j=1,2$ and
\begin{equation*}
    \sigma_j(F)=\lim_{n\to\infty}g_{n,j}(1)
\end{equation*}
for $j=1,2$.
\end{lemma}
\begin{subproof}
First, we note that, for each $r\in (0,\infty)$, the above integral makes sense because $\eta\mapsto \phi_n(r^{E_j}\eta)$ is Borel measurable (because it's continuous on $S$) and non-negative. Let $\epsilon>0$ and $r_0\in (0,\infty)$ be arbitrary but fixed. It is clear that the function $(0,\infty)\times S\ni (r,\eta)\mapsto \phi_n(r^{E_j}\eta)$ is continuous on its domain and therefore, in view of the compactness of $S$, we can find a $\delta>0$ for which
\begin{equation*}
|\phi_n(r^{E_j}\eta)-\phi_n(r_0^{E_j}\eta)|\leq\frac{\epsilon}{2\sigma_j(S)}\hspace{1cm}\mbox{whenever}\hspace{1cm}|r-r_0|<\delta
\end{equation*}
for all $\eta\in S$. The triangle inequality guarantees that
\begin{equation*}
|g_{n,j}(r)-g_{n,j}(r_0)|\leq \int_S|\phi_n(r^{E_j}\eta)-\phi_n(r_0^{E_j}\eta)|\,\sigma_j(d \eta)\leq\epsilon/2<\epsilon
\end{equation*}
whenever $|r-r_0|<\delta$. Thus, $g_{n,j}$ is continuous.

We observe that
\begin{equation*}
g_{n,j}(1)=\int_{S}\phi_n(\eta)\,\sigma_j(d\eta)
\end{equation*}
because $1^{E_j}=I$. By construction, we have $\chi_{K_n}(\eta)\leq\phi_{n}(\eta)\leq \chi_{\mathcal{O}_n}(\eta)$ for all $\eta\in S$ and $n\in\mathbb{N}_+$ and therefore
\begin{equation*}
\sigma_j(K_n)\leq g_{n,j}(1)\leq \sigma_j(\mathcal{O}_n)
\end{equation*}
by the monotonicity of the integral. Since
\begin{equation*}
\sigma_j(F)=\lim_{n\to\infty}\sigma_j(K_n)=\lim_{n\to\infty}\sigma_j(\mathcal{O}_n)
\end{equation*}
in view of our choice of $\mathcal{O}_n$ and $K_n$, the remaining result follows immediately from the preceding inequality (and the squeeze theorem).
\end{subproof}

\noindent Let us now complete the proof of Proposition \ref{prop:Endependence}. Given any $0<s<1< t$ and $n\in\mathbb{N}$, consider the function $f=f_{n,s,t}:\mathbb{R}^d\to [0,1]$ given by
\begin{equation*}
f(x)=\phi_n(x)\chi_{[s,t]}(P(x))
\end{equation*}
for $x\in\mathbb{R}^d$. It is clear that $f$ is Lebesgue measurable on $\mathbb{R}^d$ and non-negative. By virtue of Theorem \ref{thm:MainIntegrationFormula} (applied to the two measures $\sigma_1$ and $\sigma_2$), we have
\begin{equation}\label{eq:SameMeasure1}
\int_0^\infty \int_Sf(r^{E_1}\eta)\,\sigma_1(d\eta)r^{\mu_P-1}\,dr=\int_{\mathbb{R}^d}f(x)\,dx=\int_0^\infty \int_Sf(r^{E_2}\eta)\,\sigma_2(d\eta)r^{\mu_P-1}\,dr
\end{equation}
Upon noting that
\begin{equation*}
f(r^{E_j}\eta)=\phi_n(r^{E_j}\eta)\chi_{[s,t]}\left((P(r^{E_j}\eta)\right)=\phi_n(r^{E_j}\eta)\chi_{[s,t]}(rP(\eta))=\chi_{[s,t]}(r)\phi_n(r^{E_j}\eta)
\end{equation*}
for $r\in (0,\infty)$, $\eta\in S$, and $j=1,2$, we have
\begin{equation*}
\int_0^\infty\int_S f(r^{E_j}\eta)\,\sigma_j(d\eta)r^{\mu_P-1}\,dr=\int_{[s,t]}\int_S\phi_n(r^{E_j}\eta)\,\sigma_j(d\eta) r^{\mu_P-1}\,dr=\int_{[s,t]}g_{n,j}(r)r^{\mu_P-1}\,dr
\end{equation*}
for $j=1,2$. By virtue of the Lemma \ref{lem:IndepProof}, $r\mapsto g_{n,j}(r)r^{\mu_P-1}$ is continuous and necessarily bounded on $[s,t]$ and so the final integral above can be interpreted as a Riemann integral. In this interpretation, we have
\begin{equation}\label{eq:SameMeasure2}
\int_s^tg_{n,j}(r)r^{\mu_P-1}\,dr=\int_0^\infty \int_S f(r^{E_j}\eta)\,\sigma_j(d\eta)r^{\mu_P-1}\,dr
\end{equation}
for $j=1,2$ and $0<s<1<t$. In view of \eqref{eq:SameMeasure1} and \eqref{eq:SameMeasure2}, we conclude that
\begin{equation*}
\int_s^t g_{n,1}(r)r^{\mu_P-1}\,dr=\int_s^t g_{n,2}(r)r^{\mu_P-1}\,dr
\end{equation*}
for all $0<s<1<t$. In view of continuity of the integrands, an application of the fundamental theorem of calculus now guarantees that $g_{n,1}(1)=g_{n,2}(1)$ for each $n\in\mathbb{N}$. Therefore
\begin{equation*}
\sigma_1(F)=\lim_{n\to\infty}g_{n,1}(1)=\lim_{n\to\infty}g_{n,2}(1)=\sigma_2(F)
\end{equation*}
by virtue of Lemma \ref{lem:IndepProof}.
\end{proof}

\noindent In view of Proposition \ref{prop:Endependence}, we will denote by $\Sigma_P$ and $\sigma_P$ the unique $\sigma$-algebra and measure on $S$ which, respectively, satisfy
\begin{equation*}
    \Sigma_P=\Sigma_{P,E}\hspace{1cm}\mbox{and}\hspace{1cm}\sigma_P=\sigma_{P,E}
\end{equation*}
for all $E\in\Exp(P)$. We will henceforth assume this notation.

\begin{proposition}\label{prop:SymInvariance}
For any $O\in\Sym(P)$ and $F\in\Sigma_P$, $OF\in\Sigma_P$ and 
\begin{equation*}
\sigma_P(O F)=\sigma_P(F).
\end{equation*} 
That is, the measure $\sigma_P$ is invariant under the action by $\Sym(P)$. 
\end{proposition}

\begin{proof}
Let $O\in\Sym(P)$, $F\in\Sigma_P$ and, for $E\in \Exp(P)$, define $E'=O^* EO$. In view of Proposition \ref{prop:ExpP}, we note that $E'\in \Exp(P)$. Observe that
\begin{equation}\label{eq:TildeConjugation}
    \widetilde{(OF)_E}=\bigcup_{0<r<1}r^E (OF)=\bigcup_{0<r<1}O\left(O^* r^E O F\right)=O\left(\bigcup_{0<r<1} r^{E'}F\right)=O \widetilde{F_{E'}}
\end{equation}
thanks to Proposition \ref{prop:ContinuousGroupProperties}.
In view of Proposition \ref{prop:Endependence}, we have $O\widetilde{F_{E'}}\in \mathcal{M}_d$ because $F\in \Sigma_P=\Sigma_{P,E'}$ and $O$ is linear. Using \eqref{eq:TildeConjugation}, we find that $\widetilde{(OF)_E}\in\mathcal{M}_d$ and therefore  $OF\in\Sigma_{P,E}=\Sigma_P$ by virtue of Proposition \ref{prop:Endependence}. In view of the fact that $O$ is orthogonal,
\begin{equation*}
\sigma_{P,E}(OF)=\mu_P\cdot m\left(\widetilde{(OF)_E}\right)=\mu_P\cdot m\left(O \widetilde{F_{E'}}\right)=\mu_P\cdot m\left(\widetilde{F}_{E'}\right)=\sigma_{P,E'}(F)
\end{equation*}
and therefore, a final appeal to Proposition \ref{prop:Endependence} guarantees that
\begin{equation*}
    \sigma_P(OF)=\sigma_{P,E}(OF)=\sigma_{P,E'}(F)=\sigma_P(F),
\end{equation*}
as desired.
\end{proof}

\begin{proof}[Proof of Theorem \ref{thm:BestIntegrationFormula}]
Together, the results of Propositions \ref{prop:Regular}, \ref{prop:Endependence} and \ref{prop:SymInvariance}, guarantee that $\sigma_P$ is a Radon measure satisfying Properties \ref{property:Completion} and \ref{property:Invariance}. Property \ref{property:DefiningConditionofsigma} follows directly from Proposition \ref{prop:Endependence} and the definition of $\sigma_P$ in terms of $\sigma_{P,E}$ for any $E\in\Exp(P)$. Similarly, Properties \ref{property:BestPointIsomorphism} and \ref{property:BestIntegrationFormula} follow from Theorem \ref{thm:MainIntegrationFormula} by virtue of Proposition \ref{prop:Endependence}. 
\end{proof}

\section{Using a smooth structure on $S$ to compute $\sigma$.}\label{sec:SigmaForSmoothP}

\noindent In this section, we shall study the special case in which a positive homogeneous function $P$ on $\mathbb{R}^d$ is smooth\footnote{Many of the results in this section remain valid (with appropriate modification) under the weaker assumption that $P\in C^k(\mathbb{R}^d)$ for $k=1,2,\dots$. In this setting, $S$ is easily seen to be a $C^k$ manifold. Because working in the smooth category is sufficient for our purposes, we shall not pursue the greater level of generality but invite the reader to do so.}$^{\mbox{\tiny{,}}}$\footnote{To avoid trivialities, we assume that $d>1$ throughout.}.
Under this additional assumption, we shall find that $\grad P$ is everywhere non-vanishing on $S$ and so $S$ is a smooth compact embedded hypersurface in $\mathbb{R}^d$.\\

\noindent We first set up some notation: For a smooth manifold $M$, we shall denote by $\mathcal{A}(M)$ its unique maximal atlas. Also on $M$, the collection of smooth vector fields is denoted by $\mathfrak{X}(M)$ and, for each $k=1,2,\dots$, the set of (smooth) differential $k$-forms on $M$ will be denoted by $\Omega^k(M)$.  In this section, we integrate non-smooth differential forms and, for the generality needed here, we shall refer the reader to \cite{naber_topology_2011} for background (Another perspective is given in \cite{amann_analysis_2009}). To this end, let us denote the Lebesgue $\sigma$-algebra of measurable sets on $M$ by $\mathcal{L}(M)$. We note that $F\in\mathcal{L}(M)$ if and only if, \begin{equation*}
    \varphi(F\cap \mathcal{U})\in\mathcal{M}_d
\end{equation*}
for every chart $(\varphi,\mathcal{U})\in\mathcal{A}(M)$. An $n$-form $\omega$ on $M$, with $n=\dim(M)$, is said to be (Lebesgue) measurable, if in each coordinate system $(\varphi,\mathcal{U})$, the local representation
\begin{equation*}
    \omega=h_{\varphi}(x)dx^1\wedge dx^2\wedge \cdots\wedge dx^n
\end{equation*}
in the coordinates $\varphi=(x^1,x^2,\dots,x^n)$ has $h_{\varphi}(x)$ a Lebesgue measurable function on $U=\varphi(\mathcal{U})\subseteq\mathbb{R}^n$. The collection of measurable $n$-forms on $M$ is denoted by $\mathcal{L}(\Lambda^d(M))$ and, naturally, $\Omega^{d}(M)\subseteq \mathcal{L}(\Lambda^d(M))$. As standard, we shall use Einstein's summation convention without explicit mention.\\

\noindent We view $\mathbb{R}^d$ as smooth oriented Riemannian manifold with its standard Euclidean metric $\overline{g}$, oriented smooth atlas $\mathcal{A}_+(\mathbb{R}^d)$, and Riemannian volume form $ d\Vol_{\mathbb{R}^d}$. Given any $E\in\Exp(P)$, consider $\mathcal{E}_E\in \mathfrak{X}(\mathbb{R}^d)$ defined, at each $x\in\mathbb{R}^d$, by
\begin{equation*}
    (\mathcal{E}_E)_x(f)=\frac{d}{dt}f(x+t(Ex))\big\vert_{t=0}\hspace{1cm}
\end{equation*}
for $f\in C^\infty(\mathbb{R}^d)$. In the standard (global) chart with coordinates $x=(x^1,x^2,\dots,x^d)$, $(\mathcal{E}_E)_{x}\in T_{x}(\mathbb{R}^d)$ is given by
\begin{equation*}
    (\mathcal{E}_E)_{x}=(Ex)^{\alpha}\partial_{x^\alpha}=E^\alpha_\beta x^{\beta}\partial_{x^\alpha}
\end{equation*}
where $(E_\alpha^\beta)$ is the standard matrix representation for $E$ and  $\partial_{x^{\alpha}}=\partial/\partial x^\alpha$. By an abuse of notation, we shall write $\grad P$ to denote both the function
\begin{equation*}
\mathbb{R}^d\ni x\mapsto \grad P(x)=\left(\frac{\partial P}{\partial x^1},\frac{\partial P}{\partial x^2},\dots,\frac{\partial P}{\partial x^d}\right)\in\mathbb{R}^d,
\end{equation*}
where $\frac{\partial P}{\partial x^{\alpha}}=\frac{\partial P}{\partial x^{\alpha}}\vert_{x}$ for $\alpha=1,2\dots,d$, and its canonical identification $\grad P\in \mathfrak{X}(\mathbb{R}^d)$ given by \begin{equation*}
    \grad P_x=\overline{g}^{\alpha\beta}\frac{\partial P}{\partial x^{\alpha}}\partial_{x^{\beta}}=\delta^{\alpha\beta}\frac{\partial P}{\partial x^{\alpha}}\partial_{x^{\beta}}=\sum_{\alpha=1}^d\frac{\partial P}{\partial x^{\alpha}}\partial_{x^\alpha}
\end{equation*}
in standard Euclidean coordinates $x=(x^\alpha)$. Of course, for each $x\in\mathbb{R}^d$, the Riemannian norm $|\grad P_x |_{\overline{g}}$ of $\grad P_x \in T_x (\mathbb{R}^d)$ coincides with the Euclidean norm $| \grad P(x) |$ of $ \grad P(x) \in \mathbb{R}^d$. These equivalent quantities (functions) will be henceforth denoted by $|\grad P|$.

\begin{proposition}\label{prop:InnerProdIsOne}
For each $\eta\in S$, 
\begin{equation*}
    \overline{g}(\grad P,\mathcal{E})_{\eta}=\grad P(\eta)\cdot (E\eta)=1.
\end{equation*}
In particular, $\grad P$ (and $dP$) never vanishes on $S$ and so $S$ is a smooth compact embedded hypersurface in $\mathbb{R}^d$.
\end{proposition}
\begin{proof}
Given that $E\in\Exp(P)$ and $P\in C^\infty(\mathbb{R}^d)$, we differentiate the identity $rP(x)=P(r^Ex)$ to find that
\begin{equation*}
    P(x)=\frac{d}{dr}P(r^Ex)=\grad P(r^Ex)\cdot\left(\left(r^{E-I}E\right)x\right)
\end{equation*}
for $r>0$ and $x\in\mathbb{R}^d$. In particular, when $r=1$ and $x=\eta\in S$, we have
\begin{equation*}
    1=\grad P(\eta)\cdot \left(E\eta\right)=\frac{\partial P}{\partial x^\alpha}(E\eta)^\alpha=\overline{g}_{\alpha \beta}\left(\grad P_\eta\right)^{\alpha}(\mathcal{E}_\eta)^\beta=\overline{g}(\grad P,\mathcal{E})_\eta.
\end{equation*}
Thus, our (necessarily) compact level set $S$ is a smooth embedded hypersurface of $\mathbb{R}^d$ in view of the regular level set theorem.
\end{proof}
\noindent We shall denote by $\iota:  S \hookrightarrow \mathbb{R}^d$ the canonical inclusion map and set $d'=d-1$. As an embedded submanifold of $\mathbb{R}^d$, $S$ is a Riemannian submanifold of $\mathbb{R}^d$ with metric $g^S$ given by
\begin{equation*}
    g^S(X,Y)=\overline{g}(\iota_*(X),\iota_*(Y))
\end{equation*}
for $X,Y\in\mathfrak{X}(S)$; here, for each $\eta\in S$,  $\iota_*:T_\eta(S)\to T_\eta(\mathbb{R}^d)$ is the pushforward of $\iota$.  In view of the preceding proposition, $N:=\grad P/|\grad P|\in \mathfrak{X}(\mathbb{R}^d)$ is a smooth unit normal vector field along $S$ and it determines an orientation on the Riemannian manifold $S$. Equipped with this orientation, $(S,g^S)$ is an oriented Riemannian manifold and we shall denote by $d\Vol_S$ the Riemannian volume form and by $\mathcal{A}_+(S)$ its corresponding (maximal) oriented atlas. By virtue of Proposition 15.21 of \cite{lee_introduction_2003}, $d\Vol_S=(N\iprod d\Vol_{\mathbb{R}^d})\vert_S$, i.e.,
\begin{eqnarray*}
    d\Vol_S(X_1,X_2,\dots,X_{d'})&=&d\Vol_{\mathbb{R}^d}(N,\iota_*(X_1),\iota_*(X_2),\dots,\iota_*(X_{d'}))\\
    &=&\frac{1}{|\grad P|}d\Vol_{\mathbb{R}^d}(\grad P,\iota_*(X_1),\iota_*(X_2),\dots,\iota_*(X_{d'}))
\end{eqnarray*}
for any collection $\{X_1,X_2,\dots,X_{d'}\}\in \mathfrak{X}(S)$. Beyond $d\Vol_S\in \Omega^{d'}(S)$, we consider the following smooth $d'$-form(s): Given $E\in \Exp(P)$, define $d\sigma_{P,E}\in\Omega^{d'}(S)$ by
\begin{equation*}
    d\sigma_{P,E}(X_1,X_2,\dots,X_{d'})=d\Vol_{\mathbb{R}^d}(\mathcal{E}_E,\iota_*(X_1),\iota_*(X_2),\dots,\iota_*(X_{d'}))
\end{equation*}
for $X_1,X_2,\dots,X_{d'}\in\mathfrak{X}(S)$. We have
\begin{proposition}\label{prop:FormRiemannRelation}
For any $E\in\Exp(P)$,
\begin{equation*}
    d\sigma_{P,E}=\frac{1}{|\grad P|}d\Vol_S.
\end{equation*}
In particular, $d\sigma_{P,E}$ is positively oriented and is independent of $E\in\Exp(P)$.
\end{proposition}
\noindent Before proving the proposition, we first treat a lemma of a purely linear algebraic nature. 
\begin{lemma}\label{lem:determinants}
Let $v_1,v_2,\dots,v_{d'}$ be linearly independent vectors in $\mathbb{R}^d$ and suppose that $w\in\mathbb{R}^d \setminus\{0\}$ is such that $w\perp v_i$ for all $i$. Then, for any $z\in\mathbb{R}^d$ for which $z\cdot w=1$,
\begin{equation*}
\det(z, v_1,v_2,\dots,v_{d'})=\frac{1}{|w|}\det(n,v_1,v_2,\dots,v_{d'})=\frac{1}{|w|^2}\det(w,v_1,v_2,\dots,v_{d'}).
\end{equation*}
where $n:=w/|w|$.
\end{lemma}

\begin{proof}
Given $z\in\mathbb{R}^d$ such that $z\cdot w=1$, it follows that 
\begin{equation*}
z=\frac{1}{|w|}n+a_1v_1+a_2v_2+\cdots a_{d'}v_{d'}.
\end{equation*}
By the multilinearity of the determinant map, we have
\begin{eqnarray*}
\det(z,v_1,v_2,\dots,v_{d'}) &=&\det\lp \frac{1}{|w|}n+a_1v_2+a_2v_2+\cdots a_{d'} v_{d'},v_1,v_2,\dots,v_{d'}\rp\\
    &=&\frac{1}{|w|}\det(n,v_1,v_2,\dots,v_{d'})+\det(a_1v_1+\cdots+a_{d'} v_{d'}, v_1,v_2,\dots,v_{d'})\\
&=&\frac{1}{|w|}\det(n, v_1,v_2,\dots,v_{d'})+0
\end{eqnarray*}
where we have used the fact that the columns of the matrix $(a_1v_1+\cdots+a_{d'} v_{d'}, v_1,v_2,\dots,v_{d'})$ are linearly dependent to conclude that the final determinant is zero.
\end{proof}
\begin{proof}[Proof of Proposition \ref{prop:FormRiemannRelation}]
We fix $E\in\Exp(P)$ and note that the assertion at hand is a local one. Thus, it suffices to verify that, for any $\eta\in S$ and $X_1,X_2,\dots,X_{d'}\in T_\eta(S)$, 
\begin{equation*}
    d\sigma_{P,E}(X_1,X_2,\dots,X_{d'})=\frac{1}{|\grad P_{\eta}|}d\Vol_{\mathbb{R}^d}(N_\eta,\iota_*(X_1),\iota_*(X_2),\dots,\iota_*(X_{d'})).
\end{equation*}
Fix $\eta\in S$ and let $(\mathcal{O},\varphi)$ be a coordinate chart centered at $\eta$ with local coordinates $u=(u^{\alpha})$. As usual, denote by $x=(x^{\alpha})$ the Euclidean coordinates on $\mathbb{R}^d$.  For each $i=1,2,\dots,{d'}$, \begin{equation*}
X_i=X_i^\alpha \partial_{u^{\alpha}}\hspace{1cm}\mbox{and}\hspace{1cm}\iota_*(X_i)=v_i^\beta\partial_{x^{\beta}}
\end{equation*}
where
\begin{equation*}
v_i^\beta =X_i^\alpha\frac{\partial x^\beta}{\partial u^\alpha}.
\end{equation*}
For each $i=1,2,\dots,d'$, we set $v_i=(v_i^1,v_i^2,\dots,v_i^{d'})\in\mathbb{R}^d$. Also, let $w=\grad P(\eta)\in\mathbb{R}^d$ with $|w|=|\grad P(\eta)|=|\grad P_\eta|$, set $n=w/|w|\in\mathbb{R}^d$ and note that
\begin{equation*}
    N_\eta=\frac{1}{|\grad P_\eta|}\sum_{k=1}^d\frac{\partial P}{\partial x^{k}}\partial_{x^k}=n^{\mu}\partial_{x^\mu}.
\end{equation*}Given that $\grad P$ is normal to $S$, we have
\begin{equation*}
    v_i\cdot w=\overline{g}_{\mu,\nu}v_i^\mu w^\nu= \overline{g}(\iota_*(X_i),\grad P)_\eta=0
\end{equation*}
and therefore $w\perp v_i$ for each $i=1,2,\dots,{d'}$. Upon recalling that $(\mathcal{E}_E)_\eta=(E\eta)^\alpha\partial_{x^{\alpha}}$, set $z=((E\eta)^1,(E\eta)^2,\dots,(E\eta)^d)\in\mathbb{R}^d$ and observe that $z\cdot w=1$ by virtue of Proposition \ref{prop:InnerProdIsOne}. An appeal to the lemma guarantees that
\begin{eqnarray*}
d\sigma_{P,E}(X_1,X_2,\dots,X_{d'})&=&d\Vol_{\mathbb{R}^d}(\mathcal{E}_E,\iota_*(X_1),\iota_*(X_2),\dots,\iota_*(X_{d'}))\\
&=&\det(z,v_1,v_2,\dots,v_{d'})\\
&=&\frac{1}{|w|}\det(n,v_1,v_2,\dots,v_{d'})\\
&=&\frac{1}{|\grad P_\eta|}d\Vol_{\mathbb{R}^d}(N_\eta,\iota_*(X_1),\iota_*(X_2),\dots,\iota_*(X_{d'}))\\
&=&\frac{1}{|\grad P|}d\Vol_S(X_1,X_2,\dots,X_{d'}).
\end{eqnarray*}
\end{proof}
\noindent By virtue of the preceding proposition, we shall denote by $d\sigma_P$ the unique smooth $d'$-form on $S$ which satisfies
\begin{equation}\label{eq:sigmaForm}
    d\sigma_P=d\sigma_{P,E}=\frac{1}{|\grad P|}d\Vol_S
\end{equation}
for all $E\in\Exp(P)$. In this notation, we have this section's central result.

\begin{theorem}\label{thm:RiemannLebesgue}
Let $P$ be a smooth positive homogeneous function and let $S=\{\eta\in\mathbb{R}^d:P(\eta)=1\}$. Then $S$ is a compact smooth embedded hypersurface of $\mathbb{R}^d$. Viewing $\mathbb{R}^d$ as an oriented Riemannian manifold with its usual orientation and metric $\overline{g}$, $N=\grad P/|\grad P|$ is a smooth unit normal vector field along $S$. As a submanifold of $\mathbb{R}^d$, $S$ is a oriented Riemannian manifold of dimension $d'=d-1$ with its induced Riemannian metric $g^S$, volume form $d\Vol_S\in\Omega^{d'}(S)$ and orientation determined by $N$. The $\sigma$-algebras $\Sigma_P$ and $\mathcal{L}(S)$ on $S$ coincide and the smooth $d'$-form $d\sigma_P\in\Omega^{d'}(S)$, defined by \eqref{eq:sigmaForm}, coincides with the measure $\sigma_P$ in the sense that
\begin{equation}\label{eq:FormsAndMeasures}
\int_S g(\eta)\,\sigma_P(d\eta)=\int_S g\,d\sigma_P
\end{equation}
for all $g\in L^1(S,\Sigma_P,\sigma_P)$; here, the left hand side represents the Lebesgue integral of $g$ with respect to $\sigma_P$ and the right hand side is the integral of the measurable $d'$-form $g\, d\sigma_P$. Furthermore, the measure $\sigma_P$ and the canonical Riemannian volume measure $\Vol_S$ on $S$ are mutually absolutely continuous.
\end{theorem}

\noindent Before we are able to prove the above theorem, we first treat two lemmas which we will find useful in computation.

\begin{lemma}\label{lem:JacobianRelation}
Let $(\mathcal{U},\varphi)\in \mathcal{A}(S)$ and $E\in\Exp(P)$. Set $U=\varphi(\mathcal{U})\subseteq\mathbb{R}^{d'}$, $V=(0,1)\times U$ and define $\rho_{E,\varphi}:V\to \mathbb{R}^d$ and $h_{E,\varphi}:U\to \mathbb{R}$, respectively, by
\begin{equation*}
    \rho_{E,\varphi}(y)=\psi_E(r,\varphi^{-1})=r^E\varphi^{-1}(u)
\end{equation*}
for $y=(r,u)\in V$ and
\begin{equation*}
    h_{E,\varphi}(u)=\det\left.\left(E\varphi^{-1}(u)\right\vert D_u\varphi^{-1}(u)\right)
\end{equation*}
for $u\in U$; here, the vertical bar separates the first column of the (necessarily) $d\times d$ matrix from the rightmost $d\times d'$ submatrix and $D_u$ denotes the Jacobian in the coordinates $u=(u^1,u^2,\dots,u^{d'})\in U$. Then, $\rho_{E,\varphi}$ is a diffeomorphism onto its image $\rho_{E,\varphi}(V)=\widetilde{\mathcal{U}_E}$ and its Jacobian matrix $D\rho_{E,\varphi}$ has
\begin{equation}\label{eq:JacobianRelation1}
    \det(D\rho_{E,\varphi}(y))=r^{\mu_P-1}h_{E,\varphi}(u)
\end{equation}
for all $y=(r,u)\in V$. Furthermore, $h_{E,\varphi}$ is everywhere non-zero, smooth and
%\begin{equation}\label{eq:JacobianRelation2}
%    h_{E,\varphi}(u)=(d\sigma_P)_{\varphi^{-1}(u)}(\partial_{u^1},\partial_{u^2},\dots,\partial_{u^{d'}})
%\end{equation}
\begin{equation}\label{eq:JacobianRelation2}
    d\sigma_P=h_{E,\varphi}(u)\,du^1\wedge du^2\wedge\cdots\wedge du^{d'}
\end{equation}
in the coordinates $u=(u^1,u^2,\dots,u^{d'})\in U$. 
\end{lemma}
\begin{proof}
The map $\rho_{E,\varphi}$ is smooth because $\varphi$ is smooth. By virtually the same argument made in the proof of Proposition \ref{prop:PsiHomeomorphism}, which here uses the fact that $P\in C^\infty(\mathbb{R}^d)$, we conclude that $\rho_{E,\varphi}$ is a diffeomorphism onto 
\begin{equation*}\rho_{E,\varphi}(V)=\bigcup_{0<r<1}r^E\mathcal{U}=\widetilde{\mathcal{U}_E}
\end{equation*}
with inverse $\rho_{E,\varphi}^{-1}(x)=(P(x),\varphi((P(x))^{-E}x))$. For $y=(r,u)\in V$, observe that
\begin{eqnarray*}
D\rho_{E,\varphi}(y)&=&\left.\left(\frac{d}{dr}(r^E\varphi^{-1}(u)) \right\vert D_u\left[r^E\varphi^{-1}(u)\right]\right)\\
&=&\left.\left(r^{E-I}E\varphi^{-1}(u)\right\vert r^E D_u\varphi^{-1}(u)\right)\\
&=&r^E\left.\left(\frac{1}{r}E\varphi^{-1}(u)\right\vert D_u\varphi^{-1}(u)\right).
\end{eqnarray*}
Using properties of the determinant, we have
\begin{eqnarray*}
    \det(D\rho_E(y))&=&\det(r^E)\det\left.\left(\frac{1}{r}E\varphi^{-1}(u)\right\vert D_u\varphi^{-1}(u)\right)\\
    &=&r^{\tr E}r^{-1}\det\left.\left(E\varphi^{-1}(u)\right\vert D_u\varphi^{-1}(u)\right)\\
    &=&r^{\mu_P-1}h_{E,\varphi}(u)
\end{eqnarray*}
for all $y=(r,u)\in V$ thus proving \eqref{eq:JacobianRelation1}. It is clear that $h_{E,\varphi}$ is smooth and, by virtue of the fact that $\rho_{E,\varphi}:(0,1)\times U\to \widetilde{\mathcal{U}_E}$ is a diffeomorphism, \eqref{eq:JacobianRelation1} guarantees that $h_{E,\varphi}$ is everywhere non-vanishing. Finally, 
in view of Proposition \ref{prop:FormRiemannRelation} (and \eqref{eq:sigmaForm}), we find that
\begin{eqnarray*}
    h_{E,\varphi}(u)&=&
    \det\left.\left(E\varphi^{-1}(u)\right\vert D_u\varphi^{-1}(u)\right)\\ \nonumber
    &=&
    (d\Vol_{\mathbb{R}^d})_{\varphi^{-1}(u)}(\mathcal{E}_E,\iota_*(\partial_{u^1}),\iota_*(\partial_{u^2}),\dots,\iota_*(\partial_{u^{d'}}))\\ \nonumber
    &=&\left(d\sigma_P\right)_{\varphi^{-1}(u)}(\partial_{u^1},\partial_{u^2},\dots,\partial_{u^{d'}})\\
\end{eqnarray*}
for $u\in U$ and so \eqref{eq:JacobianRelation2} is satisfied.
\end{proof}

\begin{lemma}\label{lem:LocalIntegralFormula}
Let $g\in L^1(S,\Sigma_P,\sigma_P)$ be supported on the domain of some chart on $S$. Then, for any $(\mathcal{U},\varphi)\in\mathcal{A}^+(S)$ such that $\supp(g)\subseteq\mathcal{U}$, the pushforward $(\varphi^{-1})^*(g)=g\circ\varphi^{-1}$ is Lebesgue measurable on $U=\varphi(\mathcal{U})\subseteq\mathbb{R}^{d'}$ and 
\begin{equation*}
\int_S g(\eta)\sigma_P(d\eta)=\int_{U}(\psi^{-1})^*(g\, d\sigma_P).
\end{equation*}
\end{lemma}
\begin{proof}
Let $(\mathcal{U},\varphi)$ be a chart on $S$ for which $\supp(g)\subseteq \mathcal{U}$ and assume the notation of Lemma \ref{lem:JacobianRelation}. Given $E\in\Exp(P)$, observe that
\begin{equation*}
    f(x)=\mu_P\, \chi_{(0,1)}(P(x))g(P(x)^{-E}x),
\end{equation*}
defined for $x\in\mathbb{R}^d\setminus\{0\}$, is supported on $\widetilde{\mathcal{U}_E}=\rho_{E,\varphi}(V)$. An appeal to Corollary \ref{cor:IntegrateOnS} guarantees that $f$ is absolutely integrable on $\rho_{E,\varphi}(V)$ and
\begin{equation}\label{eq:LocalIntegralFormula1}
\int_S g(\eta)\,\sigma_P(d\eta)=\int_{\rho_{E,\varphi}(V)}f(x)\,dx.
\end{equation}
Given that $\rho_{E,\varphi}$ is a diffeomorphism, an appeal to Theorem 15.11 of \cite{apostol_mathematical_1974} guarantees that
\begin{equation}\label{eq:LocalIntegralFormula2}
\int_{\rho_{E,\varphi}(V)}f(x)\,dx=\int_V f(\rho_{E,\varphi}(y))|\det(D\rho_{E,\varphi}(y))|\,dy;
\end{equation}
in particular, $V\ni y\mapsto f(\rho_{E,\varphi}(y))|\det(D\rho_{E,\varphi}(y))$ is Lebesgue measurable on $\mathbb{R}^{d}$. Let us now view the Lebesgue measure $dy$ on $\mathbb{R}^d$ as the completion of the product measure $dr\times du$ on the product space $\mathbb{R}\times\mathbb{R}^{d'}$. By virtue of Lemma \ref{lem:JacobianRelation}, we have
\begin{eqnarray}\label{eq:LocalIntegralFormula3}\nonumber
    f(\rho_{E,\varphi}(y))|\det(D\rho_{E,\varphi}(y))|
    &=& f(r^E\varphi^{-1}(u))\,r^{\mu_{P}-1}|h_{E,\varphi}(u)|\\ \nonumber
    &=& \mu_P\,\chi_{(0,1)}(r)g(\varphi^{-1}(u))\, r^{\mu_P-1}|h_{E,\varphi}(u)|\\ 
    &=& \mu_P\,r^{\mu_P-1}g(\varphi^{-1}(u))|h_{E,\varphi}(u)|
\end{eqnarray}
for $y=(r,u)\in V$. By virtue of Fubini's theorem, $dr$-almost every $r\in (0,1)$, the $r$-section 
\begin{equation*}
    U\ni u\mapsto \mu_P r^{\mu_P-1}g(\varphi^{-1}(u))|h_{E,\varphi}(u)|
\end{equation*}
is Lebesgue measurable and,  upon recalling that $h_{E,\varphi}$ is smooth and everywhere nonzero, we conclude that $(\varphi^{-1})^*(g)=g\circ\varphi^{-1}$ is Lebesgue measurable on $U$. In view of \eqref{eq:LocalIntegralFormula3}, Fubini's theorem also guarantees that
\begin{eqnarray}\label{eq:LocalIntegralFormula4}\nonumber
    \int_V f(\rho_{E,\varphi}(y))|\det(D\rho_{E,\varphi}(y))|\,dy
    &=&\int_{(0,1)}\int_U \mu_P\, r^{\mu_P-1}g(\varphi^{-1}(u))|h_{E,\varphi}(u)|\,du\,dr\\\nonumber
    &=&\int_0^1\mu_P\, r^{\mu_P-1}\left(\int_U g(\varphi^{-1}(u))|h_{E,\varphi}(u)|\,du\right)\,dr\\
    &=&\int_U g(\varphi^{-1}(u))|h_{E,\varphi}(u)|\,du
\end{eqnarray}
By combining \eqref{eq:LocalIntegralFormula1}, \eqref{eq:LocalIntegralFormula2} and \eqref{eq:LocalIntegralFormula4}, we have shown that $(\varphi^{-1})^*g=g\circ\varphi^{-1}$ is Lebesgue measurable on $U=\varphi(\mathcal{U})$ and
\begin{equation*}
    \int_S g(\eta)\sigma_P(d\eta)=\int_U g(\varphi^{-1}(u))|h_{E,\varphi}(u)|\,du.
\end{equation*}
Finally, if $(\mathcal{U},\varphi)\in\mathcal{A}_+(S)$, an appeal to Proposition \ref{prop:FormRiemannRelation} and \eqref{eq:JacobianRelation2} of Lemma \ref{lem:JacobianRelation} guarantees that
\begin{equation*}
    |h_{E,\varphi}(u)|=h_{E,\varphi}(u)=(d\sigma_P)_{\varphi^{-1}(u)}(\partial_{u^1},\partial_{u^2},\dots,\partial_{u^{d'}})>0
\end{equation*}
for all $u=(u^1,u^2,\dots,u^{d'})\in U$ and thus
\begin{eqnarray*}
        \int_S g(\eta)\sigma_P(d\eta)&=&\int_U g(\varphi^{-1}(u))h_{E,\varphi}(u)\,du\\
        &=&\int_U g(\varphi^{-1}(u))\,h_{E,\varphi}(u)\,du^1\wedge du^2\wedge\cdots \wedge du^{d'}\\
        &=&\int_U (\varphi^{-1})^*(g\cdot d\sigma_P),
\end{eqnarray*}
as desired. 
\end{proof}

\begin{remark}\label{rmk:LocalIntegralFormula}
In studying the proof of Lemma \ref{lem:LocalIntegralFormula}, we deduce the (slightly) more general statement: Given any chart $(\mathcal{U},\varphi)\in\mathcal{A}(S)$ and $g\in L^1(S,\Sigma_P,\sigma_P)$ for which $\supp(g)\in\mathcal{U}$, $(\varphi^{-1})^*g$ is Lebesgue measurable on $U=\varphi(\mathcal{U})$ and 
\begin{equation*}
    \int_S g(\eta)\sigma_P(d\eta)=\begin{cases}
    \displaystyle\int_U(\varphi^{-1})^*(g\,d\sigma_P) &\mbox{ if }(\mathcal{U},\varphi)\in\mathcal{A}_+(S)\\
    &\\
    -\displaystyle\int_U(\varphi^{-1})^*(g\,d\sigma_P)&\mbox{ if }(\mathcal{U},\varphi)\in\mathcal{A}_-(S)\coloneqq\mathcal{A}(S)\setminus\mathcal{A}_+(S).
    \end{cases}
\end{equation*}
\end{remark}
\begin{proof}[Proof of Theorem \ref{thm:RiemannLebesgue}]
In view of Proposition \ref{prop:InnerProdIsOne} and the discussion following its proof, it remains to prove the assertions in the last two sentences in the statement of the theorem. Given any $F\in \Sigma_P$ and chart $(\mathcal{U},\varphi)\subseteq \mathcal{A}(S)$, we have $\chi_{\varphi(F\cap \mathcal{U})}=(\varphi^{-1})^*\left(\chi_{F\cap\mathcal{U}}\right)$ is Lebesgue measurable on $U=\varphi(\mathcal{U})$ and so it follows (Exercise 4.6.2 of \cite{naber_topology_2011}) that $F\in\mathcal{L}(S)$. Consequently, $\mathcal{B}(S)\subseteq\Sigma_P\subseteq\mathcal{L}(S)$.

Let $g\in L^1(S,\Sigma_P,\sigma_P)$, which is necessarily Lebesgue measurable on $S$ in view of the results of the previous paragraph. Now, let  $\{(\mathcal{U}_j,\varphi_j)\}\subseteq\mathcal{A}_+(S)$ be a countable atlas on $S$ and let $\{\kappa_j\}$ be a smooth partition of unity subordinate to the cover $\{\mathcal{U}_j\}$. For each $j\in\mathbb{N}$, observe that $\kappa_j g\in L^1(S,\Sigma_P,\sigma_P)$ and has $\supp(\kappa_j g)\subseteq \mathcal{U}_j$. By virtue of Lemma \ref{lem:LocalIntegralFormula}, we have $(\varphi_j^{-1})^*(\kappa_j g\,d\sigma_P)$ is integrable on $U_j=\varphi(\mathcal{U}_j)$ and
\begin{equation}\label{eq:LebesgueRiemann1}
\int_S \kappa_j(\eta)g(\eta)\sigma_P(d\eta)=\int_{U_j}(\varphi_j^{-1})^*(\kappa_j g\,d\sigma_P)
\end{equation}
for each $j\in\mathbb{N}$. With the help of Proposition \ref{prop:FormRiemannRelation}, it is easy to see that $g\,d\sigma_P$ and $\kappa_j g\,d\sigma_P$ (for $j\in\mathbb{N}$) are Lebesgue measurable $d'$-forms on $S$. In view of (4.4.6) of \cite{naber_topology_2011}, \eqref{eq:LebesgueRiemann1} ensures that, for each $j\in\mathbb{N}$, $\kappa_j g\,d\sigma_P$ is integrable (in the sense of forms) on $S$ and
\begin{equation*}
\int_S \kappa_j(\eta)g(\eta)\sigma_P(d\eta)=\int_S \kappa_j g\,d\sigma_P.
\end{equation*}
By the monotone convergence theorem, it follows that
\begin{equation*}
\sum_{j=1}^\infty\left| \int_S \kappa_j g\,d\sigma_P\right|=\sum_{j=1}^\infty \left|\int_S \kappa_j(\eta)g(\eta)\sigma(d\eta)\right|\leq \sum_{j=1}^\infty\int_S \kappa_j(\eta)|g(\eta)|\sigma_P(d\eta)=\|g\|_{L^1(S,\Sigma_P,\sigma_P)}<\infty.
\end{equation*}
Therefore, in view of the construction on p. 242 of \cite{naber_topology_2011}, we conclude that the $d'$-form $g\,d\sigma_P$ is integrable and
\begin{equation*}
\int_S g(\eta)\sigma_P(d\eta)=\sum_{j=1}^\infty \int_S \kappa_j(\eta)g(\eta)\sigma_P(d\eta)=\sum_{j=1}^\infty \int_{S} \kappa_j g\,d\sigma_P=\int_S g\,d\sigma_P
\end{equation*}
by virtue of the dominated convergence theorem; this is \eqref{eq:FormsAndMeasures}.

Finally, given that $\abs{\grad P}$ is continuous and non-vanishing on the compact set $S$,
\begin{equation*}
    C_1 := \inf_S\frac{1}{|\grad P|}\quad\mbox{and}\quad C_2 := \sup_{S}\frac{1}{\abs{\grad P }}
\end{equation*}
are both positive real numbers. For each $F\in \Sigma_P\subseteq\mathcal{L}(S)$, \eqref{eq:FormsAndMeasures} guarantees that
\begin{equation*}
    \sigma_P(F)=\int_S\chi_F(\eta)\sigma_P(d\eta)=\int_S \chi_F\,d\sigma_P
\end{equation*}
By virtue of Proposition \ref{prop:FormRiemannRelation}, it follows that
\begin{equation*}
C_1\Vol_S(F)=C_1\int_S \chi_F\,d\Vol_S\leq \int_S \frac{\chi_F}{|\grad P|}d\Vol_S=\int_S \chi_F\,d\sigma_P=\sigma_P(F)
\end{equation*}
and
\begin{equation*}
    \sigma_P(F)=\int_S\chi_F d\sigma_P=\int_S\frac{\chi_F}{|\grad P|}\,d\Vol_S\leq C_2\int_S\chi_F\,d\Vol_S=C_2\Vol_S(F)
\end{equation*}
where we have used the definition of the Riemannian volume measure on $S$, c.f., \cite{amann_analysis_2009}. In short, there are positive constants $C_1$ and $C_2$ for which
\begin{equation}\label{eq:FormsAndMeasures2}
    C_1\Vol_S(F)\leq\sigma_P(F)\leq C_2\Vol_S(F)
\end{equation}
for all $F\in\Sigma_P$. In particular, \eqref{eq:FormsAndMeasures} holds for all $F\in\mathcal{B}(S)$ and so it follows that the completions of the $\sigma$-algebra $\mathcal{B}(S)$ with respect to $\sigma_P$ and $\Vol_S$ coincide. We know, however, that $\Vol_S$, which is defined on $\mathcal{L}(S)$, is a Radon measure (Proposition 1.5 \cite[Chapter XII]{amann_analysis_2009}) and, by virtue of Proposition \ref{prop:Regular}, it follows that $\Sigma_P=\mathcal{L}(S)$ and \eqref{eq:FormsAndMeasures2} holds for all $F$ in this common $\sigma$-algebra. Thus $\sigma_P$ and $\Vol_S$ are mutually absolutely continuous and the theorem is proved.
\end{proof}

\noindent We immediately obtain the following corollary which allows us to compute the Lebesgue integral with respect to $\sigma_P$ in coordinates.
\begin{corollary}\label{cor:IntegralFormula}
Let $g\in L^1(S,\mathcal{L}(S),\sigma_P)$. Then, given any  countable (or finite) atlas $\{(\mathcal{U}_j,\varphi_j)\}\subseteq\mathcal{A}_+(S)$,  smooth partition of unity $\{\kappa_j\}$ subordinate to $\{\mathcal{U}_j\}$, and $E\in\Exp(P)$,
\begin{equation*}
\int_S g(\eta)\sigma_P(d\eta)=\sum_{j}\int_S \kappa_jg\,d\sigma_P=\sum_j\int_{U_j}\kappa_j(\varphi^{-1}(u))g(\varphi^{-1}(u))h_{E,\varphi_j}(u)\,du
\end{equation*}
where, for each $j$, $U_j=\varphi_j(\mathcal{U}_j)\subseteq\mathbb{R}^{d'}$ and
\begin{equation*}
    h_{E,\varphi_j}(u)=\det(E\varphi^{-1}(u)\vert D_u\varphi^{-1})
\end{equation*}
for $u=(u^1,u^2,\dots,u^{d'})\in U_j$. 
\end{corollary}

\vspace{.4cm}
\noindent{\small\bf Acknowledgement:} We thank Laurent Saloff-Coste for helpful discussions.

\appendix
\section{Appendix} %\label{sec:Appendix}

This appendix amasses some facts about continuous one-parameter subgroups of $\Gl(\mathbb{R}^d)$ and, in particular, those which are contracting.

\begin{proposition}[see Section 8 of \cite{randles_convolution_2017}]\label{prop:ContinuousGroupProperties}
Let $E,G\in\End(\mathbb{R}^d)$ and $A\in\Gl(\mathbb{R}^d)$. Also, let $E^*$ denote the adjoint of $E$. Then, for all $t,s>0$, the following statements hold:

\vspace{.3cm}
\begin{tabular}{lllll}
$\bullet$ $1^E=I$ &  $\bullet$ $t^{E^*}=(t^E)^*$ & $\bullet$ $(t^E)^{-1}=t^{-E}$ &   $\bullet$ If $EG=GE$, then $t^Et^G=t^{E+G}$\\
\vspace{.1cm}\\
$\bullet$ $(st)^E=s^Et^E$ & $\bullet$ $At^EA^{-1}=t^{AEA^{-1}}$&  $\bullet$ $\det\left(t^E\right)=t^{\tr E}$\\
\end{tabular}
\end{proposition}

\noindent We recall from the introduction that a continuous one-parameter group $\{T_t\}$ is said to be contracting if $\lim_{t\to 0}\|T_t\|=0$. The notion of a contracting group (or semigroup) is closely related to notion of asymptotic stability of solutions to linear systems \cite{braun_differential_1993}. By virtue of the Banach-Steinhaus theorem, we have the following useful characterization of contracting groups.
\begin{proposition}\label{prop:ContractingCharacterization}
Let $\{T_r\}$ be a continuous one-parameter group. Then $\{T_r\}$ is contracting if and only if
\begin{equation}\label{eq:ContractingSufficient}
\lim_{t\to 0}|T_rx|=0
\end{equation}
for all $x\in\mathbb{R}^d$.
\end{proposition}

\begin{proof}It is clear that \eqref{eq:ContractingSufficient} is a necessary condition for $\{T_r\}$ to be contracting. We must therefore prove \eqref{eq:ContractingSufficient} also sufficient. To this end, we assume that the continuous one-parameter group $\{T_r\}$ satisfies \eqref{eq:ContractingSufficient}. By virtue of the continuity of $\{T_r\}$ and \eqref{eq:ContractingSufficient}, we have
\begin{equation*}
\sup_{0<r\leq 1}|T_r x|<\infty
 \end{equation*}
for each $x\in\mathbb{R}^d$. From the Banach-Steinhaus theorem, it follows that $\|T_r\|\leq C$ for all $0<r\leq 1$. Now, suppose that $\{T_r\}$ is not contracting. In this case, one can find a sequence $\{\eta_n\}\subseteq\mathbb{S}$ and a sequence $r_n\rightarrow 0$ for which $\lim_n|T_{r_n}\eta_n|>0$. But because the unit sphere is compact, $\{\eta_n\}$ has a convergence subsequence $\eta_{n_k}\rightarrow \eta$ with $|\eta|=1$ . Observe that, for all $n$,
 \begin{equation*}
 |T_{r_n}(\eta-\eta_n)|\leq C|\eta-\eta_n|
 \end{equation*}
 and so it follows that
 \begin{equation*}
 \lim_{k\rightarrow\infty}|T_{t_{r_k}}\eta|=\lim_{k\rightarrow\infty}|T_{r_{n_k}}\eta_{n_k}|>0,
 \end{equation*}
 a contradiction.
 \end{proof}

\begin{lemma}\label{lem:OperatorBoundsforContractingGroup}
Let $\{T_r\}\subseteq\GldR$ be a continuous one-parameter group and let $E\in\End(\mathbb{R}^d)$ be its generator, i.e., $T_r=r^E$ for all $r>0$. If $\{T_r\}$ is contracting, then $E\in\Gl(\mathbb{R}^d)$ and there is a positive constant $C$ for which
\begin{equation*}
\|T_r\|\leq C+r^{\|E\|}
\end{equation*}
for all $r>0$.
\end{lemma}
\begin{proof}
If for some non-zero vector $\eta$, $E\eta=0$, then $r^E\eta=\eta$ for all $r>0$ and this would contradict our assumption that $\{T_r\}$ is contracting. Hence $E\in\Gl(\mathbb{R}^d)$ and, in particular, $\|E\|>0$. From the representation $T_r=r^E$, it follows immediately that $\|T_r\|\leq r^{\|E\|}$ for all $r\geq 1$ and so it remains to estimate $\|T_r\|$ for $r<1$. Given that $\{T_r\}$ is continuous and contracting, the map $r\mapsto \|T_r\|$ is continuous and approaches $0$ as $r\rightarrow 0$ and so it is necessarily bounded for $0<r\leq 1$.
\end{proof}

\begin{proposition}\label{prop:ContractingLimits}
Let $\{T_r\}_{r>0}\subseteq\Gl(\mathbb{R}^d)$ be a continuous one-parameter contracting group.  Then, for all non-zero $x\in\mathbb{R}^d$,
\begin{equation*}
\lim_{r\rightarrow 0}|T_r x|=0\hspace{.5cm}\mbox{ and }\hspace{.5cm}\lim_{r\rightarrow\infty}|T_r x|=\infty.
\end{equation*}
\end{proposition}
\begin{proof}
The validity of the first limit is clear. Upon noting that $|x|=|T_{1/r}T_rx|\leq \|T_{1/r}\||T_r x|$ for all $r>0$, the second limit follows at once.
\end{proof}
\begin{proposition}\label{prop:ScaleFromSphere}
Let $\{T_r\}_{r>0}$ be a continuous one-parameter contracting group. There holds the following:
\begin{enumerate}[label=(\alph*), ref=(\alph*)]
\item\label{item:ScaleFromSphere_1} For each non-zero $x\in \mathbb{R}^d$, there exists $r>0$ and $\eta\in \mathbb{S}$ for which $T_r\eta=x$. Equivalently,
\begin{equation*}
\mathbb{R}^d\setminus\{0\}=\{T_r\eta:r>0\mbox{ and }\eta\in \mathbb{S}\}.
\end{equation*}
\item\label{item:ScaleFromSphere_2} For each sequence $\{x_n\}\subseteq\mathbb{R}^d$ such that $\lim_n|x_n|=\infty$, $x_n=T_{r_n}\eta_n$ for each $n$, where $\{\eta_n\}\subseteq \mathbb{S}$ and $r_n\rightarrow\infty$ as $n\rightarrow\infty$.
\item\label{item:ScaleFromSphere_3} For each sequence $\{x_n\}\subseteq\mathbb{R}^d$ such that $\lim_n|x_n|=0$, $x_n=T_{r_n}\eta_n$ for each $n$, where $\{\eta_n\}\subseteq \mathbb{S}$ and $r_n\rightarrow 0$ as $n\rightarrow\infty$.
\end{enumerate}
\end{proposition}
\begin{proof}
In view of Proposition \ref{prop:ContractingLimits}, the assertion \ref{item:ScaleFromSphere_1} is a straightforward application of the intermediate value theorem. For \ref{item:ScaleFromSphere_2}, suppose that $\{x_n\}\subseteq\mathbb{R}^d$ is such that $|x_n|\rightarrow \infty$ as $n\rightarrow\infty$. In view of \ref{item:ScaleFromSphere_1}, take $\{\eta_n\}\subseteq S$ and $\{r_n\}\subseteq (0,\infty)$ for which $x_n=T_{r_n}\eta_n$ for each $n$. In view of Lemma \ref{lem:OperatorBoundsforContractingGroup},
\begin{equation*}
\infty=\liminf_n |x_n|\leq\liminf_n \left(C+r_n^M\right)|\eta_n|\leq C+\liminf_n r_n^M,
\end{equation*}
where $C,M>0$ and therefore $r_n\rightarrow\infty$. If instead $\lim_n x_n=0$,
\begin{equation*}
\infty=\lim_{n\rightarrow\infty}\frac{|\eta_n|}{|x_n|}=\lim_{n\rightarrow\infty}\frac{|T_{1/r_n}x_n|}{|x_n|}\leq\limsup_n\|T_{1/r_n}\|\leq\limsup_n(C+(1/r_n)^M)
\end{equation*}
from which we see that $r_n\rightarrow 0$, thus proving \ref{item:ScaleFromSphere_3}.
\end{proof}

\begin{proposition}\label{prop:ContractingCapturesCompact}
Let $\{T_r\}$ be a continuous contracting one-parameter group. Then for any open neighborhood $\mathcal{O}\subseteq\mathbb{R}^d$ of the origin and any compact set $K\subseteq\mathbb{R}^d$, $K\subseteq T_r(\mathcal{O})$ for sufficiently large $r$.
\end{proposition}
\begin{proof}
Assume, to reach a contradiction, that there are sequences $\{x_n\}\subseteq K$ and $r_n\rightarrow\infty$ for which $x_n\notin T_{r_n}(\mathcal{O})$ for all $n$. Because $K$ is compact, $\{x_n\}$ has a subsequential limit and so by relabeling, let us take sequences $\{\zeta_k\}\subseteq K$ and $\{t_k\}\subseteq (0,\infty)$ for which $\zeta_k\rightarrow \zeta$, $t_k\rightarrow\infty$ and $\zeta_k\notin T_{t_k}(\mathcal{O})$ for all $k$. Setting $s_k=1/t_k$ and using the fact that $\{T_r\}$ is a one-parameter group, we have $T_{s_k}\zeta_k\notin\mathcal{O}$ for all $k$ and so $\liminf_{k}|T_{s_k}\zeta_k|>0$, where $s_k\rightarrow 0$. This is however impossible because $\{T_r\}$ is contracting and so
\begin{equation*}
\lim_{k\rightarrow\infty}|T_{s_k}\zeta_k|\leq\lim_{k\rightarrow \infty}|T_{s_k}(\zeta_k-\zeta)|+\lim_{k\rightarrow\infty}|T_{s_k}\zeta|\leq C\lim_{k\rightarrow\infty}|\zeta_k-\zeta|+0=0
\end{equation*}
in view of Lemma \ref{lem:OperatorBoundsforContractingGroup}.
\end{proof}

\begin{proposition}\label{prop:ContractingTrace}
Let $\{T_r\}\subseteq \Gl(\mathbb{R}^d)$ be a continuous one-parameter group with generator $E$. If $\{T_r\}$ is  contracting, then $\tr E>0$. 
\end{proposition}
\begin{proof}
The supposition that $\{T_r\}$ is a contracting group implies that $r^E\to \mathbf{0}$ as $r\to 0$ in the operator-norm topology on $\End(\mathbb{R}^d)$; here $\mathbf{0}$ is zero transformation. Because the determinant is a continuous function from $\End(\mathbb{R}^d)$, equipped with operator-norm topology, into $\mathbb{R}$, we have
\begin{equation*}
0=\det(\mathbf{0})=\det\left(\lim_{r\to 0}r^E\right)=\lim_{r\to 0}\det\left(r^E\right)=\lim_{r\to 0}r^{\tr E}
\end{equation*}
in view of the preceding proposition. Therefore, $\tr E>0$.
\end{proof}

\bibliographystyle{abbrv}
\bibliography{GPIF}

\end{document}